\documentclass[a4paper,twoside,11pt]{article}

\usepackage{BeaCed_article}

\begin{document}

\title{The critical $Z$-invariant Ising model via dimers:\\
locality property}
\author{C\'edric Boutillier
\thanks{
{\small Laboratoire de Probabilit\'es et Mod\`eles Al\'eatoires, Universit\'e Paris VI Pierre et Marie Curie, Case courrier 188, 4 place Jussieu, F-75252 Paris CEDEX 05.} \small\texttt{cedric.boutillier@upmc.fr}. Supported in part by the Swiss National Foundation Grant 200020-120218/1.
}
, B\'eatrice de Tili\`ere
\thanks{
{\small Institut de Math\'ematiques, Universit\'e de Neuch\^atel, Rue Emile-Argand 11, CH-2007 Neuch\^atel.}
{\small\texttt{beatrice.detiliere@unine.ch}}.
{\small Supported in part by the Swiss National Foundations grants 47102009 and 200020-120218/1.}
} }
 
\date{}
\maketitle
\vspace{-1cm}
\begin{abstract}
We study a large class of critical two-dimensional Ising models, namely {\em
critical $Z$-invariant Ising models}. Fisher \cite{Fisher}
introduced a correspondence between the Ising model and the dimer
model on a decorated graph, thus setting dimer techniques as a
powerful tool for understanding the Ising model. In this paper, we give a full description of the dimer model corresponding to the critical $Z$-invariant Ising model, consisting of explicit expressions which only depend on the {\em local geometry} of the underlying isoradial graph. Our main result is an explicit local formula for the inverse Kasteleyn matrix, in the spirit of \cite{Kenyon3}, as a contour integral of the discrete exponential function of \cite{Mercat2,Kenyon3} multiplied by a local function. Using results of \cite{isoising1} and techniques of \cite{Bea1,Kenyon3}, this yields an explicit local formula for a natural Gibbs measure, and a local formula for the free energy. As a corollary, we recover Baxter's formula for the free energy of the critical $Z$-invariant Ising model \cite{Baxter}, and thus a new proof of it. The latter is equal, up to a constant, to the logarithm of the normalized determinant of the Laplacian obtained in \cite{Kenyon3}.
\end{abstract}

\section{Introduction}

In \cite{Fisher}, Fisher introduced a correspondence between the two-dimensional Ising model defined on a graph $G$, and the dimer model defined on a decorated version of this graph. Since then, dimer techniques have been a powerful tool for solving pertinent questions about the Ising model, see for example the paper of Kasteleyn \cite{Kasteleyn}, and the book of Mc Coy and Wu \cite{McCoyWu}. In this paper, we follow this approach to the Ising model.

We consider a large class of critical Ising models, known as {\em critical $Z$-invariant Ising models}, introduced in \cite{BaxterZ}. More precisely, we consider Ising models defined on graphs which have the property of having an {\em isoradial embedding}. We suppose that the Ising coupling constants naturally depend on the geometry of the embedded graph, and are such that the model is invariant under {\em star-triangle} transformations of the underlying graph, {\em i.e.} such that the Ising model is {\em $Z$-invariant}. We suppose moreover that the coupling constants are {\em critical} by imposing a {\em generalized self-duality property}. The standard Ising model on the square, triangular and honeycomb lattice at the critical temperature are examples of critical $Z$-invariant Ising models. In the mathematics literature, this model has been studied by Mercat \cite{Mercat}, where the author proves equivalence between existence of Dirac spinors and criticality; by Chelkak and Smirnov \cite{ChelkakSmirnov}, where the authors prove conformal invariance in the scaling limit; in \cite{isoising1}, where we give 
a complete description of the equivalent dimer model, in the case where the underlying graph is periodic. The critical $Z$-invariant Ising model has also been widely studied in the physics literature, see for example \cite{BaxterZ,Perk1,Perk2,Perk3,Martinez1,Martinez2,CostaSantos}.

Let $G=(V(G),E(G))$ be an infinite, locally finite, isoradial graph with critical coupling constants on the edges. Then by Fisher, the Ising model on $G$ is in correspondence with the dimer model on a decorated graph $\GD$, with a well chosen positive weight function $\nu$ on the edges. We refer to this model as the {\em critical dimer model on the Fisher graph $\GD$ of $G$}. It is defined as follows. A {\em dimer configuration} of $\GD$ is a subset of edges of $\GD$ such that every vertex is incident to exactly one edge of $M$. Let $\M(\GD)$ be the set of dimer configurations of $\GD$. Dimer configurations of $\GD$ are chosen according to a probability measure, known as {\em Gibbs measure}, satisfying the following properties. If one fixes a perfect matching in an annular region of $\GD$, then perfect matchings inside and outside of this annulus are independent. Moreover, the probability of occurrence of an interior matching is proportional to the product of its edges weights given by the weight function $\nu$.

The key objects used to obtain explicit expressions for relevant quantities of the dimer model are the {\em Kasteleyn matrix}, denoted by $K$, and its inverse. A Kasteleyn matrix is an oriented adjacency matrix of the graph $\GD$, whose coefficients are weighted by the function $\nu$, introduced by Kasteleyn \cite{Kast61}. Our main result is Theorem \ref{thm:inverse_intro}, consisting of an explicit {\em local} expression for an inverse $K^{-1}$ of the Kasteleyn matrix $K$. It can loosely be stated as follows, refer to Section \ref{subsc42} for detailed definitions and to Theorem \ref{inverse} for a precise statement.

\begin{thm}\label{thm:inverse_intro}
Let $x,y$ be two vertices of $\GD$. Then the infinite matrix $K^{-1}$, whose coefficient $K^{-1}_{x,y}$ is given below, is an inverse Kasteleyn matrix.
\begin{equation*}
K^{-1}_{x,y}=
\frac{1}{(2\pi)^2}
\oint_{\C_{x,y}}f_{x}(\lambda)f_{y}(-\lambda)
\expo_{\xb,\yb}(\lambda)\log\lambda \ud\lambda+C_{x,y},
\end{equation*}
where $f_x$ is a complex-valued function depending on the vertex $x$ only; $\expo_{\xb,\yb}$ is the discrete exponential function introduced in \cite{Mercat2}, see also \cite{Kenyon3}; $C_{x,y}$ is a constant equal to $\pm\frac{1}{4}$ when $x$ and $y$ are close, and $0$ else; $\C_{x,y}$ is a simple closed curve oriented counterclockwise containing all
poles of the integrand, and avoiding a half-line $d_{x,y}$ starting from zero.
\end{thm}
Before stating implications of Theorem \ref{thm:inverse_intro} on the critical dimer model on the Fisher graph $\GD$, let us make a few comments.
\begin{enumerate}
\item Theorem \ref{thm:inverse_intro} is in the spirit of the work of Kenyon \cite{Kenyon3}, where the author obtains explicit local expressions for the critical Green's function defined on isoradial graphs, and for the inverse Kasteleyn matrix of the critical dimer model defined on bipartite isoradial graphs. Surprisingly both, the expression for the Green's function of \cite{Kenyon3} and our expression for the inverse Kasteleyn matrix involve the discrete exponential function. This relation is pushed even further in Corollary \ref{cor:cst}, see the comment after the statement of Theorem \ref{thm:baxter_intro}. Note that the expression for the inverse Kasteleyn matrix obtained in \cite{Kenyon3} does not hold in our setting, since the dimer model we consider is defined on the Fisher graph $\GD$ of $G$, which is {\em not} bipartite and not {\em isoradial}.
\item Theorem \ref{thm:inverse_intro} should also be compared with the explicit expression for the inverse Kasteleyn matrix obtained in \cite{isoising1}, in the case where the graph $\GD$ is periodic. Then, by the uniqueness statement of Proposition $5$ of \cite{isoising1}, the two expressions coincide. The proofs are nevertheless totally different in spirit, and
we have not yet been able to understand the identity by an explicit computation, except in the case where $G=\ZZ^2$.
\item The most interesting features of Theorem \ref{thm:inverse_intro} are the following. First, there is no periodicity assumption on the graph $\GD$. Secondly, the expression for $K^{-1}_{x,y}$ is {\em local}, meaning that it only depends on the local geometry of the underlying isoradial graph $G$. More precisely, it only depends on an edge-path of $G$ between vertices $\hat{\xb}$ and $\hat{\yb}$ of $G$, naturally constructed from $x$ and $y$. This implies that changing the isoradial graph $G$ away from this path does not change the expression for $K^{-1}_{x,y}$. Thirdly, explicit computations of $K^{-1}_{x,y}$ become tractable, whereas, even in the periodic case, they remain very difficult with the explicit expression for $K^{-1}_{x,y}$ given in \cite{isoising1}.
\item The structure of the proof of Theorem \ref{thm:inverse_intro} is taken from \cite{Kenyon3}. The idea is to find complex-valued functions that are in the kernel of the Kasteleyn matrix $K$, then to define $K^{-1}$ as a contour integral of these functions, and to define the contours of integration in such a way that $KK^{-1}=\Id$. The great difficulty lies in actually finding the functions that are in the kernel of $K$, since there is no general method to construct them; and in defining the contours of integration. Indeed, the Fisher graph $\GD$ is obtained from an isoradial graph $G$, but has a more complicated structure, so that the geometric argument of \cite{Kenyon3} does not work. The proof of Theorem \ref{thm:inverse_intro} being long, it is postponed until the last section of this paper.
\end{enumerate}

Using Theorem \ref{thm:inverse_intro}, an argument similar to \cite{Bea1}, and the results obtained in \cite{isoising1}, yields Theorem \ref{thm:Gibbs_intro} giving an explicit expression for a Gibbs measure on $\M(\GD)$. When the graph $\GD$ is periodic, this measure coincides with the Gibbs measure obtained in \cite{isoising1} as weak limit of Boltzmann measures on a natural toroidal exhaustion of the graph $\GD$. A precise statement is given in Theorem \ref{thm:measure} of Section \ref{sec:measure}.
\begin{thm}\label{thm:Gibbs_intro} 
There is a unique Gibbs measure $\P$ defined on $\M(\GD)$, such that the probability of occurrence of a subset of edges $\{e_1=x_1 y_1,\cdots,x_k y_k\}$ in a dimer configuration of $\GD$, chosen with respect to the Gibbs measure $\P$ is:
\begin{equation*}
\P(e_1,\cdots,e_k)=\left(\prod_{i=1}^k K_{x_i,y_i}\right)\Pf\left(^t (K^{-1})_{\{x_1,y_1,\cdots,x_k,y_k\}}\right), 
\end{equation*}
where $K^{-1}$ is given by Theorem \ref{thm:inverse_intro}, and $(K^{-1})_{\{x_1,y_1,\cdots,x_k,y_k\}}$ is the sub-matrix of $K^{-1}$, whose lines and columns are indexed by vertices $\{x_1,y_1,\cdots,x_k,y_k\}$.
\end{thm}
Theorem \ref{thm:Gibbs_intro} is a result about Gibbs measures with no periodicity assumption on the underlying graph. There are only very few examples of such instances in statistical mechanics. Moreover, the Gibbs measure $\P$ inherits the properties of the inverse Kasteleyn matrix $K^{-1}$ of Theorem \ref{thm:inverse_intro}: it is local, and allows for explicit computations, examples of which are given in Appendix \ref{app:calculs}.

Let us now assume that the graph $G$ is $\ZZ^2$-periodic. Using Theorem \ref{thm:inverse_intro}, and the techniques of \cite{Kenyon3}, we obtain an explicit expression for the free energy of the critical dimer model on the graph $\GD$, see Theorem \ref{thm:free_energy}, depending only on the angles of the fundamental domain. Using Fisher's correspondence between the Ising and dimer models, this yields a new proof of Baxter's formula for the free energy of the critical $Z$-invariant Ising model, denoted by $f_I$, see Section \ref{sec:free_energy} for definitions. 
\begin{thm}[\cite{Baxter}]\label{thm:baxter_intro}
\begin{equation*}
f_I=-|V(G_1)|\frac{\log 2}{2} -\sum_{e\in E(G_1)}\left[\frac{\theta_e}{\pi}\log\tan\theta_e+\frac{1}{\pi}\left(L(\theta_e)+L\left(\frac{\pi}{2}-\theta_e\right)\right)
\right].
\end{equation*}
\end{thm}
Note that the free energy $f_I$ of the critical $Z$-invariant Ising model is, up to a multiplicative constant $-\frac{1}{2}$ and an additive constant, the logarithm of the normalized determinant of the Laplacian obtained by Kenyon \cite{Kenyon3}. See also Corollary \ref{cor:cst}, where we explicitly determine the constant of proportionality relating the characteristic polynomials of the critical Laplacian and of the critical dimer model on the Fisher graph $\GD$, whose existence had been established in \cite{isoising1}.

{\bf Outline of the paper}
\begin{enumerate}
\item[Section $2$:] Definition of the critical $Z$-invariant Ising model.
\item[Section $3$:] Fisher's correspondence between the Ising and dimer models.
\item[Section $4$:] Definition of the Kasteleyn matrix $K$. Statement of Theorem \ref{inverse}, giving an explicit local expression for the coefficient $K^{-1}_{x,y}$ of the inverse Kasteleyn matrix. Asymptotic expansion of $K^{-1}_{x,y}$, as $|x-y|\rightarrow\infty$, using techniques of \cite{Kenyon3}.
\item[Section $5$:] Implications of Theorem \ref{inverse} on the critical dimer model on the Fisher graph $\GD$: Theorem \ref{thm:measure} gives an explicit local expression for a natural Gibbs measure, and Theorem \ref{thm:free_energy} gives an explicit local expression for the free energy. Corollary: Baxter's formula for the free energy of the critical $Z$-invariant Ising model. 
\item[Section $6$:] Proof of Theorem \ref{inverse}:
 \end{enumerate}

\noindent {\em Acknowledgments:} we would like to thank Richard Kenyon for asking the questions solved in this paper.

\section{The critical $Z$-invariant Ising model}

Consider an unoriented finite graph $G=(V(G),E(G))$, together with a collection of positive real numbers $J=(J_e)_{e\in E(G)}$ indexed by the edges of $G$. The {\em Ising model on $G$ with coupling constants $J$} is defined as follows. A \emph{spin configuration} $\sigma$ of $G$ is a function of the vertices of $G$ with values in $\{-1,+1\}$. The probability of occurrence of a spin configuration $\sigma$ is given by the {\em Ising Boltzmann measure}, denoted $P^J$:
\begin{equation*}
P^J(\sigma)=\frac{1}{Z^J}\exp\left(\sum_{e=uv\in E(G)}J_e\sigma_u\sigma_v\right),
\end{equation*}
where 
$
 Z^J=\sum_{\sigma\in\{-1,1\}^{V(G)}}\exp\left(\sum_{e=uv\in E(G)}J_e\sigma_u\sigma_v\right),
$
is the {\em Ising partition function}. 

We consider Ising models defined on a class of embedded graphs which have an additional property called {\em isoradiality}. A graph $G$ is said to be {\em isoradial}~\cite{Kenyon3}, if it has an embedding in the plane such that every face is inscribed in a circle of radius~1, and
all circumcenters of the faces are in the closure of the faces. From now on, when we speak 
of the graph $G$, we mean the graph together with a particular isoradial embedding in the
plane. Examples of isoradial graphs are the square and the honeycomb lattice,
see Figure \ref{fig:isingcrit} (left) for a more general example of isoradial graph.

To such a graph is naturally associated the {\em diamond graph}, denoted by $\GR$, defined as follows. Vertices of $\GR$ consist in the vertices
of $G$, and the circumcenters of the faces of $G$. The circumcenter of each
face is then joined to all vertices which are on the boundary of this face, see
Figure \ref{fig:isingcrit} (center). Since $G$ is isoradial,
all faces of $\GR$ are side-length-$1$ rhombi. Moreover, each edge $e$ of
$G$ is the diagonal of exactly one rhombus of $\GR$; we let $\theta_e$ be the
half-angle of the rhombus at the vertex it has in common with $e$, see Figure \ref{fig:isingcrit} (right). 

\begin{figure}[ht]
\begin{center}
\includegraphics[width=\linewidth]{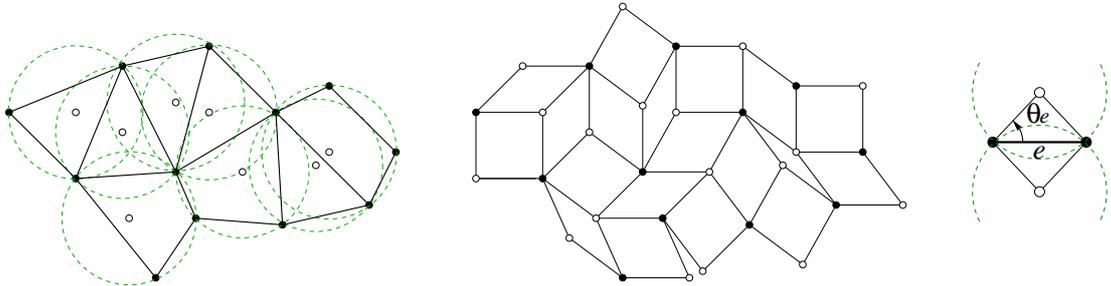}
\caption{Left: example of isoradial graph. Center: corresponding diamond graph. Right: rhombus half-angle associated to an edge $e$ of the graph.}
\label{fig:isingcrit}
\end{center}
\end{figure}

The same construction can be done for infinite and toroidal isoradial graphs, in which case the embedding is on a torus. When the isoradial graph is infinite and non periodic, in order to ensure that the embedding is locally finite, we assume that there exists an $\varepsilon>0$ such that the half-angle of every rhombus of $\GR$ lies between $\varepsilon$ and $\frac{\pi}{2}-\varepsilon$. This implies in particular that vertices of $G$ have bounded degree.

It is then natural to choose the coupling constants $J$ of the Ising model defined on an isoradial graph $G$, to depend on the geometry of the embedded graph: let us assume that $J_e$ is a function of $\theta_e$, the rhombus half-angle assigned to the edge $e$. 

We impose two more conditions on the coupling constants. First, we ask that the Ising model on $G$ with coupling constants $J$ as above is {\em $Z$-invariant}, that is, invariant under {\em star-triangle} transformations of the underlying graph. Next, we impose that the Ising model satisfies a generalized form of {\em self-duality}. These conditions completely determine the coupling constants $J$, known as 
{\em critical coupling constants}: for every edge $e$ of $G$,
\begin{equation}\label{eq:critical_coupling}
  J(\theta_e) = \frac{1}{2} \log \left(\frac{1+\sin \theta_e}{\cos \theta_e}\right).
\end{equation}
The $Z$-invariant Ising model on an isoradial graph with this particular choice of coupling constants is referred to as the \emph{critical $Z$-invariant Ising model}. This model was introduced by Baxter in \cite{BaxterZ}. A more detailed definition is given in \cite{isoising1}.

\section{Fisher's correspondence between the Ising and dimer models}

Fisher \cite{Fisher} exhibits a correspondence between the Ising model on any graph
$G$ drawn on a surface without boundary, and the dimer model on a ``decorated'' version of $G$. Before explaining this correspondence, let us first recall the definition of the dimer model.

\subsection{Dimer model}

Consider a finite graph $\GD=(V(\GD),E(\GD))$, and suppose that edges of $\GD$ are assigned a positive weight function $\nu=(\nu_e)_{e\in E(\GD)}$. The {\em dimer model on $\GD$ with weight function $\nu$} is defined as follows.

A {\em dimer configuration} $M$ of $\GD$, also called {\em perfect matching}, is a subset of edges of $\GD$ such that every vertex is incident to exactly one edge of $M$. Let $\M(\GD)$ be the set of dimer configurations of the graph $\GD$. The probability of occurrence of a dimer configuration $M$ is given by the {\em dimer Boltzmann measure}, denoted $\P^\nu$:
\begin{equation*}
 \P^\nu(M)=\frac{\prod_{e\in M}\nu_e}{\Z^\nu},
\end{equation*}
where $\Z^\nu=\sum_{M\in \M(\GD)}\prod_{e\in M}\nu_e$ is the {\em dimer partition function}.

\subsection{Fisher's correspondence}

Consider an Ising model on a finite graph $G$ embedded on a surface without boundary, with coupling constants $J$. We use the following slight variation of Fisher's correspondence \cite{Fisher}.

The decorated graph, on which the dimer configurations live, is constructed from $G$ as follows. Every vertex of degree $k$ of $G$ is replaced by a {\em
decoration} consisting of $3k$ vertices: a triangle is attached to
every edge incident to this vertex, and these triangles are linked by edges in
a circular way, see Figure \ref{fig:decorated_graph} below. This new graph, denoted by $\GD$, is also embedded on the surface without boundary and has vertices of degree $3$. It is referred to as the {\em Fisher graph} of $G$. 

\begin{figure}[ht]
\begin{center}
\includegraphics[height=2.8cm]{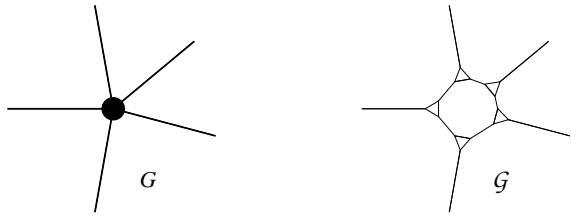}
\caption{Left: a vertex of $G$ with its incoming edges. Right: corresponding decoration in $\GD$.}
\label{fig:decorated_graph}
\end{center}
\end{figure}
Fisher's correspondence uses the high temperature expansion\footnote{It is also possible to use the \emph{low temperature expansion}, if the spins of the the Ising model do not sit on the vertices but on the faces of $G$.} of the Ising partition function, see for example \cite{Baxter}:
\begin{equation*}\label{eq:Zhightemp}
Z^J=\left(\prod_{e\in E(G)}\cosh(J_e)\right)2^{|V(G)|}\sum_{\mathsf{C}\in\mathsf{P}} \prod_{e\in\mathsf{C}} \tanh(J_e),
\end{equation*}
where $\mathsf{P}$ is the family of all polygonal contours drawn on $G$, for
which every edge of $G$ is used at most once. This expansion defines a measure on the set of polygonal contours $\mathsf{P}$ of $G$: the probability of occurrence of a polygonal contour $\mathsf{C}$ is proportional to the product of the weights of the edges it contains, where the weight of an edge $e$ is $\tanh(J_e)$.

Here comes the correspondence: to any contour configuration $\mathsf{C}$ coming from the high-temperature expansion of the Ising model on $G$, we associate $2^{|V(G)|}$ dimer configurations on $\GD$: edges present (resp. absent) in $\mathsf{C}$ are
absent (resp. present) in the corresponding dimer configuration of $\GD$. Once
the state of these edges is fixed, there is, for every decorated vertex,
exactly two ways to complete the configuration into a dimer configuration. Figure \ref{fig:Fisher_correspondence} below gives an example in the case where $G$ is the square lattice $\ZZ^2$.\\

\begin{figure}[ht]
\begin{center}
\includegraphics[width=\linewidth]{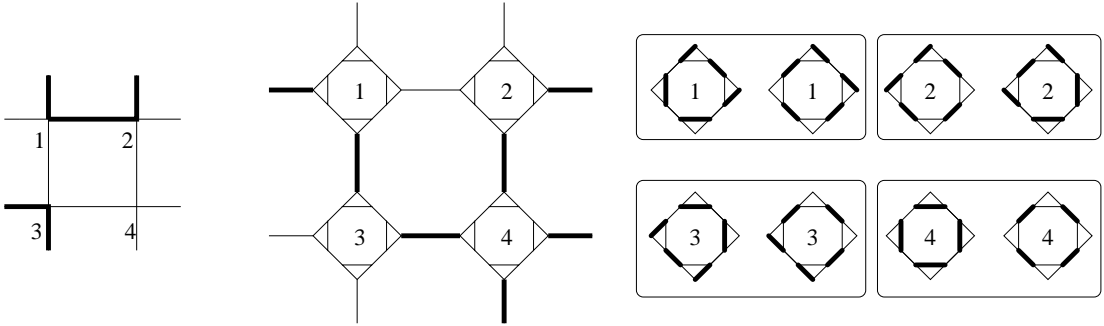}
\caption{Polygonal contour of $\ZZ^2$, and corresponding dimer configurations
  of the associated Fisher graph.}
\label{fig:Fisher_correspondence}
\end{center}
\end{figure}

Let us assign, to an edge $e$ of $\GD$, weight $\nu_e=1$, if it belongs to a
decoration; and weight $\nu_e=\coth{J_e}$, if it corresponds to an edge of $G$.
Then the correspondence is measure-preserving: every contour configuration $\mathsf{C}$
has the same number ($2^{|V(G)|}$) of images by this correspondence, and the product of the weights of the edges in $\mathsf{C}$, $\prod_{e\in \mathsf{C}} \tanh(J_e)$
is proportional to the weight $\prod_{e\not\in\mathsf{C}} \coth(J_e)$  of any of its corresponding dimer configurations for a proportionality factor, $\prod_{e\in E(G)} \tanh(J_e)$, which is independent of $\mathsf{C}$.

As a consequence of Fisher's correspondence, we have the following relation between the Ising and dimer partition functions:
\begin{equation}
Z^J=\left(\prod_{e\in E(G)}\sinh(J_e)\right)\Z^\nu.
\label{eq:Z-fisher}
\end{equation}

Fisher's correspondence between Ising contour configurations and dimer configurations naturally 
extends to the case where $G$ is an infinite planar graph.

\subsection{Critical dimer model on Fisher graphs}

Consider a critical $Z$-invariant Ising model on an isoradial graph $G$ on the torus, or on the whole plane. 
Then, the dimer weights of the corresponding dimer model on the Fisher graph $\GD$ are:
\begin{equation*}
  \nu_e=\begin{cases}
    1 & \text{if $e$ belongs to a decoration,}\\
    \nu(\theta_e)=\cot\left(\frac{\theta_e}{2}\right) & \text{if $e$ comes from an edge of $G$.}
  \end{cases}
\end{equation*}
We refer to these weights as {\em critical dimer weights}, and to the corresponding dimer model as {\em critical dimer model on the Fisher graph $\GD$}.

\section{Kasteleyn matrix on critical infinite Fisher graphs}

In the whole of this section, we let $G$ be an infinite isoradial graph, and $\GD$ be the corresponding Fisher graph. We suppose that edges of $\GD$ are assigned the dimer critical weight function denoted by $\nu$. Recall that $G^{\diamond}$ denotes the diamond graph associated to $G$.

\subsection{Kasteleyn and inverse Kasteleyn matrix}\label{subsec41}

The key object used to obtain explicit expressions for the dimer model on the Fisher graph $\GD$ is the {\em Kasteleyn matrix} introduced by Kasteleyn in \cite{Kast61}. It is a weighted, oriented adjacency matrix of the graph $\GD$ defined as follows.

A {\em Kasteleyn orientation} of $\GD$ is an orientation of the edges of $\GD$ such that all 
elementary cycles are {\em clockwise odd}, i.e. when
traveling clockwise around the edges of any elementary cycle of $\GD$, the
number of co-oriented edges is odd. When the graph is planar, such an orientation always exists \cite{Kasteleyn}. For later purposes, we need to keep track of the orientation of the
edges of $\GD$.
We thus choose a specific Kasteleyn orientation of $\GD$ in which every triangle of every decoration is oriented clockwise. Having a Kasteleyn orientation of the graph $\GD$ then amounts to finding a Kasteleyn orientation of the planar graph obtained from $\GD$ by contracting each triangle to a single vertex, which exists by Kasteleyn's theorem \cite{Kasteleyn}. Refer to Figure \ref{fig:Kast_orientation} for an example of such an
orientation in the case where $G=\ZZ^2$.\\

\begin{figure}[h]
\begin{center}
\includegraphics[height=6cm]{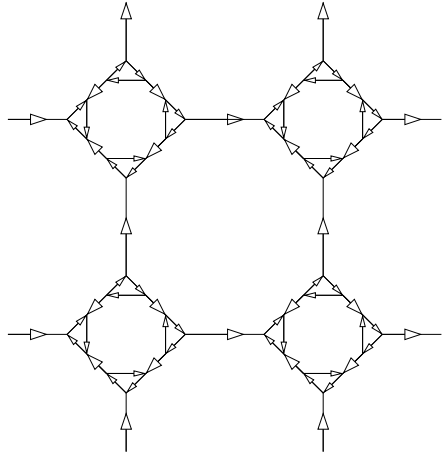}
\caption{An example of Kasteleyn orientation of the Fisher graph of $\ZZ^2$, in which every triangle of every decoration is oriented clockwise.}
\label{fig:Kast_orientation}
\end{center}
\end{figure}
The {\em Kasteleyn matrix} corresponding to such an orientation is an infinite matrix, whose rows and columns are indexed by vertices of $\GD$, defined by:

\begin{equation*}
K_{x,y}=\eps_{x,y}\nu_{xy},
\end{equation*}
where 
\begin{equation*}
\eps_{x,y}=
\begin{cases}
 1&\text{ if }x\sim y,\text{ and } x\rightarrow y\\
-1&\text{ if } x\sim y,\text{ and } x\leftarrow y\\
0&\text{ else}.
\end{cases}
\end{equation*}

Note that $K$ can be interpreted as an operator acting on $\CC^{V(\GD)}$:
\begin{equation*}
\forall f \in\CC^{V(\GD)},\quad (Kf)_x=\sum_{y\in V(\GD)}K_{x,y}f_{y}.
\end{equation*}

An {\em inverse of the Kasteleyn matrix $K$}, denoted $K^{-1}$, is an infinite matrix whose rows and columns are indexed by vertices of $\GD$, and which satisfies 
$KK^{-1}=\Id$.

\subsection{Local formula for an inverse Kasteleyn matrix}\label{subsc42}

In this section, we state Theorem \ref{inverse} proving an explicit {\em local} expression for the coefficients of an inverse $K^{-1}$ of the Kasteleyn matrix $K$. This inverse is the key object for the critical dimer model on the Fisher graph $\GD$. Indeed, it yields an explicit local expression for a Gibbs measure on dimer configurations of $\GD$, see Section \ref{sec:measure}. It also allows for a simple derivation of Baxter's formula for the free energy of the critical $Z$-invariant Ising model, see Section \ref{sec:free_energy}.

This section is organized as follows. The explicit expression for the coefficients of $K^{-1}$ given by Theorem \ref{inverse} below, see also Theorem \ref{thm:inverse_intro}, is a contour integral of an integrand involving two quantities:
\begin{itemize}
\item a complex-valued function depending on a complex parameter, and on the vertices
of the graph $\GD$ only, defined in Section \ref{subsec53}.
\item the {\em discrete exponential function} which first appeared in \cite{Mercat2}, see also \cite{Kenyon3}. It is a complex valued function depending  on a complex parameter, and on an edge-path between pairs of vertices of the graph $G$, defined in Section \ref{subsec:discr_exp}.
\end{itemize}
Theorem \ref{inverse} is then stated in Section \ref{subsec55}. Since the proof is long, it is postponed until Section \ref{sec:proof}. In Section \ref{sec:asympt}, using the same technique as \cite{Kenyon3}, we give the asymptotic expansion of the coefficient $K^{-1}_{x,y}$ of the inverse Kasteleyn matrix, as $|x-y|\rightarrow\infty$.

\subsubsection{Preliminary notations}\label{subsec51}

From now on, vertices of $\GD$ are written in normal symbol, and vertices of $G$ in boldface. Let $x$ be a vertex of $\GD$, then $x$ belongs to the decoration corresponding to a
unique vertex of $G$, denoted by~$\xb$. Conversely, vertices of $\GD$ of the decoration corresponding to a vertex $\xb$ of $G$ are labeled as follows, refer to Figure \ref{fig:notations} for an example. Let $d(\xb)$ be the degree of the vertex $\xb$ in $G$, then the corresponding decoration of $\GD$ consists of $d(\xb)$ triangles, labeled from $1$ to $d(\xb)$ in counterclockwise order. For the $k$-th triangle, let $v_k(\xb)$ be the vertex incident to an edge of $G$, and let $w_k(\xb),z_{k}(\xb)$ be the two other vertices, in counterclockwise order, starting from $v_{k}(\xb)$.

\begin{figure}[h]
\begin{center}
\includegraphics[height=3.3cm]{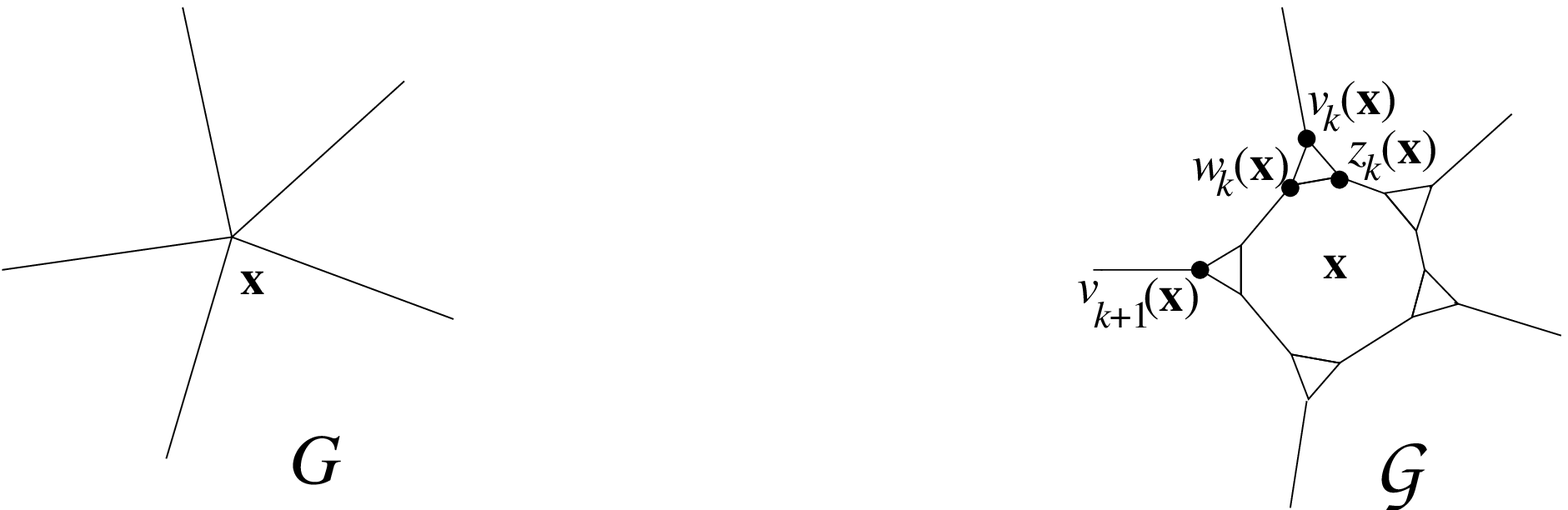}
\caption{Notations for vertices of $\GD$.}
\label{fig:notations}
\end{center}
\end{figure}

Later on, when no confusion occurs, we will drop the argument $\xb$ in
the above labeling. Define a vertex $x$ of $\GD$ to be of {\it type `$v$'}, if
$x=v_k(\xb)$ for some $k\in\{1,\cdots,d(\xb)\}$, and similarly for `$w$' and `$z$'.

The isoradial embedding of the graph $G$ fixes an embedding of the corresponding diamond graph $\GR$. There is a natural way of assigning rhombus unit-vectors of $\GR$ to vertices of $\GD$: for every vertex $\xb$ of $G$, and every $k\in\{1,\cdots,d(\xb)\}$, let us associate the rhombus unit-vector $e^{i\alpha_{w_{k}(\xb)}}$ to $w_k(\xb)$, $e^{i \alpha_{z_{k}(\xb)}}$ to $z_k(\xb)$, and the two rhombus-unit vectors $e^{i\alpha_{w_{k}(\xb)}}$, $e^{i\alpha_{z_{k}(\xb)}}$ to $v_k(\xb)$, as in Figure \ref{fig:rhombusvectors} below. Note that $e^{i\alpha_{w_k(\xb)}}=e^{i\alpha_{z_{k+1}(\xb)}}$.

\begin{figure}[ht]
\begin{center}
\includegraphics[height=3.6cm]{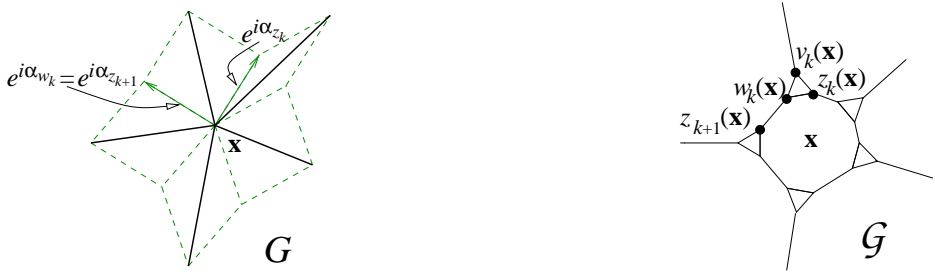}
\end{center}
\caption{Rhombus vectors of the diamond graph $\GR$ assigned to vertices of $\GD$.}
\label{fig:rhombusvectors}
\end{figure}

\subsubsection{Complex-valued function on the vertices of $\GD$}\label{subsec53}

Let us introduce the complex-valued function defined on vertices of $\GD$ and depending on a complex parameter, involved in the integrand of the contour integral of $K^{-1}$ given by Theorem \ref{inverse}. Define $f:V(\GD)\times\CC\rightarrow\CC$, by:
\begin{align*}
f(w_k(\xb),\lambda):=f_{w_k(\xb)}(\lambda)&=\frac{e^{i\frac{\alpha_{w_k(\xb)}}{2}}}{e^{i\alpha_{w_k(\xb)}}-\lambda},\\
f(z_k(\xb),\lambda):=f_{z_k(\xb)}(\lambda)&=-\frac{e^{i\frac{\alpha_{z_k(\xb)}}{2}}}{e^{i\alpha_{z_k(\xb)}}-\lambda},\\
f(v_k(\xb),\lambda):=f_{v_k(\xb)}(\lambda)&=f_{w_k(\xb)}(\lambda)+f_{z_k(\xb)}(\lambda),
\end{align*}
for every $\xb\in G$, and every $k\in\{1,\cdots,d(\xb)\}$. In order for the function $f$ to be well defined, the angles $\alpha_{w_k(\xb)}$,
$\alpha_{z_k(\xb)}$ need to be well defined mod $4\pi$, indeed half-angles need to be well defined mod $2\pi$. Let us define them inductively as follows, see also Figure \ref{fig:angles}. 
The definition strongly depends on the Kasteleyn orientation introducted in Section \ref{subsec41}.
Fix a vertex $\xb_0$
of $G$, and set $\alpha_{z_1(\xb_0)}=0$. Then, for vertices of $\GD$ in
the decoration of a vertex $\xb\in G$, define:
\begin{align}\label{eq:angle_intra_deco}
\nonumber
&\alpha_{w_k(\xb)}=\alpha_{z_k(\xb)}+2\theta_k(\xb),\text{ where
  $\theta_k(\xb)>0$ is the rhombus half-angle of Figure \ref{fig:angles},}\\ 
&\alpha_{z_{k+1}(\xb)}=
\begin{cases}
\alpha_{w_k(\xb)}& \text{if the edge } w_k(\xb) z_{k+1}(\xb) \text{ is oriented
from } w_k(\xb) \text{ to } z_{k+1}(\xb) \\
\alpha_{w_k(\xb)}+2\pi&\text{else}.
\end{cases}
\end{align}
Here is the rule defining angles in the neighboring decoration, corresponding to a vertex $\yb$ of $G$. Let $k$ and $\l$ be indices such that $v_k(\xb)$ is adjacent to $v_\l(\yb)$ in $\GD$. Then, define:
\begin{equation}\label{eq:angle_chg_deco}
\alpha_{w_\l(\yb)}=
\left\{
\begin{array}{ll}
\alpha_{w_k(\xb)}-\pi&\text{if the edge } v_k(\xb)v_\l(\yb) \text{ is
  oriented from } v_k(\xb)\text { to } v_\l(\yb)\\
\alpha_{w_k(\xb)}+\pi&\text{ else}.
\end{array}
\right.
\end{equation}

\begin{figure}[ht]
\begin{center}
\includegraphics[width=13cm]{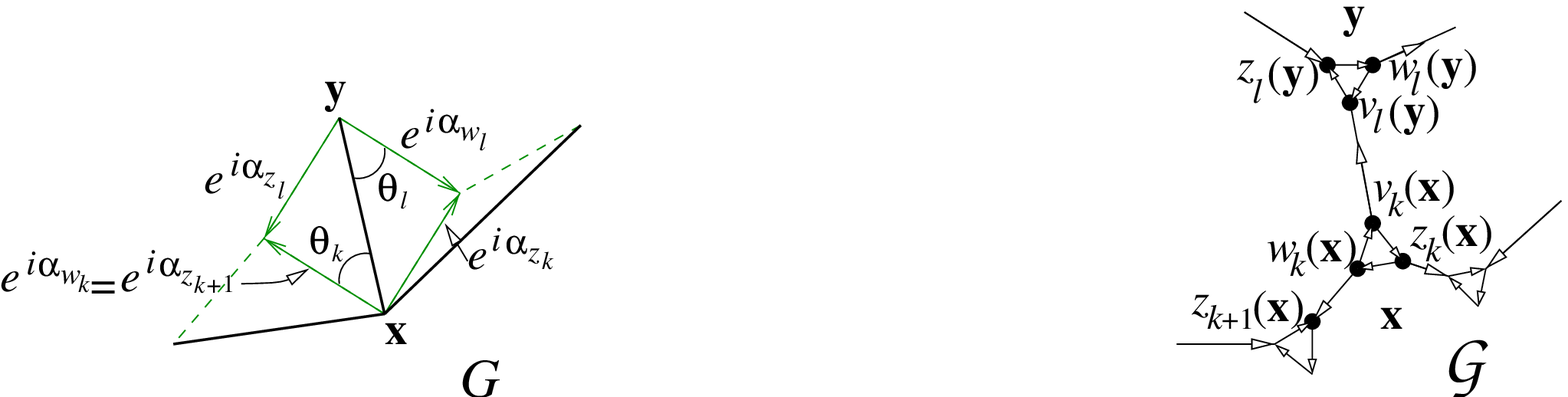}
\end{center}
\caption{Notations for the definition of the angles in $\RR/4\pi\ZZ$.}
\label{fig:angles}
\end{figure}

\begin{lem}\label{lem:angles4pi}
For every vertex $\xb$ of $G$, and every $k\in\{1,\cdots,d(\xb)\}$, the angles 
$\alpha_{w_k(\xb)}$, $\alpha_{z_k(\xb)}$, are well defined in $\RR/4\pi\ZZ$.
\end{lem}

\begin{proof}
It suffices to show that when doing the inductive procedure around a
cycle of $\GD$, we obtain the same angle modulo $4\pi$. There are two
types of cycles to consider: inner cycles of decorations, and cycles
of $\GD$ coming from the boundary of a face of $G$. These cycles are represented in Figure~\ref{fig:contour_verifangles}.

\begin{figure}[ht]
  \begin{center}
    \includegraphics[width=12.5cm]{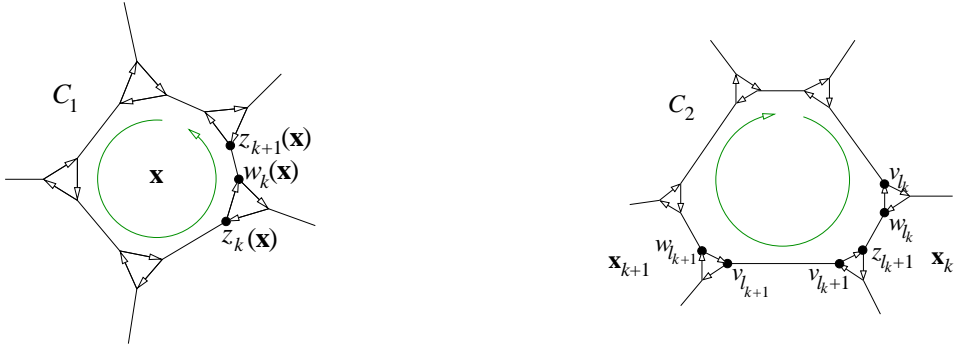}
    \caption{The two types of cycles on which the definition of the angles in $\RR/4\pi\ZZ$ needs to be checked. Left: $C_1$, the inner cycle of a decoration, oriented counterclockwise. Right: $C_2$, a cycle coming from the boundary of a face of $G$, oriented clockwise.}
    \label{fig:contour_verifangles}
  \end{center}
\end{figure}

Let $C_1$ be a cycle of the first type, that is $C_1$ is the inner
cycle of a decoration corresponding to a vertex $\xb$ of $G$, oriented counterclockwise. Let $d=d(\xb)$ be the degree of the vertex $\xb$ in $G$, then:
\begin{equation*}
  C_1= \bigl( z_1(\xb),w_1(\xb),\dots, z_d(\xb), w_d(\xb), z_1(\xb) \bigr)=\bigl(z_1,w_1,\dots,z_d,w_d,z_1\bigr).
\end{equation*}
According to our choice of Kasteleyn orientation, all edges of this cycle are oriented counterclockwise with respect to the face surrounded by $C_1$, except an odd number $n$ of edges of the form $w_k z_{k+1}$. By definition of the angles \eqref{eq:angle_intra_deco}, when jumping from $z_k$ to $z_{k+1}$, the change in angle is $2\theta_k$ if $w_{k} z_{k+1}$ is oriented counterclockwise, and $2\theta_k+2\pi$ otherwise. The total change along the cycle is thus,
\begin{equation*}
  \sum_{k=1}^{d}2\theta_k + 2\pi n = 2\pi (n +1) \equiv 0 \ \text{ mod }4\pi.
\end{equation*}

Let $C_2$ be a cycle of the second type, around a face touching $m$ decorations corresponding to vertices $\xb_1,\cdots,\xb_m$ of $G$. Suppose that $C_2$ is oriented clockwise.
\begin{multline*}
  C_2 = \bigl(w_{\ell_1}(\xb_1), z_{\ell_1+1}(\xb_1), v_{\ell_1 + 1}(\xb_1),\dots,\\
  \dots,v_{\ell_m}(\xb_m),w_{\ell_m}(\xb_m), z_{\ell_m+1}(\xb_m), v_{\ell_m + 1}(\xb_m),v_{\ell_1}(\xb_1),w_{\ell_1}(\xb_1)\bigr).
\end{multline*}
The cycle $C_2$ contains $4m$ edges, and the number $n$ of co-oriented edges along $C_2$ is odd by definition of a Kasteleyn orientation. By our choice of Kasteleyn orientation, the only edges that may be co-oriented along the cycle, are either of the form $w_{\ell_{k}}(\xb_k) z_{\ell_k +1 }(\xb_k)$, or of the form $v_{\ell_k +1}(\xb_k) v_{\ell_{k+1}}(\xb_{k+1})$.

For every $k\in\{1,\dots,m\}$, there are two different contributions to the total change in angle from $w_{\ell_k}(\xb_k)$ to $w_{\ell_{k+1}}(\xb_{k+1})$. First, from $w_{\ell_k}(\xb_k)$ to $z_{\ell_k +1}(\xb_k)$, by \eqref{eq:angle_intra_deco},
there is a contribution of $2\pi$ if the corresponding edge
is co-oriented. Then, from $z_{\ell_k +1}(\xb_k)$ to $w_{\ell_{k+1}}(\xb_{k+1})$, by \eqref{eq:angle_chg_deco}, the angle is changed by $2\theta_{\ell_{k}+1}(\xb_k)-\pi =2\theta_{\ell_{k+1}}(\xb_{k+1})-\pi$, plus an extra contribution of $2\pi$ if the edge $v_{\ell_k +1}(\xb) v_{\ell_{k+1}}(\xb_{k+1})$ is co-oriented.
Thus the total change in angle is:
\begin{equation*}
  \sum_{k=1}^{m} (2\theta_{\ell_k}(\xb_k)-\pi) + 2\pi n.
\end{equation*}
But the terms $\pi -2\theta_{\ell_k}(\xb_k)$ are the angles of the rhombi at the center of the face surrounded by $C_2$, which sum to $2\pi$.
Therefore, the total change is equal to $2\pi(n-1)$ which is congruent to $0$ mod $4\pi$.
\end{proof}

These angles at vertices are related to the notion of spin structure on surface graphs (see \cite{CimaReshe1}, \cite{CimaReshe2}, \cite{Kuperberg}). A spin structure on a surface is equivalent to the data of a vector field with even index singularities. Kuperberg explained in \cite{Kuperberg} how to construct from a Kasteleyn orientation a vector field with odd index singularities at vertices of $G$. One can then obtain the spin structure by merging the singularities into pairs using a reference dimer configuration.

The angles $\alpha_{w_k(\xb)}$ and $\alpha_{z_k(\xb)}$ defined here are directly related to the direction of the vector field with even index singularities  obtained from Kuperberg's construction applied to our choice of Kasteleyn orientation, and from the pairing of singularities corresponding to the following reference dimer configuration:
\begin{itemize}
 \item $w_k(\xb) \leftrightarrow z_k(\xb)$,
\item $v_k(\xb) \leftrightarrow v_\l(\yb)$ if they are neighbors.
\end{itemize}

\subsubsection{Discrete exponential functions\label{subsec:discr_exp}}

Let us define the {\em discrete exponential function}, denoted $\expo$, involved in the integrand of the contour integral of $K^{-1}$ given by Theorem \ref{inverse}. This function first appeared in \cite{Mercat2}, see also \cite{Kenyon3}. In order to simplify notations, we use a different labeling
of the rhombus vectors of $\GR$. Let $\xb,\yb$ be two vertices of
$G$, and let $\yb=\xb_1,\xb_2,\cdots,\xb_{n+1}=~\xb$ be an edge-path of $G$ from
$\yb$ to $\xb$. The complex vector $\xb_{j+1}-\xb_j$ is the sum of two unit complex numbers $e^{i\beta_j}+e^{i\gamma_j}$ representing edges of the rhombus in $G^{\diamond}$ associated to the edge $\xb_j \xb_{j+1}$.

Then, $\expo:V(G)\times V(G)\times \CC\rightarrow\CC$ is defined by:
\begin{equation*}
 \expo(\xb,\yb,\lambda):=\expo_{\xb,\yb}(\lambda)=\prod_{j=1}^{n}
\left(\frac{e^{i\beta_j}+\lambda}{e^{i\beta_j}-\lambda}\right)
\left(\frac{e^{i\gamma_j}+\lambda}{e^{i\gamma_j}-\lambda}\right).
\end{equation*}
The function is well defined (independent of the choice of edge-path of $G$ from $\yb$ to $\xb$) since the product of the multipliers around a rhombus is $1$.

\subsubsection{Inverse Kasteleyn matrix}\label{subsec55}

We now state Theorem \ref{inverse} proving an explicit local formula for the coefficients of an inverse $K^{-1}$ of the Kasteleyn matrix $K$. The vertices $x$ and $y$ of $\GD$ in the statement should be thought of as being one of $w_k(\xb)$, $z_k(\xb)$, $v_k(\xb)$ for some $\xb\in G$ and some $k\in\{1,\cdots,d(\xb)\}$, and similarly for $y$.
The proof of Theorem \ref{inverse} is postponed until Section \ref{sec:proof}.

\begin{thm}\label{inverse}
Let $x,y$ be any two vertices of $\GD$. Then the infinite matrix $K^{-1}$, whose coefficient $K^{-1}_{x,y}$ is given by \eqref{invKast} below, is an inverse Kasteleyn matrix.
\begin{equation}\label{invKast}
K^{-1}_{x,y}=
\frac{1}{(2\pi)^2}
\oint_{\C_{x,y}}f_{x}(\lambda)f_{y}(-\lambda)
\expo_{\xb,\yb}(\lambda)\log\lambda \ud\lambda+C_{x,y}.
\end{equation}
The contour of
integration $\C_{x,y}$ is a simple closed curve oriented counterclockwise containing all
poles of the integrand, and avoiding the half-line $d_{x,y}$ starting from zero \footnote{ In most cases, the half-line $d_{x,y}$ is oriented from $\hat{\xb}$ to $\hat{\yb}$, where $\hat{\xb}$, $\hat{\yb}$ are vertices of $\GR$ at distance at most two from $\xb$, $\yb$, constructed in Section \ref{subsec73}. Refer to this section and Remark \ref{rem:rem1} for more details.}. The
constant $C_{x,y}$ is given by
\begin{equation*}
C_{x,y}=
\begin{cases}
\frac{1}{4}&\text{ if $x=y=w_k(\xb)$}\\
-\frac{1}{4}&\text{ if $x=y=z_k(\xb)$}\\
\frac{(-1)^{n(x,y)}}{4}&\text{ if $\xb=\yb$, $x\neq y$, and $x$ and $y$ are of type
  `$w$' or `$z$'}\\
0&\text{ else},
\end{cases}
\end{equation*}
where $n(x,y)$ is the number of edges oriented clockwise in the clockwise arc from
$x$ to $y$ of the inner cycle of the decoration corresponding to $\xb$.
\end{thm}

\begin{rem}\label{rem:localformula}
The explicit expression for $K^{-1}_{x,y}$ given in \eqref{invKast} has the very interesting feature of being {\em local}, {\em i.e.} it only depends on the geometry of the embedding of the isoradial graph $G$ on a path between $\hat{\xb}$ and $\hat{\yb}$, where $\hat{\xb}$ and $\hat{\yb}$ are vertices of $\GR$ at distance at most $2$ from $\xb$ and $\yb$, constructed in Section \ref{subsec73}. A nice consequence of this property is the following.
For $i=1,2$, let $G_i$ be an isoradial graph with corresponding Fisher graph $\GD_i$. Let $K_i$ be the Kasteleyn matrix of the graph $\GD_i$, whose edges are assigned the dimer critical weight function. Let $(K_i)^{-1}$ be the inverse of $K_i$ given by Theorem \ref{inverse}. If $G_1$ and $G_2$ coincide on a ball $B$, and if the Kasteleyn orientations on $\GD_1$ and $\GD_2$ are chosen to be the same in the ball $B$, then for every couple of vertices $(x,y)$ in $B$ at distance at least $2$ from the boundary, we have
  \begin{equation*}
    (K_{1})^{-1}_{x,y}= (K_2)^{-1}_{x,y}.
  \end{equation*}
\end{rem}

\subsubsection{Asymptotic expansion of the inverse Kasteleyn matrix}
\label{sec:asympt}

As a corollary to Theorem \ref{inverse}, and using computations analogous to those of Theorem~$4.3$ of \cite{Kenyon3}, we obtain the asymptotic expansion for the coefficient $K^{-1}_{x,y}$ of the inverse Kasteleyn matrix, as $|\xb-\yb|\rightarrow\infty$. In order to give a concise statement, let us introduce the following simplified notations. When $x$ is of type `$w$' or `$z$', let $e^{i\alpha}$ denote the corresponding rhombus unit-vector of $\GR$; and when $x$ is of type `$v$', let $e^{i\alpha_1}$, $e^{i\alpha_2}$ be the two corresponding rhombus unit-vectors of $\GR$, defined in Section \ref{subsec51}. Moreover, set:
\begin{equation*}
\epsilon_{x}=
\begin{cases}
1 &\text{ if $x$ is of type `$w$' or `$v$'}\\
-1 &\text{ if $x$ is of type `$z$'}.
\end{cases}
\end{equation*}
A superscript ``prime'' is added to these notations for the vertex $y$, and $\epsilon_y$ is defined in a similar way. 

\begin{cor}\label{cor:asymptotics}
The asymptotic expansion, as $|\xb-\yb|\rightarrow\infty$, of the coefficient $K^{-1}_{x,y}$ of the inverse Kasteleyn matrix of Theorem \ref{inverse} is:
\begin{align*}
&K^{-1}_{x,y}=\\
&=\frac{\epsilon_x\epsilon_y}{2\pi}\begin{cases}
  \Im\Bigl(\frac{e^{i\frac{\alpha+\alpha'}{2}}}{\xb-\yb}\Bigr)+o\Bigl(\frac{1}{|\xb-\yb|}\Bigr)
&\text{if $x$ and $y$ are of type `$w$' or `$z$'},\\
\Im\Bigl(
\frac{e^{i\frac{\alpha}{2}}(e^{i\frac{\alpha_1'}{2}}-e^{i\frac{\alpha_2'}{2}})}{\xb-\yb}\Bigr)+o\Bigl(\frac{1}{|\xb-\yb|}\Bigr)
&\text{if $x$ is of type `$w$' or `$z$', $y$ is of type `$v$'},\\
\Im\Bigl(
\frac{(e^{i\frac{\alpha_1}{2}}-e^{i\frac{\alpha_2}{2}})(e^{i\frac{\alpha_1'}{2}}-e^{i\frac{\alpha_2'}{2}})}{\xb-\yb}\Bigr)+o\Bigl(\frac{1}{|\xb-\yb|}\Bigr)
&\text{if $x$ and $y$ are of type `$v$'}.\\
\end{cases}
\end{align*}
\end{cor}

\begin{proof}
We follow the proof of Theorem $4.3$ of \cite{Kenyon3}. By Theorem \ref{inverse}, $K^{-1}_{x,y}$ is given by the integral:
\begin{equation}\label{eq:asymp}
K^{-1}_{x,y}=
\frac{1}{(2\pi)^2}
\int_{\C_{x,y}}f_{x}(\lambda)f_{y}(-\lambda)
\expo_{\xb,\yb}(\lambda)\log\lambda\ud\lambda+C_{x,y},
\end{equation}
where $\C_{x,y}$ is a simple closed curve oriented counterclockwise containing all poles of the integrand, and avoiding the half-line $d_{x,y}$ starting from $0$ in the direction from $\hat{\xb}$ to $\hat{\yb}$, where $\hat{\xb}$, $\hat{\yb}$ are vertices of $\GR$, at distance at most $2$ from $\xb$ and $\yb$, constructed in Section \ref{subsec73}. We choose $d_{x,y}$ to be the origin for the angles, {\em i.e.} the positive real axis $\RR_+$. The coefficient $C_{x,y}$ equals to $0$ whenever $x$ and $y$ are not in the same decoration. It will be a fortiori the case when they are far away from each other.

It suffices to handle the case where both $x$ and $y$ are of type `$w$' or `$z$'. The other cases are then derived using the relation between the functions $f$: 
\begin{equation*}
f_{v_k}(\lambda) = f_{w_k}(\lambda) + f_{z_k}(\lambda).
\end{equation*}
Define,
\begin{equation*}
  \xi = \xb -\yb = \sum_{j=1}^n e^{i\beta_j}+e^{i\gamma_j}.
\end{equation*}
where $e^{i\beta_1}, e^{i\gamma_1},\dots,e^{i\beta_n},e^{i\gamma_n}$ are the steps of a path in $\GR$ from $\yb$ to $\xb$. When $\xb$ and $\yb$ are far apart, since $\hat{\xb}$ and $\hat{\yb}$ are at distance at most $2$ from $\xb$ and $\yb$, the direction of $\overrightarrow{\xb\yb}$ is not very different from that of $\overrightarrow{\hat{\xb}\hat{\yb}}$, ensuring that $\Re(\xi) < 0 $.

The contour $\C_{x,y}$ can be deformed to a curve running counterclockwise around the ball of radius $R$ ($R$ large) around the origin from the angle $0$ to $2\pi$, then along the positive real axis, from $R$ to $r$ ($r$ small), then clockwise around the ball of radius $r$ from the angle $2\pi$ to $0$, and then back along the real axis from $r$ to $R$. The rational fraction $f_{x}(\lambda)f_{y}(-\lambda)
\expo_{\xb,\yb}(\lambda)$ behaves like $O(1)$ when $\lambda$ is small, and like $O\left(\frac{1}{\lambda^2}\right)$, when $\lambda$ is large. As a consequence, the contribution to the integral around the balls of radius $r$ and $R$ converges to $0$, as we let $r\rightarrow 0$ and $R\rightarrow \infty$. The logarithm differs by $2i\pi$ on the two sides of the ray $\RR_+$, the integral \eqref{eq:asymp} is therefore equal to:
\begin{equation*}
 K^{-1}_{x,y}=\frac{1}{2 i \pi}\int_{0}^{\infty}f_{x}(\lambda)f_{y}(-\lambda)
\expo_{\xb,\yb}(\lambda)\ud\lambda.
\end{equation*}

As in \cite{Kenyon3}, when $|\xb-\yb|$ is large, the main contribution to this integral comes from a neighborhood of the origin and a neighborhood of infinity. When $\lambda$ is small, we have:
\begin{equation*}
  \frac{e^{i\frac{\alpha}{2}}}{e^{i\alpha}-\lambda}=e^{-i\frac{\alpha}{2}+O(\lambda)}
  ,\quad\frac{e^{i\frac{\alpha'}{2}}}{e^{i\alpha}+\lambda}=e^{-i\frac{\alpha'}{2}+O(\lambda)},\quad \frac{e^{i\beta}+\lambda}{e^{i\beta}-\lambda}=\exp\left( 2e^{-i\beta}\lambda+O(\lambda^3) \right).
\end{equation*}
Thus, for small values of $\lambda$:
\begin{align*}
f_{x}(\lambda)f_{y}(-\lambda)\expo_{\xb,\yb}(\lambda) =\epsilon_x\epsilon_y
\frac{e^{i\frac{\alpha}{2}}}{e^{i\alpha} -\lambda} \frac{e^{i\frac{\alpha'}{2}}}{e^{i\alpha'}+\lambda} \prod_{j=1}^n \frac{e^{i\beta_j}+\lambda}{e^{i\beta_j}-\lambda}\frac{e^{i\gamma_j}+\lambda}{e^{i\gamma_j}-\lambda}\\
= \epsilon_x\epsilon_y e^{-i\frac{\alpha+\alpha'}{2}}\exp\left(2\overline{\xi}\lambda + O(\lambda) +O(n\lambda^3)\right).
\end{align*}

Integrating this estimate from $0$ to $\frac{1}{\sqrt{n}}$ gives
\begin{align*}
\int_0^{\frac{1}{\sqrt{n}}} f_x(\lambda) f_y(-\lambda) \expo_{\xb,\yb}(\lambda) \ud \lambda &= \epsilon_x\epsilon_y\int_{0}^{\frac{1}{\sqrt{n}}} e^{-i\frac{\alpha+\alpha'}{2}} \exp\left( 2\bar{\xi}\lambda +O(n^{-\frac{1}{2}}) \right) \ud \lambda \\
  &= -\epsilon_x\epsilon_y\frac{e^{-i\frac{\alpha+\alpha'}{2}}}{2\bar{\xi}}\left( 1+O\left(\frac{1}{\sqrt{n}}\right) \right).
\end{align*}

Similarly, when $\lambda$ is large, we have:
\begin{equation*}
\frac{e^{i\alpha}}{e^{i\alpha}-\lambda}= -\frac{e^{i\frac{\alpha}{2}+O(\lambda^{-1})}}{\lambda}, 
\quad \frac{e^{i\alpha'}}{e^{i\alpha'}+\lambda} = \frac{e^{i\frac{\alpha'}{2}+O(\lambda^{-1})}}{\lambda},\quad \frac{e^{i\beta}+\lambda}{e^{i\beta}-\lambda}=
-\exp\left( 2 e^{i\beta}\lambda^{-1}+O(\lambda^{-3}) \right).
\end{equation*}
Thus for large values of $\lambda$,
\begin{align*}
&f_{x}(\lambda)f_{y}(-\lambda)\expo_{\xb,\yb}(\lambda)=-\epsilon_x\epsilon_y\frac{e^{i\frac{\alpha+\alpha'}{2}}}{\lambda^2}\exp\left(2\xi\lambda^{-1} +O(\lambda^{-1})+O(n\lambda^{-3})\right).
\end{align*}

Computing the integral of this estimate for $\lambda\in (\sqrt{n},\infty)$ gives
\begin{equation*}
\int_{\sqrt{n}}^{\infty} f_x(\lambda)f_y(-\lambda) \expo_{\xb,\yb}(\lambda) \ud \lambda =\epsilon_x\epsilon_y \frac{e^{i\frac{\alpha+\alpha'}{2}}}{2\xi}\left( 1+O\left( \frac{1}{\sqrt{n}} \right) \right).
\end{equation*}
The rest of the integral between $n^{-\frac{1}{2}}$ and $n^{\frac{1}{2}}$ is negligible (see \cite{Kenyon3}). As a consequence,
\begin{equation*}
K^{-1}_{x,y} = \frac{\epsilon_x\epsilon_y}{2 i \pi}\left( \frac{e^{i\frac{\alpha+\alpha'}{2}}}{2\xi}\left( 1+o(1) \right)-\frac{e^{-i\frac{\alpha+\alpha'}{2}}}{2\bar{\xi}}\left( 1+o(1) \right)\right) = \frac{\epsilon_x\epsilon_y}{2\pi}\Im\left( \frac{e^{i\frac{\alpha+\alpha'}{2}}}{\xb-\yb} \right)+ o\left( \frac{1}{|\xb-\yb|} \right).
\end{equation*}
\end{proof}

\begin{rem}
Using the same method with an expansion of the integrand to a higher order would lead to a more precise asymptotic expansion of $K^{-1}_{x,y}$.
\end{rem}

\section{Critical dimer model on infinite Fisher graphs}

Let $\GD$ be an infinite Fisher graph obtained from an infinite isoradial graph $G$. Assume that edges of $\GD$ are assigned the dimer critical weight function $\nu$. In this section, we give a full description of the critical dimer model on the Fisher graph $\GD$, consisting of explicit expressions which only depend on the local geometry of the underlying isoradial graph $G$. More precisely, in Section \ref{sec:measure}, using the method of \cite{Bea1}, we give an explicit local formula for a Gibbs measure on dimer configurations of $\GD$, involving the inverse Kasteleyn matrix of Theorem \ref{inverse}. When the graph is periodic, this measure coincides with the Gibbs measure of \cite{isoising1}, obtained as weak limit of Boltzmann measures on a natural toroidal exhaustion. Then, in Section \ref{sec:free_energy}, we assume that the graph $\GD$ is periodic and, using the method of \cite{Kenyon3}, we give an explicit local formula for the free energy of the critical dimer model on $\GD$. As a corollary, we obtain Baxter's celebrated formula for the free energy of the critical $Z$-invariant Ising model.

\subsection{Local formula for the critical dimer Gibbs measure}
\label{sec:measure}

A {\em Gibbs measure} on the set of dimer configurations $\M(\GD)$ of $\GD$, is a probability measure on $\M(\GD)$, which satisfies the following. If one fixes a perfect matching in an annular region of $\GD$, then perfect matchings inside and outside of this annulus are independent. Moreover, the probability of occurrence of an interior matching is proportional to the product of the edge weights.

In order to state Theorem \ref{thm:measure}, we need the following definition. Let $\mathcal{E}$ be a finite subset of edges of $\GD$. The \emph{cylinder set} $A_{\mathcal{E}}$ is defined to be the set of dimer configurations of $\GD$ containing the subset of edges $\mathcal{E}$; it is often convenient to identify $\mathcal{E}$ with $A_{\mathcal{E}}$. Let $\mathcal{F}$ be the $\sigma$-algebra of $\mathcal{M}(\GD)$ generated by the cylinder sets $\bigr(A_{\mathcal{E}}\bigr)_{\mathcal{E}\text{ finite}}$. Recall that $K$ denotes the infinite Kasteleyn matrix of the graph $\GD$, and let $K^{-1}$ be the matrix inverse given by Theorem \ref{inverse}. 

\begin{thm}\label{thm:measure}
There is a unique probability measure $\P$ on $(\M(\GD),\F)$, such that for every finite collection of edges $ \mathcal{E}=\{e_1=x_1 y_1,\cdots,e_k=x_k y_k\}\subset E(\GD)$, the probability of the corresponding cylinder set is
\begin{equation}\label{eq:Gibbs}
\P(A_{\mathcal{E}})=\P(e_1,\cdots,e_k)=\left(\prod_{i=1}^k K_{x_i,y_i} \right)\Pf\bigl(\,^t(K^{-1})_{\{x_1,y_1,\cdots,x_k,y_k\}}\bigr),
\end{equation}
where $K^{-1}$ is given by Theorem \ref{inverse}, and $(K^{-1})_{\{x_1,y_1,\cdots,x_k,y_k\}}$ is the sub-matrix of $K^{-1}$ whose rows and columns are indexed by vertices $\{x_1,y_1,\cdots,x_k,y_k\}$.

Moreover, $\P$ is a Gibbs measure. When $\GD$ is $\ZZ^2$-periodic, $\P$ is the Gibbs measure obtained as weak limit of the Boltzmann measures $\P_n$ on the toroidal exhaustion $\{\GD_n=\GD/(n\ZZ^2)\}_{n\ge 1}$ of  $\mathcal{G}$.
\end{thm}
\begin{proof}
The idea is to use Kolmogorov's extension theorem. The structure of the proof is taken from \cite{Bea1}. Let $\mathcal{E}=\bigl\{e_1=x_1 y_1,\dots e_n = x_n y_n\bigr\}$ be a finite subset of edges of $\GD$. Denote by $\mathcal{F}_{\mathcal{E}}$ the $\sigma$-algebra generated by the cylinders $\bigl(A_{\mathcal{E}'}\bigr)_{\mathcal{E}'\subset\mathcal{E}}$.
We define an additive function $\mathcal{P}_{\mathcal{E}}$ on $\mathcal{F}_{\mathcal{E}}$, by giving its value on every cylinder set $A_{\mathcal{E'}}$, with $\mathcal{E}'=\{e_{i_1},\dots,e_{i_k}\} \subset \mathcal{E}$: 
\begin{equation*}
  \mathcal{P}_{\mathcal{E}}(A_{\mathcal{E'}}) = \left(\prod_{j=1}^k K_{x_{i_j},y_{i_j}} \right)\Pf\bigl({}^t(K^{-1})_{\{x_{i_1},y_{i_1},\cdots,x_{i_k},y_{i_k}\}}\bigr).
\end{equation*}
It is a priori not obvious that $\P_{\E}$ defines a probability measure on $\F_\E$. Let us show that this is indeed the case. Let $V_{\mathcal{E}}$ be the subset of vertices of $G$ to which the decorations containing $x_1, y_1, \dots, x_n, y_n$ retract. Define $Q$ to be a simply-connected subset of rhombi of $\GR$ containing all rhombi adjacent to vertices of 
$V_{\mathcal{E}}$. 

By Proposition 1 of \cite{Bea1}, there exists a $\ZZ^2$-periodic rhombus tiling of the plane containing $Q$, which we denote by ${\GR}^p$. Moreover, the rhombus tiling ${\GR}^p$ is the diamond graph of a unique isoradial graph $G^{p}$ whose vertices contains the subset ${V}_{\mathcal{E}}$. Let $\GD^{p}$ be the Fisher graph of $G^p$, then the graphs $\GD$ and $\GD^{p}$ coincide on a ball $B$ containing $x_1, y_1, \dots, x_n, y_n$.

Endow edges of $\GD^{p}$ with the critical weights and a periodic Kasteleyn orientation, coinciding with the orientation of the edges of $\GD$ on the common ball $B$.
We know by Remark \ref{rem:localformula}, that the coefficients of the inverses $(K^{p})^{-1}$ and $K^{-1}$ given by Theorem \ref{inverse} are equal for all pairs of vertices in $B$.

On the other hand, we know by Corollary \ref{cor:asymptotics} that the coefficient $(K^p)^{-1}_{x,y}\rightarrow 0$, as $|\xb-\yb|\rightarrow \infty$. By Proposition 5 of \cite{isoising1}, stating uniqueness of the inverse Kasteleyn matrix decreasing at infinity in the periodic case, we deduce that $(K^{p})^{-1}$ is in fact the inverse computed in \cite{isoising1} by Fourier transform. In Theorem $6$ of \cite{isoising1}, we use the inverse $(K^{p})^{-1}$ to construct the Gibbs measure $\mathcal{P}^p$ on $\mathcal{M}(\mathcal{G}^p)$ obtained as weak limit of Boltzmann measures on the natural toroidal exhaustion of $\GD^p$, which has the following explicit expression: let $\{e_1'=x_1' y_1',\cdots,e_k'=x_k'y_k'\}$ be a subset of edges of $\GD^p$, then
\begin{equation*}
  \mathcal{P}^p(e'_1,\dots e'_k)= \left(\prod_{j=1}^{k} K^p_{x'_j y'_j} \right) \Pf \bigl({}^t(K^p)^{-1}_{\{x'_1,\dots,y'_k\}}\bigr).
\end{equation*}

In particular, the expression of the restriction of $\mathcal{P}^p$ to events involving edges in $\mathcal{E}$ (seen as edges of $\GD^p$) is equal to the formula defining $\mathcal{P}_{\mathcal{E}}$. As a consequence, $\mathcal{P}_{\mathcal{E}}$ is a probability measure on $\mathcal{F}_\mathcal{E}$.

The fact that Kolmogorov's consistency relations are satisfied by the collection of probability measures $(\mathcal{P}_\mathcal{E})_{\mathcal{E}}$ is immediate once we notice that any finite number of these probability measures can be interpreted as finite-dimensional marginals of a Gibbs measure on dimer configurations of a periodic graph with a large enough fundamental domain. The measure $\P$ of Theorem \ref{thm:measure} is thus the one given by Kolmogorov's extension theorem applied to the collection $\left( \mathcal{P}_{\mathcal{E}} \right)_{\mathcal{E}}$. The Gibbs property follows from the Gibbs property of the probability measures $\P^p$.
\end{proof}

The probability of single edges are computed in details in Appendix~\ref{app:calculs}, using the explicit form for $K^{-1}$ given by Theorem~\ref{inverse}. 
Consider a rhombus of $\GR$ with vertices $\xb,\mathbf{t},\yb,\mathbf{u}$: $\xb$ and $\yb$ are vertices of $G$; $\mathbf{t}$ and $\mathbf{u}$ are vertices of $G^*$. Let $\theta$ be the half-angle of this rhombus, measured at $\xb$.
In $\GD$, we have that if $v=v_k(\xb)$ and $v'=v_\ell(\yb)$ are end points of the  edge $e$ coming from $\xb\yb$ in the decoration process, then we have
\begin{equation}
  \P(e) =  K_{v,v'} K^{-1}_{v',v}  =\frac{1}{2}+\frac{\pi-2\theta}{2\pi \cos\theta}.
  \label{eq:prob_vkvl}
\end{equation}
In Fisher's correspondence, the presence of the edge $e$ in the dimer configuration corresponds to the fact that $\xb\yb$ is  not covered by a piece of contour. In the low temperature expansion of the Ising model, this thus means that the two spins at $\mathbf{t}$ and $\mathbf{u}$ have the same sign. Since they can be both $+$ or $-$, the probability that they are both $+$, as a function of $\phi=\frac{\pi}{2}-\theta$, the angle of the rhombus measured at $\mathbf{u}$, is:
\begin{equation*}
  \P\left(\begin{array}{c}\includegraphics[width=20mm]{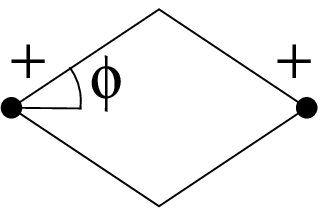}\end{array}\right)= \frac{1}{2}\left( \frac{1}{2} + \frac{\pi-2\theta}{2\pi\cos\theta} \right) = \frac{1}{4} + \frac{\phi}{2\pi\sin\phi}.
\end{equation*}

\subsection{Free energy of the critical dimer model}\label{sec:free_energy}

In this section, we suppose that the isoradial graph $G$ is periodic, so that the corresponding Fisher graph $\GD$ is. A natural exhaustion of $\GD$ by toroidal graphs is given by $\{\GD_n\}_{n\geq 1}$, where $\GD_n=\GD/n\ZZ^2$, and similarly for $G$. The graphs $\GD_1$ and $G_1$ are known as the {\em fundamental domains} of $\GD$ and $G$ respectively. In order to shorten notations in Theorem \ref{thm:free_energy} below and in the proof, we write $|E_n|~=~|E(G_n)|$, and $|V_n|~=~|V(G_n)|$.

The {\em free energy per fundamental domain} of the critical dimer model on the Fisher graph $\GD$ is denoted $f_D$, and is defined by:
\begin{equation*}
f_D=-\lim_{n\rightarrow\infty}\frac{1}{n^2}\log \Z_n^\nu,
\end{equation*}
where $\Z_n^\nu$ is the dimer partition function of the graph $\GD_n$, whose edges are assigned the critical weight function $\nu$.

\begin{thm}\label{thm:free_energy}
The free energy per fundamental domain of the critical dimer model on the Fisher graph $\GD$ is given by:
\begin{align*}
f_D=-(|E_1|+&|V_1|)\frac{\log 2}{2}+\\
&+\sum_{e\in E(G_1)}\left[\frac{\pi-2\theta_e}{2\pi}\log \tan\theta_e-\frac{1}{2}\log\cot\frac{\theta_e}{2} -\frac{1}{\pi}\left(L(\theta_e)+L\left(\frac{\pi}{2}-\theta_e\right)\right)\right],
\end{align*}
where $L$ is the Lobachevsky function, $L(x)=-\int_{0}^{x}\log 2\sin t\,\ud t$.
\end{thm}

\begin{proof}
For the proof, we follow the argument given by Kenyon~\cite{Kenyon3} to compute the normalized determinant of the Laplacian and the Dirac operator on isoradial graphs. See also \cite{Bea} for a detailed computation of the free energy of dimer models on isoradial bipartite graphs. 

In \cite{isoising1}, we proved that 
\begin{equation*}
  f_D = -\frac{1}{2} \iint_{\TT^2}\log \det \hat{K}(z,w) \frac{\ud z}{2i\pi z}\frac{\ud w}{2i\pi w},
\end{equation*}
where $\hat{K}(z,w)$ is the Fourier transform of the infinite periodic Kasteleyn matrix of the critical dimer model on $\mathcal{G}$.
 
We need the following definition of \cite{KeSchlenk}. Recall that $G^{\diamond}$ is the diamond graph associated to the isoradial graph $G$. A {\em train-track} of $G^{\diamond}$ is a path of
edge-adjacent rhombi of $G^{\diamond}$, which does not turn:
on entering a face, it exits along the opposite edge. As a consequence, each
rhombus in a train-track has an edge parallel to a fixed unit vector. We assume that each
train-track is extended as far as possible in both direction, so that it is a
bi-infinite path. 
 
The idea of the proof is to first understand how the free energy $f_D$ is changed when the isoradial embedding of the graph $G$ is modified by tilting a family of parallel train-tracks. If we can compute the free energy in an extreme situation, that is when all the rhombi are flat, \emph{i.e} with half-angles equal to $0$ or $\pi/2$, then we can compute the free energy of the initial graph by integrating the variation of the free energy along the deformation from the trivial flat embedding back to the initial graph, by tilting all the families of train-tracks successively until all rhombi recover their original shape.
Let $G^\text{flat}$ be the trivial flat isoradial graph, and let $\GD^\text{flat}$ be the corresponding Fisher graph.

Consider a train-track $T$ of $\GR$. All the rhombi of the train-track have a common parallel, with  an angle equal to $\alpha$. Let us compute the derivative of $f_D$ as we tilt the train-track $T$ as well as all its copies.
The function  $\log\det\hat{K}(z,w)$ is integrable over the unit torus, is a differentiable function of $\alpha$ for all $(z,w)\in \TT^2 \setminus \{1,1\}$, and the derivative is also uniformly integrable on the torus. The dimer free energy $f_D$ is thus differentiable with respect to $\alpha$, and
\begin{align*}
    \frac{\ud f_D}{\ud \alpha}&=-\frac{1}{2} \iint_{\TT^2}\frac{\ud}{\ud \alpha}\log\det\hat{K}(z,w) \frac{\ud z}{2i\pi z}\frac{\ud w}{2i\pi w}\\
    &= -\frac{1}{2}\sum_{u,v\in V(\mathcal{G}_1)} \iint_{\TT^2} \hat{K}^{-1}(z,w)_{v,u}\frac{\ud \hat{K}(z,w)_{u,v}}{\ud \alpha} \frac{\ud  z}{2i\pi z}\frac{\ud w}{2i\pi w}\\
&= - \sum_{e=uv\in E(\GD_1)} K^{-1}_{v,u}\frac{\ud K_{u,v}}{\ud \alpha} =  -\sum_{e\in E(G_1)} \mathcal{P}(e) \frac{\ud \log \nu_e}{\ud \alpha}.
\end{align*}
In the last line, we used the uniqueness of the inverse Kasteleyn matrix whose coefficients decrease at infinity, given by Proposition 5 of \cite{isoising1}. We also used the fact that if $e=uv$, then $\nu_e = |K_{u,v}|$ and  $\mathcal{P}(e)= K_{u,v} K^{-1}_{v,u}$. Note that the sum is restricted to edges coming from $E(G_1)$, because they are the only edges of $\GD_1$ with a weight depending on an angle.

There is in this sum a term for every rhombus in the fundamental domain $G_1$, and that term only depends on the half-angle of this rhombus.
The variation of the free energy of the dimer model along the deformation from $\GD^{\text{flat}}$ to $\GD$ is thus, up to an additive constant, the sum over all rhombi in a fundamental domain of the variation of a function $\mathsf{f}(\theta)$ depending only on the geometry  of the rhombus:
\begin{equation*}
\Delta f_D = f_D(\GD) - f_D(\GD^{\text{flat}}) =  \sum_{e\in E(G_1)} \mathsf{f}(\theta_e) - \mathsf{f}(\theta_e^{\text{flat}}),
\end{equation*}
with 
\begin{equation*}
  \frac{\ud \mathsf{f}(\theta)}{\ud\theta} = - \P(\theta) \frac{\ud \log \nu(\theta)}{\ud \theta}, 
\end{equation*}
where $\P(\theta)$ is the probability of occurrence in a dimer configuration of the edge of $\GD$ coming from an edge of $G$, whose half rhombus-angle is $\theta$; and $\theta_e^{\text{flat}}$ is the angle, equal to $0$ or $\frac{\pi}{2}$, of the degenerate rhombus in $G^{\text{flat}}$ associated to the edge $e$.

Recalling that by definition of $\nu$, and by Equation \eqref{eq:prob_vkvl}:
\begin{equation*}
\nu(\theta)=\cot\frac{\theta}{2}, \quad \P(\theta)=\frac{1}{2}+\frac{\pi-2\theta}{2\pi \cos\theta}, 
\end{equation*}
and using the fact that $\frac{\ud}{\ud \theta} \log\cot \frac{\theta}{2}= -\frac{1}{\sin\theta}$, we deduce an explicit 
formula for the derivative of $\mathsf{f}(\theta)$:
\begin{align*}
\frac{\ud \mathsf{f}(\theta)}{\ud \theta}& =\left( \frac{1}{2}+\frac{\pi-2\theta}{2\pi\cos \theta} \right) \frac{1}{\sin\theta}=\frac{1}{2\sin\theta}+\frac{\pi-2\theta}{\pi\sin 2\theta}.
\end{align*}
Using integration by parts, a primitive of $\frac{\ud \mathsf{f}(\theta)}{\ud \theta}$ is given by:
\begin{align*}
\fs(\theta)
&=\frac{1}{2}\log\tan\frac{\theta}{2}+\frac{\pi-2\theta}{2\pi}\log\tan\theta-\frac{1}{\pi}\left(L(\theta)+L\left(\frac{\pi}{2}-\theta\right)\right).
\end{align*}
The problem is that  $\mathsf{f}(\theta)$ goes to $-\infty$ when $\theta$ approaches $0$, and thus the free energy diverges when the graph $G$ becomes flat. One can instead study $s_D$, the \emph{entropy} of the corresponding dimer model, defined by
\begin{equation}
  s_D = -f_D -\sum_{e\in E(\GD_1)} \mathcal{P}(e) \log \nu_e,
  \label{eq:def_entropy}
\end{equation}
which behaves well when the embedding degenerates, since it is insensible to gauge transformations, and thus provides a better measure of the disorder in the model. 

The sum in \eqref{eq:def_entropy} is over all edges in the fundamental domain $\GD_1$ of the Fisher graph $\GD$. But since edges of the decoration have weight $1$, only the edges coming from edges of $G$ contribute. Thus, as for $f_D$,
the variation of $s_D$ can be written as the sum of contributions $\mathsf{s}(\theta_e)$ of all rhombi in a fundamental domain $G_1$.
\begin{equation}
  \Delta s_D = s_D(\GD) - s_D(\GD^{\text{flat}}) = \sum_{e\in E(G_1)} \mathsf{s}(\theta_e) - \mathsf{s}(\theta_e^{\text{flat}}),
  \label{eq:var_entropy}
\end{equation}
where 
\begin{align*}
  \mathsf{s}&(\theta) = -\mathsf{f}(\theta) - \P(\theta)\log \nu(\theta) \\
&=-\frac{1}{2}\log\tan\frac{\theta}{2}-\frac{\pi-2\theta}{2\pi}\log\tan\theta+\frac{1}{\pi}\left(L(\theta)+L\left(\frac{\pi}{2}-\theta\right)\right)-\left(\frac{1}{2}+\frac{\pi-2\theta}{2\pi \cos\theta}\right)\log\cot\frac{\theta}{2}\\
&=\frac{\pi-2\theta}{2\pi}\left(\log\cot\theta-\frac{\log\cot\frac{\theta}{2}}{\cos\theta}\right)+\frac{1}{\pi}\left(L(\theta)+L\left(\frac{\pi}{2}-\theta\right)\right).
\end{align*}
Since $\displaystyle\lim_{\theta\rightarrow 0}\Bigl(\log\cot\theta-\frac{\log\cot\frac{\theta}{2}}{\cos\theta}\Bigr)=-\log 2$, and $L(0)=L\left(\frac{\pi}{2}\right)=0$, we deduce that the values of $\mathsf{s}(\theta^{\text{flat}})$ are given by the following limits:
\begin{equation*}
  \lim_{\theta\rightarrow 0} \mathsf{s}(\theta) =-\frac{1}{2}\log 2, \quad\text{and}\quad
  \lim_{\theta\rightarrow\frac{\pi}{2}}\mathsf{s}(\theta)=0.
\end{equation*}

Let us now evaluate $s_D(\GD^{\text{flat}})$. Since the sum of rhombus-angles around a vertex is $2\pi$, and since rhombus half-angles are equal to $0$ or $\pi/2$, there is in $G^{\text{flat}}$, around each vertex, exactly two rhombi with half-angle $\theta^{\text{flat}}$ equal to $\frac{\pi}{2}$.
Let us analyze the dimer model on $\GD_n^{\text{flat}}$: the ``long'' edges with weight equal to $\infty$ (corresponding to rhombus half-angles $0$) are present in a random dimer configuration with probability $1$. The configuration of the long edges is thus frozen. The ``short'' ones with weight equal to $1$ (corresponding to rhombus half-angles $\pi/2$) are present in a random dimer configuration with probability $\frac{1}{2}+\frac{1}{\pi}$.

As noted above, around every vertex of $G_n^{\text{flat}}$, there are exactly two ``short'' edges. This implies that ``short'' edges form a collection of $k$ disjoint cycles covering all vertices of the graph $G_n^{\text{flat}}$, for some positive integer $k$. Such a cycle cannot be trivial, because surrounding a face would require an infinite number of flat rhombi, but there is only a finite number of them in $G_n^{\text{flat}}$. Since these cycles are disjoint, their number is bounded by the number of edges crossing the ``boundary'' of $G_n^{\text{flat}}$. Therefore $k=O(n)=o(n^2)$. 

Since all ``long'' edges are taken in a random dimer configuration, this implies that, for every cycle of ``short'' ones, edges are either all present, or all absent. The logarithm of the number of configurations for these cycles grows slower than $O(n^2)$, and therefore does not  contribute to the entropy. The main contribution comes from the decorations, which have two configurations each:
\begin{equation*}
 s_D(\GD^{\text{flat}})=|V_1|\log 2.
\end{equation*}
The total number of ``short'' edges is $|V_n|$ (two halves per vertex of $G_n$), and thus the number of long ones is $|E_n|-|V_n|$. From \eqref{eq:var_entropy}, we deduce that:
\begin{multline*}
s_D=(|E_1|+|V_1|)\frac{\log 2}{2}+\\
+\sum_{e\in E(G_1)}\left[
\frac{\pi-2\theta_e}{2\pi}\left(\log\cot\theta_e-\frac{\log\cot\frac{\theta_e}{2}}{\cos\theta_e}\right)+\frac{1}{\pi}\left(L(\theta_e)+L\left(\frac{\pi}{2}-\theta_e\right)\right)\right],
\end{multline*}
and using \eqref{eq:def_entropy}, we deduce Theorem \ref{thm:free_energy}.
\end{proof}

The {\em free energy per fundamental domain} of the critical $Z$-invariant Ising model, denoted $f_I$, is defined by:
\begin{equation*}
f_I=-\lim_{n\rightarrow\infty}\frac{1}{n^2}\log Z_n^J, 
\end{equation*}
where $Z_n^J$ is the partition function of the critical $Z$-invariant Ising model on the toroidal graph $G_n$. In \cite{BaxterZ}, Baxter gives an explicit expression for the free energy of  $Z$-invariant Ising models (not only critical), by transforming the graph $G$ with star-triangle transformations to make it look like large pieces of $\ZZ^2$ glued together, and  making use of the celebrated computation of Onsager ~\cite{Onsager} on $\ZZ^2$.
As a corollary to Theorem~\ref{thm:free_energy}, we obtain an alternative proof of this formula at the critical point.

\begin{thm}[\rm\cite{BaxterZ}]
\begin{equation*}
f_I=-|V_1|\frac{\log 2}{2} -\sum_{e\in E(G_1)}\left[\frac{\theta_e}{\pi}\log\tan\theta_e+\frac{1}{\pi}\left(L(\theta_e)+L\left(\frac{\pi}{2}-\theta_e\right)\right)
\right].
\end{equation*}
\end{thm}
\begin{proof}
Using the high temperature expansion, and Fisher's correspondence, we have from Equation \eqref{eq:Z-fisher} that:
\begin{equation*}
Z_n^J=\Bigl(\prod_{e\in E(G_n)}\sinh J(\theta_e)\Bigr)\Z_n^\nu=\Bigl(\prod_{e\in E(G_1)}\sinh J(\theta_e)\Bigr)^{n^2}\Z_n^\nu. 
\end{equation*}
As a consequence, 
\begin{equation*}
 f_I=f_D-\sum_{e\in E(G_1)}\log\sinh J(\theta_e).
\end{equation*}
Moreover, 
\begin{equation*}
\sinh J(\theta)=\sinh \log\sqrt{\frac{1+\sin\theta}{\cos\theta}}=\sqrt{\frac{\tan\frac{\theta}{2}\tan\theta}{2}},
\end{equation*}
so that by Theorem \ref{thm:free_energy}:
\begin{align*}
f_I=&-(|E_1|+|V_1|)\frac{\log 2}{2}\\
&+\sum_{e\in E(G_1)}\left[\frac{\pi-2\theta_e}{2\pi}\log \tan\theta_e-\frac{1}{2}\log\cot\frac{\theta_e}{2} -\frac{1}{\pi}\left(L(\theta_e)+L\left(\frac{\pi}{2}-\theta_e\right)\right)\right]+\\
&+|E_1|\frac{\log 2}{2}-\sum_{e\in E(G_1)}\frac{1}{2}\left[\log\tan\frac{\theta_e}{2}+\log\tan\theta_e\right]\\
=&-|V_1|\frac{\log 2}{2} -\sum_{e\in E(G_1)}\left[\frac{\theta_e}{\pi}\log\tan\theta_e+\frac{1}{\pi}\left(L(\theta_e)+L\left(\frac{\pi}{2}-\theta_e\right)\right)
\right].
\end{align*}
\end{proof}

The {\em critical Laplacian matrix} on $G$ is defined in \cite{Kenyon3} by:
\begin{equation*}
\Delta_{u,v}=
\begin{cases}
\tan\theta_{uv}&\text{ if }u\sim v\\
-\sum_{u'\sim u}\tan\theta_{uu'}&\text{ if }u=v\\
0&\text{ otherwise}. 
\end{cases}
\end{equation*}
The {\em characteristic polynomial} of the critical Laplacian on $G$, denoted $P_{\Delta}(z,w)$, is defined by $P_{\Delta}(z,w)=\det \widehat{\Delta}(z,w)$, where $\widehat{\Delta}(z,w)$ is the Fourier transform of the critical Laplacian~$\Delta$. 

In a similar way, the {\em characteristic polynomial} of the critical dimer model on the Fisher graph $\GD$, denoted by $P(z,w)$ is defined by $P(z,w)=\det\widehat{K}(z,w)$, where $\widehat{K}(z,w)$ is the Fourier transform of the Kasteleyn matrix $K$ of the graph $\GD$ with critical weights.

In Theorem $8$ of \cite{isoising1}, we prove that both characteristic polynomials are equal up to a non zero multiplicative constant. Using Theorem \ref{thm:free_energy}, and \cite{Kenyon3}, we can now determine this constant explicitly.

\begin{cor}\label{cor:cst}
The characteristic polynomials of the critical dimer model on the Fisher graph $\GD$, and of the critical Laplacian on the graph $G$, are related by the following explicit multiplicative constant.
\begin{equation*}
  \forall(z,w)\in\CC^{2},\quad \quad  P(z,w)= 2^{|V_1|}\prod_{e\in E(G_1)}\Bigl(\cot^2\frac{\theta_e}{2}-1\Bigr) P_\Delta(z,w).
\end{equation*}
\end{cor}
\begin{proof}
Let $c\neq 0$, be the constant of proportionality given by Theorem $8$ of \cite{isoising1}:
\begin{equation*}
P(z,w)=cP_\Delta(z,w). 
\end{equation*}
Taking logarithm on both sides, and integrating with respect to $(z,w)$ over the unit torus  yields:
\begin{equation}\label{eq:const_c}
\log c=-2 f_D-\iint_{\TT^2}\log P_\Delta(z,w)\frac{\ud z}{2i\pi z} \frac{\ud w}{2i\pi w}. 
\end{equation}
Moreover, by \cite{Kenyon3} we have:
\begin{equation*}
\iint_{\TT^2}\log P_\Delta(z,w)\frac{\ud z}{2i\pi z} \frac{\ud w}{2i\pi w}=
2\sum_{e\in E(G_1)}\left[\frac{\theta_e}{\pi}\log\tan\theta_e+\frac{1}{\pi}\left(L(\theta_e)+L\left(\frac{\pi}{2}-\theta_e\right)\right)
\right].
\end{equation*}
Plugging this and Theorem \ref{thm:free_energy} into \eqref{eq:const_c}, we obtain:
\begin{align*}
\log c &=(|E_1|+|V_1|)\log 2+\sum_{e\in E(G_1)}\left(\log\cot\frac{\theta_e}{2}+\log\cot\theta_e\right)\\
&=\log\Bigl(2^{|E_1|+|V_1|}\prod_{e\in E(G_1)}\cot\frac{\theta_e}{2}\cot\theta_e\Bigr)\\
&=\log \Bigl(2^{|V_1|}\prod_{e\in E(G_1)}\Bigl(\cot^2\frac{\theta_e}{2}-1\Bigr)\Bigr).
\end{align*}
\end{proof}

\section{Proof of Theorem \ref{inverse}}\label{sec:proof}

Let us recall the setting: $G$ is an infinite isoradial graph, and $\GD$ is the corresponding Fisher graph whose edges are assigned the dimer critical weight function; $K$ denotes the infinite Kasteleyn matrix of the critical dimer model on the graph $\GD$, and defines an operator acting on functions of the vertices of $\GD$. We now prove Theorem \ref{inverse}, {\em i.e.} we show that $KK^{-1}=\Id$, where $K^{-1}$ is given by:
\begin{equation*}
\forall x,\,y\in V(\GD),\quad
K^{-1}_{x,y}=\frac{1}{4\pi^2}\oint_{\C_{x,y}}f_x(\lambda)f_y(-\lambda)\expo_{\xb,\yb}(\lambda)\log\lambda\ud\lambda+C_{x,y}.
\end{equation*}

This section is organized as follows. In Section \ref{subsec91}, we prove that the function, $f_x(\lambda)\expo_{\xb,\yb}(\lambda),$
seen as a function of $x\in V(\GD)$, is in the kernel of the Kasteleyn operator $K$. Then, in Section \ref{subsec72}, we give the general idea of the argument, inspired from \cite{Kenyon3}, used to prove $KK^{-1}=\Id$. This motivates the delicate part of the proof. Indeed, for this argument to run through, it is not enough to have the contour of integration $\C_{x,y}$ to be defined as a simple closed curve containing all poles of the integrand, and avoiding {\em the half-line $d_{x,y}$}; we need it to {\em avoid an angular sector $s_{x,y}$, which contains the half-line $d_{x,y}$ and avoids all poles of the integrand}. These angular sectors are then defined in Section \ref{subsec73}. The delicate part of the proof, which strongly relies on the definition of the angular sectors, is given in Section \ref{sec9}.

\subsection{Kernel of $K$}\label{subsec91}

In this section we prove that the function $f_x(\lambda)\expo_{\xb,\yb}(\lambda)$,
seen as a function of $x~\in~V(\GD)$, is in the kernel of the Kasteleyn operator $K$. In other words,

\begin{prop}\label{prop:kernelK}
Let $x$, $y$ be two vertices of $\GD$, and let $x_1,\,x_2,\,x_3$ be the three neighbors of $x$ in $\GD$, then for every $\lambda\in\CC$:
\begin{equation*}
\sum_{i=1}^3 K_{x,x_i} f_{x_i}(\lambda)\expo_{\xb_i,\yb}(\lambda) =0.
\end{equation*}
\end{prop}

\begin{proof}
There are three cases to consider, depending on whether the vertex $x$ is of type `$w$', `$z$' or `$v$'.

\underline{\em If $x=w_k(\xb)$ for some $k\in\{1,\dots,d(\xb)\}$}. Then
the three neighbors of $x$ are $x_1=z_k(\xb)$, $x_2=z_{k+1}(\xb)$ and
$x_3=v_k(\xb)$. Since $x,x_1,x_2,x_3$ all belong to the same decoration, we omit the
argument $\xb$. By our
choice of Kasteleyn orientation, we have:
$$
K_{w_k,z_k}=-1,\quad  K_{w_k,z_{k+1}}=\eps_{w_k,z_{k+1}},\quad K_{w_k,v_k}=1.
$$
By definition of the angles in $\RR/4\pi\ZZ$ associated to vertices of $\GD$, see
\eqref{eq:angle_intra_deco}, we have:
$$
\alpha_{z_{k+1}}=\left\{
\begin{array}{ll}
\alpha_{w_k} &\mbox{ if }\eps_{w_k,z_{k+1}}=1\\
\alpha_{w_k}+2\pi &\mbox{ if }\eps_{w_k,z_{k+1}}=-1.\\
\end{array}\right.
$$
Using the definition of the function $f$, we deduce that:
\begin{equation}\label{eq:kw01}
K_{w_k,z_{k+1}}f_{z_{k+1}}(\lambda) = -f_{w_k}(\lambda).
\end{equation}
As a consequence,
\begin{align}\label{eq:Kw}
\sum_{i=1}^3 K_{x,x_i} f_{x_i}(\lambda)\expo_{\xb_i,\yb}&(\lambda)=\nonumber\\
&= \left[K_{w_k,z_k}f_{z_k}(\lambda)+K_{w_k,z_{k+1}}f_{z_{k+1}}(\lambda)+K_{w_k,v_k}f_{v_k}(\lambda)\right]\expo_{\xb,\yb}(\lambda)\nonumber\\
&=\left[ -f_{z_k}(\lambda) -f_{w_k}(\lambda) + f_{v_k}(\lambda)\right]\expo_{\xb,\yb}(\lambda).
\end{align}
Since by definition we have $f_{v_k}(\lambda)=f_{w_k}(\lambda)+f_{z_k}(\lambda)$, we deduce that \eqref{eq:Kw} is equal to~$0$.

\underline{\em If $x=z_k(\xb)$ for some $k\in\{1,\dots,d(\xb)\}$}. Then
the three neighbors of $x$ are $x_1=w_{k-1}(\xb)$, $x_2=w_k(\xb)$ and
$x_3=v_k(\xb)$. By our
choice of Kasteleyn orientation, we have:
$$
K_{z_k,w_{k-1}}=\eps_{z_k,w_{k-1}},\quad  K_{z_k,w_k}=1,\quad K_{z_k,v_k}=-1.
$$
Similarly to the case where $x=w_k$, we have:
\begin{equation*}
K_{z_k,w_{k-1}}f_{w_{k-1}}(\lambda) =f_{z_k}(\lambda).
\end{equation*}
As a consequence,
\begin{align}\label{eq:Kz}
\sum_{i=1}^3 K_{x,x_i} f_{x_i}(\lambda)\expo_{\xb_i,\yb}&(\lambda)=\nonumber\\
&= \left[K_{z_k,w_{k-1}}f_{w_{k-1}}(\lambda)+K_{z_k,w_k}f_{w_k}(\lambda)+K_{z_k,v_k}f_{v_k}(\lambda)\right]\expo_{\xb,\yb}(\lambda)\nonumber\\
&=\left[ f_{z_k}(\lambda) +f_{w_k}(\lambda) - f_{v_k}(\lambda)\right]\expo_{\xb,\yb}(\lambda).
\end{align}
Again, we deduce that \eqref{eq:Kz} is equal to $0$.

\underline{\em If $x=v_k(\xb)$ for some $k\in\{1,\dots,d(\xb)\}$}. Then the three
neighbors of $x$ are $x_1=z_k(\xb)$, $x_2=w_k(\xb)$ and $x_3=v_\l(\xb')$,
where $\l$ and $\xb'$ are such that $v_k(\xb)\sim v_\l(\xb')$. Although all vertices do not
belong to the same decoration, we omit the arguments $\xb$ and $\xb'$, knowing
that the index $k$ (resp. $\l$) refers to the vertex $\xb$ (resp. $\xb'$). By our choice of Kasteleyn orientation, we have:
\begin{equation*}
K_{v_k,z_k}=1,\quad K_{v_k,w_k}=-1,\quad 
K_{v_k,v_\l}= \eps_{v_k,v_\l}\cot\left(\frac{\scriptstyle \alpha_{w_k}-\alpha_{z_k}}{\scriptstyle4}\right).
\end{equation*}

Using the fact that
$\expo_{\xb',\yb}(\lambda)=\expo_{\xb',\xb}(\lambda)\expo_{\xb,\yb}(\lambda)$,
we deduce:
\begin{align*}
\sum_{i=1}^3 K_{x,x_i} f_{x_i}(\lambda)&\expo_{\xb_i,\yb}(\lambda)=\nonumber\\
&=\left[K_{v_k,z_k}f_{z_k}(\lambda)+K_{v_k,w_k}f_{w_k}(\lambda)
+K_{v_k,v_\l}\expo_{\xb',\xb}(\lambda)f_{v_\l}(\lambda)\right]\expo_{\xb,\yb}(\lambda)\nonumber\\
&=\left[
f_{z_k}(\lambda)-f_{w_k}(\lambda)+
\eps_{v_k,v_\l}\cot\left(\frac{\scriptstyle
    \alpha_{w_k}-\alpha_{z_k}}{\scriptstyle4}\right)\expo_{\xb',\xb}(\lambda)f_{v_\l}(\lambda)
\right]\expo_{\xb,\yb}(\lambda).
\end{align*}

Moreover, by definition of the angles in $\RR/4\pi\ZZ$
associated to vertices of $\GD$, see \eqref{eq:angle_chg_deco}, we have:
\begin{equation*}
\alpha_{w_\l}=\alpha_{w_k}-\eps_{v_k,v_\l}\pi,\quad \alpha_{z_\l}=\alpha_{z_k}-\eps_{v_k,v_\l}\pi,
\end{equation*}
and using the definition of the function $f$, we deduce:
\begin{equation*}
f_{v_\l}(\lambda)=\left(\frac{e^{i\frac{\alpha_{w_\l}}{2}}}
  {e^{i\alpha_{w_\l}}-\lambda}-\frac{e^{i\frac{\alpha_{z_\l}}{2}}}
  {e^{i\alpha_{z_\l}}-\lambda}\right)
=i\eps_{v_k,v_\l}
\left(\frac{e^{i\frac{\alpha_{w_k}}{2}}} {e^{i\alpha_{w_k}}+\lambda}-\frac{e^{i\frac{\alpha_{z_k}}{2}}} {e^{i\alpha_{z_k}}+\lambda}\right),
\end{equation*}
so that:
\begin{align}\label{eq:Kv0}
K_{v_k,v_\l}\expo_{\xb',\xb}&(\lambda)f_{v_\l}(\lambda)=\nonumber\\
&=i\cot\left(\frac{\scriptstyle\alpha_{w_k}-\alpha_{z_k}}{\scriptstyle 4}\right)\frac{\bigl(e^{i\alpha_{w_k}}+\lambda\bigr)\bigl(e^{i\alpha_{z_k}}+\lambda\bigr)}{\bigl(e^{i\alpha_{w_k}}-\lambda\bigr)\bigl(e^{i\alpha_{z_k}}-\lambda\bigr)}\left(\frac{e^{i\frac{\alpha_{w_k}}{2}}} {e^{i\alpha_{w_k}}+\lambda}-\frac{e^{i\frac{\alpha_{z_k}}{2}}} {e^{i\alpha_{z_k}}+\lambda}\right)\nonumber\\
&=f_{w_k}(\lambda)-f_{z_k}(\lambda).
\end{align}
As a consequence,
\begin{equation}\label{eq:Kv}
\sum_{i=1}^3 K_{x,x_i} f_{x_i}(\lambda)\expo_{\xb_i,\yb}(\lambda)
=\left[
f_{z_k}(\lambda)-f_{w_k}(\lambda)+f_{w_k}(\lambda)-f_{z_k}(\lambda)\right]\expo_{\xb,\yb}(\lambda)=0.
\end{equation}
\end{proof}

\subsection{General idea of the argument}\label{subsec72}

The general argument used to prove Theorem \ref{inverse}, {\em i.e.} $KK^{-1}=\Id$, is inspired from \cite{Kenyon3}, where Kenyon computes a local explicit expression for the inverse of the Kasteleyn matrix of the critical dimer model on a bipartite, isoradial graph. It cannot be applied as such to our case of the critical dimer model on the Fisher graph $\GD$ of $G$, which is {\em not} isoradial, but it is nevertheless useful to sketch the main ideas.

Let $x,y$ be two vertices of $\GD$, and let us assume that the contour of integration $\C_{x,y}$ of the integral term of $K^{-1}_{x,y}$, is defined to be a simple closed curve oriented counterclockwise, containing all poles of the integrand and avoiding an angular sector $s_{x,y}$, which contains the half-line $d_{x,y}$, see Figure \ref{arg} below.

\begin{figure}[ht]
\centering
\includegraphics[width=4.5cm]{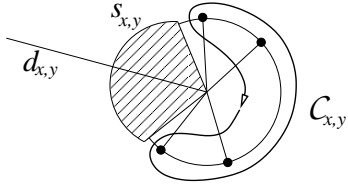}
\caption{\label{arg}
Definition of the contour of integration $\C_{x,y}$ of the integral term of $K^{-1}_{x,y}$. The poles of the integrand are the thick points.}
\end{figure}

The argument runs as follows.
Let $x_1,x_2,x_3$ be the three neighbors of $x$ in
$\GD$. When $x\neq y$, the goal is to show that the intersection of the three sectors
$\bigcap_{i=1}^3 s_{x_i,y}$ is non empty. Then, the three contours
$\C_{x_i,y}$ can be continuously deformed to a common contour $\C$ without
meeting any pole, and
\begin{multline}\label{eq:general1}
\sum_{i=1}^3  K_{x,x_i}\oint_{\C_{x_i, y}}f_{x_i}(\lambda)f_y(-\lambda)\expo_{\xb_i,\yb}(\lambda)\log\lambda\frac{\ud \lambda}{(2\pi)^2}= \\
=\oint_\C \left(\sum_{i=1}^3 K_{x,x_i}f_{x_i}(\lambda)\expo_{\xb_i,\yb}(\lambda)\right)f_y(-\lambda) \log\lambda\frac{\ud \lambda}{(2\pi)^2} =0,
\end{multline}
since by Proposition~\ref{prop:kernelK}, the sum in brackets is zero. Note that in \cite{Kenyon3}, Kenyon has a result similar to Proposition \ref{prop:kernelK}, giving functions that are in the kernel of the Kasteleyn matrix of the dimer model on a bipartite, isoradial graph.

When $x=y$, the goal is to show that the intersection of the three sectors is empty, and explicitly compute:
 \begin{equation}\label{eq:general2}
\sum_{i=1}^3  K_{x,x_i}\oint_{\C_{x_i, x}}f_{x_i}(\lambda)f_x(-\lambda)\expo_{\xb_i,\xb}(\lambda)\log\lambda\frac{\ud \lambda}{(2\pi)^2}=1.
\end{equation}
In our case things turn out to be more complicated, since we cannot define angular sectors $s_{x,y}$ so that \eqref{eq:general1} and \eqref{eq:general2} hold as such. We define them in such a way that when $x$ and $y$ do not belong to the same triangle of a decoration $\xb$, the intersection of the three sectors $\bigcap_{i=1}^3 s_{x_i,y}$ is non empty, and \eqref{eq:general1} holds. Then, in order for $(KK^{-1})_{x,y}$ to be equal to $\delta_{x,y}$ in all other cases, we need to have an additional constant $C_{x,y}$ in the expression of the inverse $K^{-1}_{x,y}$. Let us mention that angular sectors are not constructed explicitly in \cite{Kenyon3}, since geometric considerations suffice. The construction of Case $1$ of Section \ref{subsubsec823} could also be used in that setting.

\subsection{Definition of the angular sectors $s_{x,y}$}\label{subsec73}

Let $x,y$ be two vertices of $\GD$. In this section, we construct the angular sector $s_{x,y}$ containing the half-line $d_{x,y}$ and avoiding all poles of the integrand of $K^{-1}_{x,y}$. In order to do this, we first recall some facts about
isoradial graphs, then we encode the poles of the integrand of $K^{-1}_{x,y}$ in an edge-path $\gamma_{x,y}$ of the diamond graph $\GR$. Using this path $\gamma_{x,y}$, we define the angular sector $s_{x,y}$.

\subsubsection{Train-tracks and minimal paths}\label{subsubsec821}

Let $G$ be an infinite isoradial graph, and $G^{\diamond}$ be the associated
diamond graph. Let us first recall the following definition of \cite{KeSchlenk}. A {\em train-track} of $G^{\diamond}$ is a path of
edge-adjacent rhombi of $G^{\diamond}$, which does not turn:
on entering a face, it exits along the opposite edge. As a consequence, each
rhombus in a train-track has an edge parallel to a fixed unit vector, called
the {\em direction} of the train-track. We assume that each
train-track is extended as far as possible in both direction, so that it is a
bi-infinite path. By \cite{KeSchlenk}, we know that a train-track can not intersect
itself, and that two distinct train-tracks cross each other at most once.

Let $\xb,\yb$ be two vertices of $G$, we now define the notion of {\em minimal
  path} of $G^{\diamond}$ from $\xb$ to $\yb$. We say that a train-track of $G^{\diamond}$ {\em
  separates} $\xb$ from $\yb$, if when deleting it, $\xb$ and $\yb$ are in two
distinct connected components. Now, observe that each edge of $G^{\diamond}$
belongs to a unique train-track whose direction is given by that edge, it is called
the {\em train-track associated to the edge}. An edge-path $\gamma$ of
$G^{\diamond}$, from $\xb$ to $\yb$, is called {\em minimal}, if all the
train-tracks associated to edges
of $\gamma$ separate $\xb$ from $\yb$, and all those train-tracks are distinct. In
general there is not uniqueness of the minimal path, but all minimal paths
consist of the same steps, taken in a different order.

\subsubsection{Encoding the poles of the integrand of $K^{-1}_{x,y}$}\label{subsubsec822}

Let $x,y$ be two vertices of $\GD$, and let $\xb,\yb$ be the corresponding vertices of $G$. In this section, we define an edge-path $\gamma_{x,y}$ of the diamond graph $\GR$ encoding the poles of the integrand $f_x(\lambda)f_y(-\lambda)\expo_{\xb,\yb}(\lambda)\log\lambda$ of $K^{-1}_{x,y}$.

By definition of the exponential function and of a minimal path, the poles of
$\expo_{\xb,\yb}(\lambda)$ are encoded in the steps of a minimal path of $\GR$, oriented from
$\yb$ to $\xb$. It would be natural to see the poles of the integrand of $K^{-1}_{x,y}$,
as representing the steps of a path of $\GR$ passing through $\yb$ and $\xb$,
obtained by adding the steps corresponding to the poles of $f_x(\lambda)$ and
$f_y(-\lambda)$. However the situation is a bit trickier due to possible
cancellations of these poles with factors
appearing in the numerator of $\expo_{\xb,\yb}(\lambda)$. This is taken into account in the following way.

By definition, the function $f_x(\lambda)$ has either $1$ or $2$ poles. Let us suppose it has $2$, denoted by $e^{i\alpha_1}$, $e^{i\alpha_2}$ (the case where it has $1$ is similar and easier to handle), denote by $T_x^1$, $T_x^2$ the two associated train-tracks, and let
$T_x=\{T_x^1,T_x^2\}$. Similarly, let us suppose that $f_y(-\lambda)$ has $2$ poles $-e^{i\alpha_1'}$, $-e^{i\alpha_2'}$, defining two train-tracks $T_y^1$, $T_y^2$, and let $T_y=\{T_y^1,T_y^2\}$.

Let $\gamma_{x,y}$ be a minimal path from $\yb$ to $\xb$. If $T_x^1$ separates $\yb$ from $\xb$, then the pole $e^{i\alpha_1}$ is cancelled by the exponential, and we leave
$\gamma_{x,y}$ unchanged. If not, this pole remains, and we extend $\gamma_{x,y}$
by adding the step $e^{i\alpha_1}$ so that the initial vertex of $e^{i\alpha_1}$ is incident to $\xb$, let $\hat{\xb}$ be the ending vertex of $e^{i\alpha_1}$. This yields a new path of $\GR$, also denoted $\gamma_{x,y}$, from $\yb$ to $\hat{\xb}$, passing through $\yb$ and $\xb$. We then do the same procedure for the pole $e^{i\alpha_2}$.

When dealing with a pole of $f_y(-\lambda)$, one needs to be careful since, even when the corresponding train-track separates $\yb$ from $\xb$, the exponential function might not cancel the pole if it has already canceled the same pole of $f_x(\lambda)$. This happens when $T_x$ and $T_y$ have a common train-track. The procedure to extend $\gamma_{x,y}$ runs as follows. If $T_y^1$ separates $\yb$ from $\xb$, and $T_y^1\cap T_x=\emptyset$, then the pole $-e^{i\alpha_1'}$ is canceled by the exponential function, and we leave $\gamma_{x,y}$ unchanged. If not this pole remains, and we extend $\gamma_{x,y}$ by adding the step $-e^{i\alpha_1'}$ so that the ending vertex of $-e^{i\alpha_1'}$ is $\yb$, let $\hat{\yb}$ be the initial vertex of $-e^{i\alpha_1'}$. The same procedure is done for $-e^{i\alpha_2'}$, yielding a new path $\gamma_{x,y}$ of $\GR$, from $\hat{\yb}$ to $\hat{\xb}$, passing through $\yb$ and $\xb$, and encoding the poles of the integrand of $K^{-1}_{x,y}$.

\subsubsection{Obtaining an angular sector $s_{x,y}$ from the path $\gamma_{x,y}$}\label{subsubsec823}

Let $x,y$ be two vertices of $\GD$. In this section, we construct an angular sector $s_{x,y}$,
containing the half-line $d_{x,y}$, and avoiding all poles of the integrand of $K^{-1}_{x,y}$, from the path $\gamma_{x,y}$ encoding these poles. The main step consists in constructing a convex polygon $P_{x,y}$ from $\gamma_{x,y}$. The angular sector $s_{x,y}$ is then naturally defined from $P_{x,y}$.

In order to construct the convex polygon, let us recall the following definition of \cite{Mercat1} and lemma of \cite{Bea1}. A finite, simply connected sub-graph of $\GR$ is called {\em train-track-convex}, if every train-track of $\GR$ which intersects this sub-graph crosses its boundary twice exactly.
\begin{lem}[\cite{Bea1}]\label{lem:bea}
Any finite, simply connected, train-track convex sub-graph $Q$ of $\GR$ can be completed by a finite number of rhombi in order to become a convex polygon $P$ whose opposite sides are parallel.
\end{lem}
The construction of the convex polygon $P$ from $Q$ is an explicit algorithm which consists in adding a rhombus along the boundary each time the polygon is not convex, see Figures \ref{fig:insertion_algo} and \ref{fig:insertion_algo_1} for examples. The proof consists in showing that it is possible to add a rhombus at every step, and that the algorithm ends in a finite number of steps. As a consequence of the construction, we have the following facts:
\begin{enumerate}
\item The train-tracks of $\GR$ associated to boundary edges of $Q$, and those of
  $P$ associated to boundary edges of $Q$ have the same parallel direction.
\item Each train-track corresponding to an edge of $P$ crosses the boundary of
  $P$ twice exactly.
\item The boundary of the convex polygon $P$ is independent of the order in which the rhombi are added.
\end{enumerate}
We now construct a simply-connected, train-track convex subgraph $Q_{x,y}$ from the path $\gamma_{x,y}$, encoding the poles of the integrand of $K^{-1}_{x,y}$. Applying Lemma \ref{lem:bea} yields a convex polygon $P_{x,y}$, from which we define the angular sector $s_{x,y}$. 

Recall that $T_x$ (resp. $T_y$) consists of the one/two train-track(s) associated to poles of $f_x(\lambda)$ (resp. $f_y(-\lambda)$). There are three cases to consider, depending on the intersection properties of $T_x$ and $T_y$. 

\underline{Case $1$: $T_x$ and $T_y$ have $0$ train-track in common}.
Then, by construction, the path $\gamma_{x,y}$ of Section \ref{subsubsec822} is a minimal path of $\GR$ from $\hat{\yb}$ to $\hat{\xb}$, passing through $\yb$ and $\xb$. As such, it is a simply-connected, train-track convex sub-graph of $\GR$. Indeed, $\gamma_{x,y}$ is a subgraph of $\GR$, whose interior is
the empty set, and whose boundary is $\gamma_{x,y}$ itself, consisting of
``doubled'' edges enclosing the empty interior. Since $\gamma_{x,y}$ is minimal,
all train-tracks associated to edges of $\gamma_{x,y}$ are
distinct, so that they all cross the path $\gamma_{x,y}$ once, i.e. the boundary of $\gamma_{x,y}$
twice. We define $Q_{x,y}$ to be $\gamma_{x,y}$, and apply Lemma \ref{lem:bea} in order to obtain a convex polygon $P_{x,y}$, see Figure \ref{fig:insertion_algo} (left), the numbers indicate the order in which the rhombi are added.

\begin{figure}[ht]
\centering
\includegraphics[width=10.5cm]{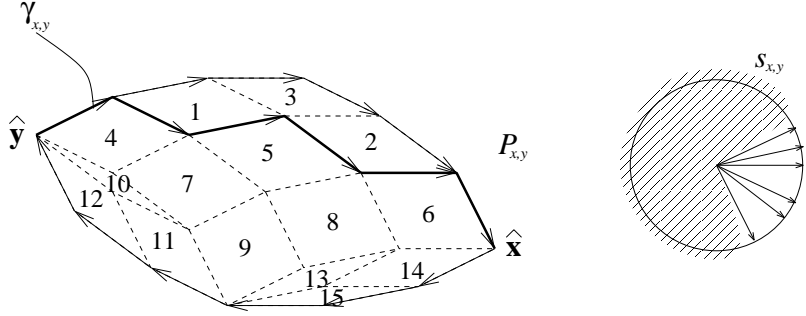}
\caption{\label{fig:insertion_algo}
Left: construction of the convex polygon $P_{x,y}$ from a minimal path
$\gamma_{x,y}$. Right: corresponding angular sector $s_{x,y}$.}
\end{figure}

Using $1$, $2$, $3$, and the fact that $\gamma_{x,y}$ is minimal, we  deduce that
the vertices $\hat{\yb}$ and $\hat{\xb}$ are on the boundary of
$P_{x,y}$, and that the boundary of $P_{x,y}$ is independent of the choice minimal path of $\GR$ from $\hat{\yb}$ to $\hat{\xb}$. Moreover, we deduce that each edge of the boundary of $P_{x,y}$ from $\hat{\yb}$ clockwise to $\hat{\xb}$ is parallel to a unique edge of $\gamma_{x,y}$. The same holds for
the boundary of $P_{x,y}$ from $\hat{\xb}$ clockwise to $\hat{\yb}$.

Let $e_1,\cdots,e_n$ be the edges on the boundary of $P_{x,y}$ from $\hat{\yb}$ to
$\hat{\xb}$, oriented clockwise. Then these edges encode the poles of the integrand 
$f_x(\lambda) f_y(-\lambda)\expo_{\xb,\yb}(\lambda)\log\lambda$ of $K^{-1}_{x,y}$. Since $P_{x,y}$ is convex, they are all included in an angular sector of size smaller than $\pi$. We define $s_{x,y}$ to be an angular sector strictly included in the complement of this sector, see Figure \ref{fig:insertion_algo} (right) where the sector $s_{x,y}$ is given in dashed lines.

\underline{Case $2$: $T_x$ and $T_y$ have $1$ train-track $T$ in common}. We specify the construction of the path $\gamma_{x,y}$ as follows. Let us first consider the poles of $f_x(\lambda)$ and $f_y(-\lambda)$  which are not in the common train-track $T$ (if any), then there is a minimal path from a vertex $\yb_0$ to a vertex $\xb_0$, passing through $\yb$ and $\xb$, encoding these poles and those of the exponential function. We then consider two cases: the common train-track $T$ may or may not separate $\yb$ from $\xb$.

Suppose that the common train-track $T$ does not separate $\yb$ from $\xb$, see also Figure \ref{fig:sametraintrack}. Then the vertices $\yb_0$, $\yb$, $\xb_0$, $\xb$ are on the same side of $T$, and the minimal path from $\yb_0$ to $\xb_0$ follows the boundary of the train-track $T$. Moreover the pole of $f_x(\lambda)$ corresponding to $T$ is
$e^{i\alpha}$, that of $f_y(-\lambda)$ is the opposite $-e^{i\alpha}$, and the exponential function cancels none of these poles. As a consequence the path $\gamma_{x,y}$ is
obtained by adding the vector $-e^{i\alpha}$ at the beginning of the path, and the vector $e^{i\alpha}$ at the end, see Figure \ref{fig:sametraintrack} (thick lines).

\begin{figure}[ht]
\centering
\includegraphics[width=8.5cm]{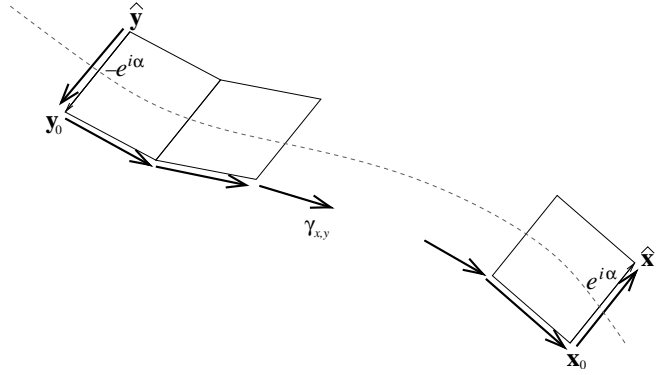}
\caption{\label{fig:sametraintrack}
The common train-track $T$ does not separate $\yb$
from $\xb$. The path $\gamma_{x,y}$ is given in thick line. }
\end{figure}

Suppose that the common train-track $T$ separates $\yb$ from $\xb$, see also Figure \ref{fig:sametraintrack1}. Then the vertices $\yb_0$, $\yb$ are on one side of the train-track $T$, and the vertices $\xb_0$, $\xb$ are on the other. The minimal path from $\yb_0$ to $\xb_0$ follows the boundary of the train-track $T$ for a number of steps (which can be $0$), then crosses it, and follows the boundary on the other side until $\xb_0$. Moreover, $f_x(\lambda)$ and $f_y(-\lambda)$ have the same pole $e^{i\alpha}$ corresponding to $T$. The exponential function cancels one of the two poles, say the one of $f_x(\lambda)$, so that the pole of $f_y(-\lambda)$ remains. As a consequence, the path $\gamma_{x,y}$ is obtained by adding the pole $e^{i\alpha}$ at the beginning of the minimal path from $\yb_0$ to $\xb_0$, and leaving $\xb_0$ unchanged. The path $\gamma_{x,y}$ thus crosses the train-track $T$ from $\hat{\yb}$ to $\yb_0$, follows the boundary for a number of steps, then crosses the train-track $T$ again, and follows the boundary until $\hat{\xb}=\xb_0$. Just by changing the order of the steps, we can choose $\gamma_{x,y}$ as in Figure \ref{fig:sametraintrack1} below (thick lines).

\begin{figure}[ht]
\centering
\includegraphics[width=8cm]{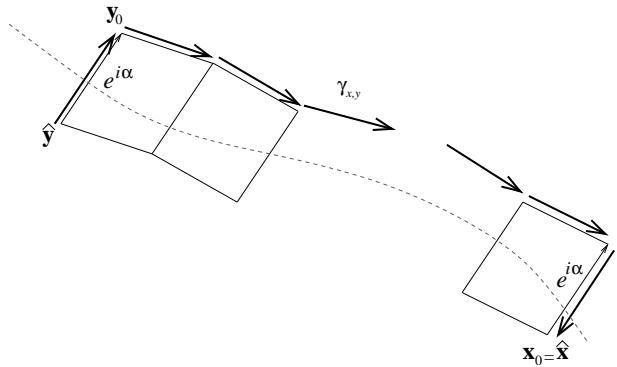}
\caption{\label{fig:sametraintrack1}
The common train-track $T$ separates $\yb$ from $\xb$. The path $\gamma_{x,y}$ is given in thick line.}
\end{figure}

Note that in both cases, the path $\gamma_{x,y}$ starts with a parallel direction of
the common train-track $T$ from $\hat{\yb}$ to $\yb_0$, then follows its boundary, and ends with the opposite parallel direction until $\hat{\xb}$. Define $Q_{x,y}$ to be the subgraph of $\GR$ made of the rhombi of the train-track $T$ that are bounded by $\gamma_{x,y}$. Then, $Q_{x,y}$ is train-track convex and we can apply Lemma \ref{lem:bea} in order to obtain a convex polygon $P_{x,y}$, see Figure \ref{fig:insertion_algo_1} below (left), the numbers indicate the order in which the rhombi are added.

\begin{figure}[ht]
\centering
\includegraphics[width=10cm]{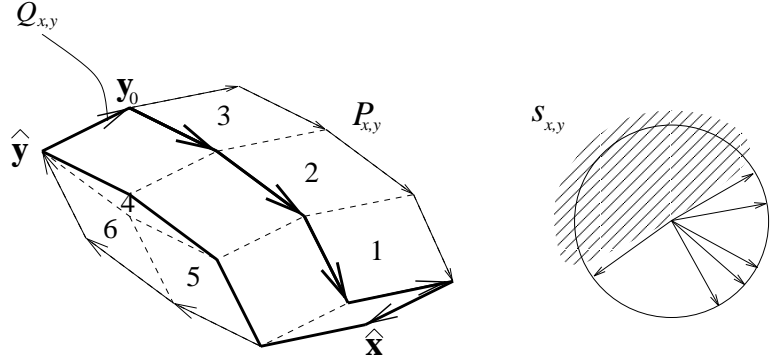}
\caption{\label{fig:insertion_algo_1}
Left: construction of the convex polygon $P_{x,y}$ from the train-track-convex subgraph $Q_{x,y}$. Right: corresponding angular sector $s_{x,y}$.}
\end{figure}

Using $1$, $2$, $3$, the fact that the part of $\gamma_{x,y}$ from $\yb_0$ to $\hat{\xb}$ is minimal, and the fact that the path $\gamma_{x,y}$ crosses the train-track $T$ twice, we deduce that the vertices $\hat{\yb}$, $\hat{\xb}$, $\yb_0$, $\xb_0$ are on the boundary of $P_{x,y}$. Moreover, we deduce that each edge on the boundary of $P_{x,y}$ from $\hat{\yb}$ to $\hat{\xb}$ passing through $\yb_0$ is parallel to a unique edge of $\gamma_{x,y}$. Let $e_1,\cdots,e_n$ denote this set of edges. Then, they encode the poles of the integrand $f_x(\lambda) f_y(-\lambda)\expo_{\xb,\yb}(\lambda)\log\lambda$ of $K^{-1}_{x,y}$, and define a sector of size $\pi$. We define $s_{x,y}$ to be an angular sector strictly included in the complement of this sector, see Figure \ref{fig:insertion_algo_1} (right) where the sector is given in dashed line.

Note that the above definition of angular sector $s_{x,y}$ is ambiguous when the path $\gamma_{x,y}$ crosses the train-track $T$ twice, without following its
boundary in between. This can only occur in the following four cases. Since all vertices involved belong to the same decoration, we omit the argument $\xb$ in the notations. Recall that $e^{i\alpha_{w_k}}=e^{i\alpha_{z_{k+1}}}$.
\begin{equation*}
(x,y)=
\begin{cases}
&(w_k,z_{k+1})\text{ then, $\gamma_{x,y}$ is the path $-e^{i\alpha_{z_{k+1}}},e^{i\alpha_{w_k}}$},\\
&(z_{k+1},w_k) \text{ then $\gamma_{x,y}$ is the path $-e^{i\alpha_{w_{k}}},e^{i\alpha_{z_{k+1}}}$},\\
&(w_k,w_k) \text{ then $\gamma_{x,y}$ is the path $-e^{i\alpha_{w_{k}}},e^{i\alpha_{w_{k}}}$},\\
&(z_{k+1},z_{k+1}) \text{ then $\gamma_{x,y}$ is the path $-e^{i\alpha_{z_{k+1}}},e^{i\alpha_{z_{k+1}}}$}.
\end{cases}
\end{equation*}
In these four cases, we set the conventions given by Figure \ref{fig6} below for the angular sector $s_{x,y}$. These are natural in the light of the proof of Theorem \ref{inverse}, see Section \ref{sec9}.\\

\begin{figure}[ht]
\begin{center}
\includegraphics[width=\linewidth]{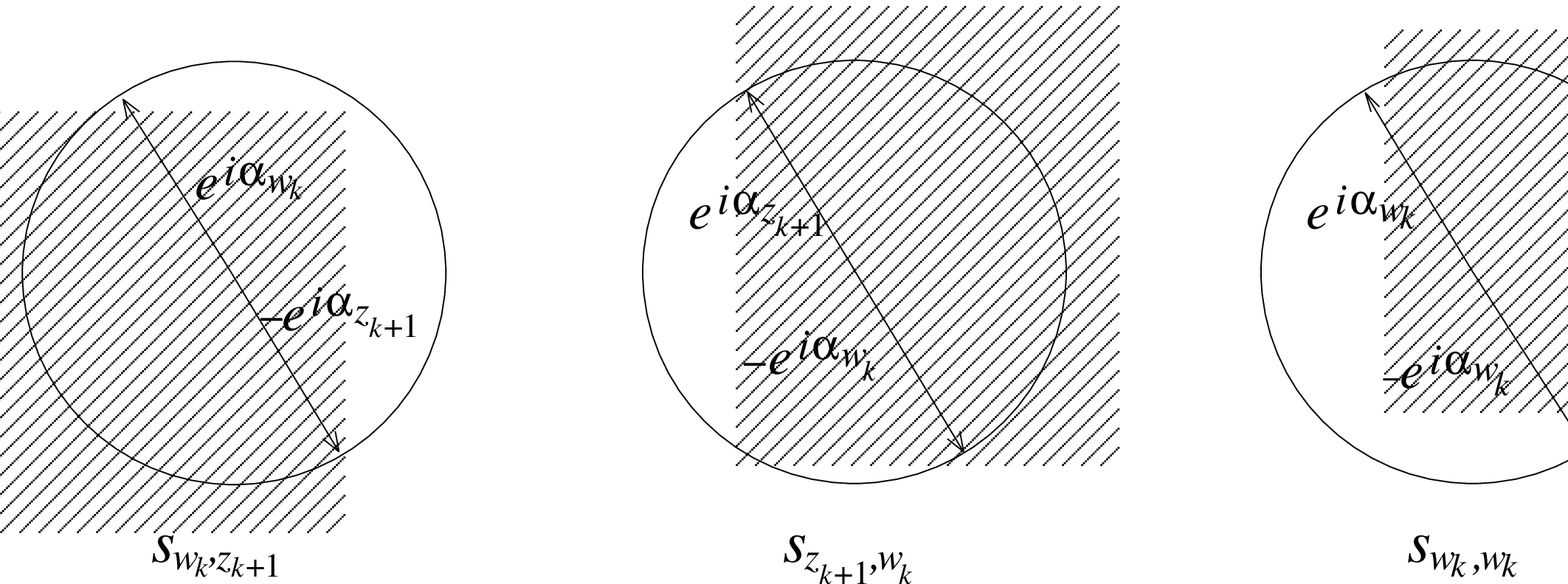}
\end{center}
\caption{\label{fig6} Definition of the angular sector $s_{x,y}$ (dashed line) in the four ambiguous instances of Case $2$: from left to right,
$(x,y)=(w_k,z_{k+1}),(z_{k+1},w_k), (w_k,w_k),(z_{k+},z_{k+1})$.}
\end{figure}

\underline{Case $3$: $T_x$ and $T_y$ have $2$ train-tracks in common}. This can only occur in two cases: either the vertices $x$ and $y$ are
neighbors in $\GD$ and the vertices $\xb$, $\yb$ are neighbors in $G$, meaning that $x=v_k(\xb)$ and $y=v_\l(\yb)$, with $k$ and $\l$ such that $v_k(\xb)\sim v_\l(\yb)$; or $x=y=v_k(\xb)$.

Suppose that $(x,y)=(v_k(\xb),v_\l(\yb))$. Then $f_{v_k(\xb)}(\lambda)$ and
$f_{v_\l(\yb)}(-\lambda)$ have the same poles $e^{i\alpha_{w_k(\xb)}}$,
$e^{i\alpha_{z_k(\xb)}}$. The exponential function $\expo_{\xb,\yb}(\lambda)$, cancels one of the pairs of poles, and adds two new ones $-e^{i\alpha_{w_k(\xb)}}$, $-e^{i\alpha_{z_k(\xb)}}$. Thus the path $\gamma_{x,y}$ follows the boundary of the rhombus associated to the edge $\xb\yb$ (note that this rhombus is already our convex polygon $P_{x,y}$). By convention, we choose the order of the edges of the path $\gamma_{x,y}$ as in Figure \ref{fig8} (left). This is natural in the light of the proof of Theorem \ref{inverse}, see Section \ref{sec9}. This defines a non-empty angular sector $s_{x,y}$, containing no pole of the integrand of $K^{-1}_{x,y}$, see Figure \ref{fig8} (right).

\begin{figure}[ht]
\begin{center}
\includegraphics[width=7cm]{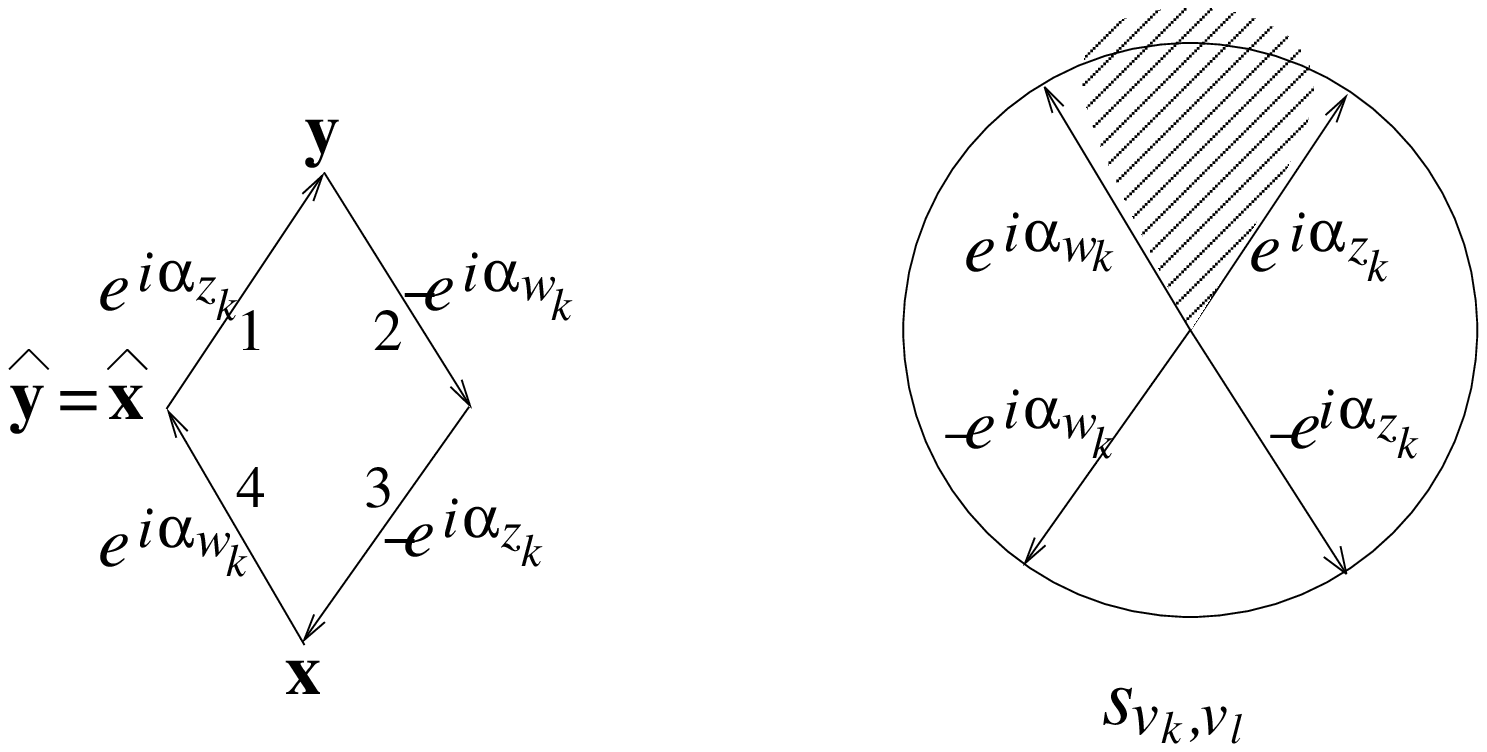}
\end{center}
\caption{\label{fig8} Left: Choice of path $\gamma_{v_k(\xb),v_\l(\yb)}$. Right: corresponding angular sector~$s_{v_k(\xb),v_\l(\yb)}$.}
\end{figure}

Suppose that $(x,y)=(v_k(\xb),v_k(\xb))$. Then $f_{v_k(\xb)}(\lambda)$ has poles $e^{i\alpha_{w_k(\xb)}}$,
$e^{i\alpha_{z_k(\xb)}}$, and $f_{v_k(\xb)}(-\lambda)$ has opposite poles. Moreover,
$\expo_{\xb,\yb}(\lambda)=1$, so that it cancels no pole. Thus the path $\gamma_{x,y}$ follows the boundary of the rhombus defined by the vectors $e^{i\alpha_{w_k(\xb)}}$ and $e^{i\alpha_{z_k(\xb)}}$. By convention, we choose the order of the edges of the path $\gamma_{x,y}$ as in Figure \ref{fig9} (left). This defines a non-empty angular sector $s_{x,y}$, containing no poles of the integrand of $K^{-1}_{x,y}$, see Figure \ref{fig9} (right).

\begin{figure}[ht]
\begin{center}
\includegraphics[width=7cm]{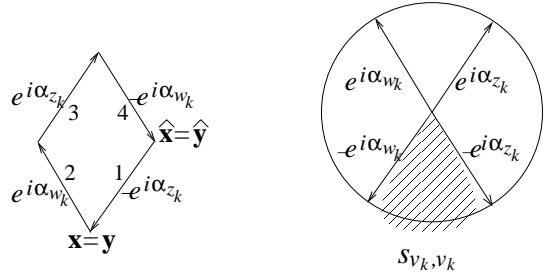}
\end{center}
\caption{\label{fig9} Left: Choice of path $\gamma_{v_k(\xb),v_k(\yb)}$. Right: corresponding angular sector~$s_{v_k(\xb),v_k(\xb)}$.}
\end{figure}

\begin{rem}\label{rem:rem1}
By the above definition, in all cases but the four exceptions to Case $2$ and Case $3$, the angular sector $s_{x,y}$ always contains the half-line $d_{x,y}$ starting from $0$, in the direction from $\hat{\xb}$ to $\hat{\yb}$. Thus, we can define the contour of integration $\C_{x,y}$ to be a closed contour oriented counterclockwise and avoiding the closed half-line $d_{x,y}$.
\end{rem}

\subsection{Proof of Theorem \ref{inverse}}\label{sec9}

In this section, we prove Theorem \ref{inverse}, {\em i.e.} $KK^{-1}=\Id$. Let $I$ and $C$ be the infinite matrices whose coefficients are the integral and constant part of $K^{-1}$ respectively:
\begin{align*}
I_{x,y}&=\frac{1}{(2\pi)^2} \oint_{\C_{x,y}}f_x(\lambda)f_y(-\lambda)\expo_{\xb,\yb}(\lambda)\log\lambda\ud\lambda,\\
C_{x,y}&=C_{x,y}.
\end{align*}
Our goal is to show that $(KK^{-1})_{x,y}=(KI)_{x,y}+(KC)_{x,y}=\delta_{x,y}$.
Let $x_1,\,x_2,\,x_3$ be the three neighbors of $x$ in $\GD$. Then,
\begin{align}\label{eq:KI}
(KI)_{x,y}&=\sum_{i=1}^3\oint_{C_{x_i,y}}K_{x,x_i} f_{x_i}(\lambda)f_y(-\lambda)\expo_{\xb_i,\yb}(\lambda)\log\lambda\frac{\ud\lambda}{(2\pi)^2},\nonumber\\
(KC)_{x,y}&=\sum_{i=1}^3 K_{x,x_i}C_{x_i,y}.
\end{align}
In Proposition \ref{prop:KI}, we handle the part $(KI)_{x,y}$: we prove that as soon as $x$ and $y$ do not belong to the same triangle of a decoration, $\bigcap_{i=1}^3 s_{x_i,y}\neq\emptyset$, so that by the general argument of Section \ref{subsec72}, $(KI)_{x,y}=0$; when $x$ and $y$ belong to the same triangle of a decoration, we explicitly compute $(KI)_{x,y}$. Then in Lemma \ref{lem:C}, we handle the part $(KC)_{x,y}$: we show that the constants $C_{x,y}$ are defined so that $(KK^{-1})_{x,y}=\delta_{x,y}$. The proof of these results is always complicated by the fact that $x$ can be of three types, `$w$',`$z$' or `$v$', {\em i.e} $x=w_k(\xb)$, $z_k(\xb)$ or $v_k(\xb)$ for some $k\in\{1,\cdots,d(\xb)\}$. The next proposition gives relations between these three cases for $KI$. Note that whenever no confusion occurs, we omit the argument $\xb$.

\begin{prop}\label{prop:casered}
For every vertex $y$ of $\GD$, the quantities $(KI)_{w_k,y}$, $(KI)_{z_k,y}$ and $(KI)_{v_k,y}$ satisfy the following:
\begin{align*}
&1.\; (KI)_{w_k,y}=-(KI)_{z_k,y}=\\
&\Bigl[\Bigl(-\oint_{\C_{z_k,y}}+\oint_{\C_{v_k,y}}\Bigr)f_{z_k}(\lambda)+\Bigl(-\oint_{\C_{w_k,y}}+\oint_{\C_{v_k,y}}\Bigr)f_{w_k}(\lambda)\Bigr]
f_y(-\lambda)\expo_{\xb,\yb}(\lambda)\log\lambda \frac{\ud\lambda}{(2\pi)^2}.\\
&2.\; (KI)_{v_k,y}=\\
&\Bigl[\Bigl(\oint_{\C_{z_k,y}}-\oint_{\C_{v_k,y}}\Bigr)f_{z_k}(\lambda)+\Bigl(-\oint_{\C_{w_k,y}}+\oint_{\C_{v_k,y}}\Bigr)f_{w_k}(\lambda)\Bigr]
f_y(-\lambda)\expo_{\xb,\yb}(\lambda)\log\lambda \frac{\ud\lambda}{(2\pi)^2}.\\
\end{align*}
$3.$ If $s_{w_k,y}\cap s_{v_k,y}\neq \emptyset$, then $(KI)_{w_k,y}=-(KI)_{z_k,y}=-(KI)_{v_k,y}=$
\begin{equation*}
\quad\quad\quad\quad\quad=\Bigl[\Bigl(-\oint_{\C_{z_k,y}}+\oint_{\C_{v_k,y}}\Bigr)f_{z_k}(\lambda)\Bigr]
f_y(-\lambda)\expo_{\xb,\yb}(\lambda)\log\lambda \frac{\ud\lambda}{(2\pi)^2}.
\end{equation*}
$4.$ If $s_{z_k,y}\cap s_{v_k,y}\neq \emptyset$, then $(KI)_{w_k,y}=-(KI)_{z_k,y}=(KI)_{v_k,y}=$
\begin{equation*}
\quad\quad\quad\quad\quad=\Bigl[\Bigl(-\oint_{\C_{w_k,y}}+\oint_{\C_{v_k,y}}\Bigr)f_{w_k}(\lambda)\Bigr]
f_y(-\lambda)\expo_{\xb,\yb}(\lambda)\log\lambda \frac{\ud\lambda}{(2\pi)^2}.
\end{equation*}
$5.$ If $s_{z_k,y}\cap s_{w_k,y}\cap s_{v_k,y}\neq \emptyset$, then
$
(KI)_{w_k,y}=(KI)_{z_k,y}=(KI)_{v_k,y}=0.
$
\end{prop}
\begin{proof}
Points $3$, $4$, $5$ are a consequence of $1$ and $2$, and the general argument of Section \ref{subsec72}. Let us prove $1$ and $2$. In equations \eqref{eq:Kw}, \eqref{eq:Kz}, \eqref{eq:Kv0}, \eqref{eq:Kv} of the proof of Proposition \ref{prop:kernelK}, we explicitly computed
\begin{equation*}
\sum_{i=1}^3 K_{x,x_i}f_{x_i}(\lambda)\expo_{\xb_i,\yb}(\lambda),
\end{equation*}
for $x=w_k(\xb)$, $z_k(\xb)$ and $v_k(\xb)$, respectively. Using this, Equation \eqref{eq:KI}, and the fact that $f_{v_k}(\lambda)=f_{w_k}(\lambda)+f_{z_k}(\lambda)$, we obtain the following.

\underline{\em If $x=w_k(\xb)$}, then the three neighbors of $x$ are $x_1=z_k(\xb)$, $x_2=z_{k+1}(\xb)$, $x_3=v_k(\xb)$, and by \eqref{eq:Kw} we have:
\begin{align}\label{eq:KIw}
&(KI)_{w_k,y}=\\
&\Bigl(-\oint_{\C_{z_k,y}} f_{z_k}(\lambda)-\oint_{\C_{z_{k+1},y}} f_{w_k}(\lambda)+\oint_{\C_{v_k,y}} [f_{w_k}(\lambda)+f_{z_k}(\lambda)] \Bigr)
f_y(-\lambda)\expo_{\xb,\yb}(\lambda)\log\lambda \frac{\ud\lambda}{(2\pi)^2}.\nonumber
\end{align}
\underline{\em If $x=z_k(\xb)$}, then the three neighbors of $x$ are $x_1=w_{k-1}(\xb),x_2=w_k(\xb),x_3=v_k(\xb)$, and by \eqref{eq:Kz} we have:
\begin{align}\label{eq:KIz}
&(KI)_{z_k,y}=\\
&\Bigl(\oint_{\C_{w_{k-1},y}} f_{z_k}(\lambda)+\oint_{\C_{w_k,y}} f_{w_k}(\lambda)-\oint_{\C_{v_k,y}} [f_{w_k}(\lambda)+f_{z_k}(\lambda)] \Bigr)
f_y(-\lambda)\expo_{\xb,\yb}(\lambda)\log\lambda \frac{\ud\lambda}{(2\pi)^2}.\nonumber
\end{align}
\underline{\em If $x=v_k(\xb)$}, then the three neighbors of $x$ are $x_1=z_k(\xb),x_2=w_k(\xb),x_3=v_\l(\xb')$, where $\l$ and $\xb'$ are such that $v_k(\xb)\sim v_\l(\xb')$ in $\GD$. Recalling that the index $\l$ refers to the decoration $\xb'$, we also omit the arguments $\xb$ and $\xb'$. Using \eqref{eq:Kv0} and \eqref{eq:Kv}, we have:
\begin{align}\label{eq:KIv}
&(KI)_{v_k,y}=\\
&\Bigl(\oint_{\C_{z_k,y}} f_{z_k}(\lambda)-\oint_{\C_{w_k,y}} f_{w_k}(\lambda)+\oint_{\C_{v_\l,y}} [f_{w_k}(\lambda)-f_{z_k}(\lambda)] \Bigr)
f_y(-\lambda)\expo_{\xb,\yb}(\lambda)\log\lambda \frac{\ud\lambda}{(2\pi)^2}.\nonumber
\end{align}
As a consequence of \eqref{eq:KIw}, \eqref{eq:KIz}, \eqref{eq:KIv},  $1$ and $2$ of Proposition \ref{prop:casered} are proved, if we show that, for every vertex $y$ of $\GD$,
\begin{equation*}
\C_{z_{k+1},y}=\C_{w_k,y},\quad\text{ and }\quad \C_{v_k,y}=\C_{v_\l,y},
\end{equation*}
which is equivalent to proving that $s_{z_{k+1},y}=s_{w_k,y}$, and $s_{v_k,y}=s_{v_\l,y}$. Recall that the construction of the sector $s_{x,y}$, given in Section \ref{subsec73},
relies on the path $\gamma_{x,y}$ encoding the poles of the integrand of $K^{-1}_{x,y}$. Recall also that, in all cases but the four exceptions to Case $2$, and Case $3$ of Section \ref{subsubsec823}, when two pairs of vertices of $\GD$ have the same path $\gamma_{x,y}$, then they have the same sector $s_{x,y}$.

\underline{\em Proof of $s_{z_{k+1},y}=s_{w_k,y}$}.
The vertices $z_{k+1}$ and $w_k$ belong to the the same decoration $\xb$, so that the exponential function has the same poles in both cases. Moreover, the functions $f_{z_{k+1}}(\lambda)$ and $f_{w_k}(\lambda)$ have the same pole $e^{i\alpha_{z_{k+1}}}=e^{i\alpha_{w_k}}$. As a consequence, by construction, the sectors $s_{z_{k+1},y}$ and $s_{w_k,y}$ are the same in all cases but the four exceptions to Case $2$, which have to be checked `manually'. Note that Case $3$ cannot occur when $x$ is of type `$w$' or `$z$'. Referring to Figure \ref{fig6}, we see that by definition:
\begin{equation*}
s_{z_{k+1},z_{k+1}}= s_{w_k,z_{k+1}},\quad \text{ and }\quad s_{z_{k+1},w_k}= s_{w_k,w_k}.
\end{equation*}

\underline{\em Proof of $s_{v_k,y}=s_{v_\l,y}$}.
The vertex $v_k$ (resp. $v_\l$) belongs to the decoration $\xb$ (resp. $\xb'$), and $k,\l,\xb,\xb'$, are such that $v_k(\xb)\sim v_\l(\xb')$ in $\GD$. Let us prove that the paths $\gamma_{v_k(\xb),y}$ and $\gamma_{v_\l(\xb'),y}$, encoding the poles of the integrand, are the same. The function $f_{v_k(\xb)}(\lambda)$ has poles $e^{i\alpha_{w_k(\xb)}}$, $e^{i\alpha_{z_k(\xb)}}$, and the function $f_{v_\l(\xb')}(\lambda)$ has opposite poles $e^{i\alpha_{w_\l(\xb')}}=-e^{i\alpha_{w_k(\xb)}}$, $e^{i\alpha_{z_\l(\xb')}}=-e^{i\alpha_{z_k(\xb)}}$. As a consequence the vertices $v_k(\xb)$ and $v_{\l}(\xb')$ define the same two train-tracks denoted by $T_x=\{T_x^1,T_x^2\}$. Let $T_y$ be the one/two train-track(s) associated to the vertex $y$. Since the construction of the path $\gamma_{x,y}$ is split according to the intersection properties of $T_x$ and $T_y$, we use the same decomposition here. 

\underline{\em If $T_x$ and $T_y$ have $0$ train-track in common}. Then the paths $\gamma_{v_k(\xb),y}$ and $\gamma_{v_\l(\xb'),y}$ are minimal, and are constructed according to Case $1$ of Section \ref{subsubsec823}. Recall that in Case~$1$, the construction of the angular sector $s_{x,y}$ from the minimal path $\gamma_{x,y}$ is independent of the choice of minimal path of $\GR$ from $\hat{\yb}$ (the initial vertex) to $\hat{\xb}$ (the final vertex). Moreover, by construction, the initial vertex of $\gamma_{v_k(\xb),y}$ and of $\gamma_{v_\l(\xb'),y}$ is the same, let us denote it by $\hat{\yb}$. So we are left with proving that the ending vertex of both paths is the same.
\begin{itemize}
\item {\em If $T_x^1, \,T_x^2$ do not separate $\yb$ from $\xb$}. Refer to Figure \ref{proof3}. Then, $\expo_{\xb,\yb}(\lambda)$ does not cancel the poles of $f_{v_k(\xb)}(\lambda)$, and the path $\gamma_{v_k(\xb),y}$ is obtained by concatenating a minimal path from $\hat{\yb}$ to $\xb$, and the vectors $e^{i\alpha_{w_k(\xb)}}$, $e^{i\alpha_{z_k(\xb)}}$. The ending vertex is $\widehat{v_k(\xb)}=\xb+e^{i\alpha_{w_k(\xb)}}+e^{i\alpha_{z_k(\xb)}}=\xb'$. When the train-tracks $T_x^1$ and $T_x^2$ do not separate $\yb$ from $\xb$, then they separate $\yb$ from $\xb'$. As a consequence, $\expo_{\xb',\yb}(\lambda)$ cancels both poles of $f_{v_\l(\xb')}(\lambda)$, and the path $\gamma_{v_\l(\xb'),y}$ is a minimal path from $\hat{\yb}$ to $\xb'$.

\begin{figure}[ht]
\begin{center}
\includegraphics[width=10.5cm]{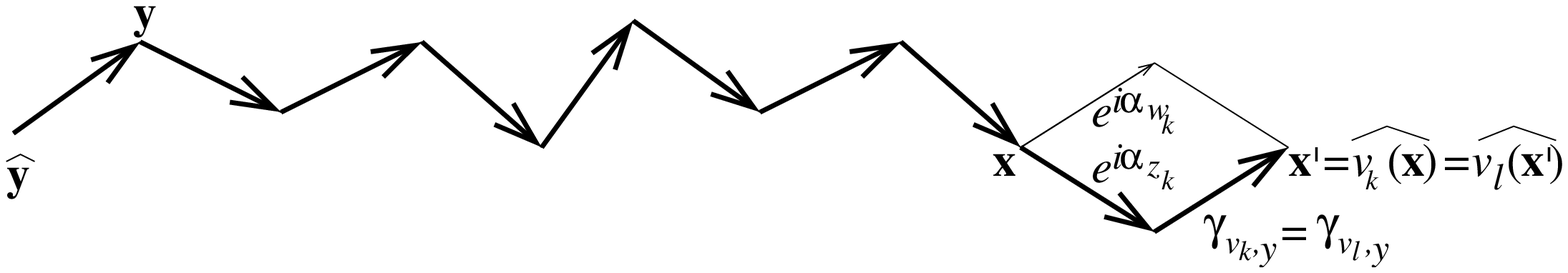}
\end{center}
\caption{\label{proof3}$T_x$ and $T_y$ have $0$ train-track in common. Construction of the paths $\gamma_{v_k(\xb),y}$ and $\gamma_{v_\l(\xb'),y}$, when $T_x^1, \,T_x^2$ do not separate $\yb$ from $\xb$.}
\end{figure}

\item{\em If only $T_x^1$ separates $\yb$ from $\xb$}. Refer to Figure \ref{proof4}. Then, $\expo_{\xb,\yb}(\lambda)$ cancels the pole $e^{i\alpha_{w_k(\xb)}}$ of
$f_{v_k(\xb)}(\lambda)$, and the pole $e^{i\alpha_{z_k(\xb)}}$ remains. The path $\gamma_{v_k(\xb),y}$ is obtained by concatenating a minimal path from $\hat{\yb}$ to $\xb$ and the vector $e^{i\alpha_{z_k(\xb)}}$. The ending vertex is $\widehat{v_k(\xb)}=\xb+e^{i\alpha_{z_k(\xb)}}$. When only $T_x^1$ separates $\yb$ from $\xb$, then only $T_x^2$ separates $\yb$ from $\xb'$. As a consequence, $\expo_{\xb',\yb}(\lambda)$ cancels the pole $-e^{i\alpha_{z_k(\xb)}}$ of $f_{v_\l(\xb')}(\lambda)$, and the pole $-e^{i\alpha_{w_k(\xb)}}$ remains. The path $\gamma_{v_\l(\xb'),y}$ is obtained by concatenating a minimal path from $\hat{\yb}$ to $\xb'$, and the vector $-e^{i\alpha_{w_k(\xb)}}$. The ending vertex is, $\widehat{v_\l(\xb')}=\xb'-e^{i\alpha_{w_k(\xb)}}=\xb+e^{i\alpha_{z_k(\xb)}}=\widehat{v_k(\xb)}$.

\begin{figure}[ht]
\begin{center}
\includegraphics[width=10.5cm]{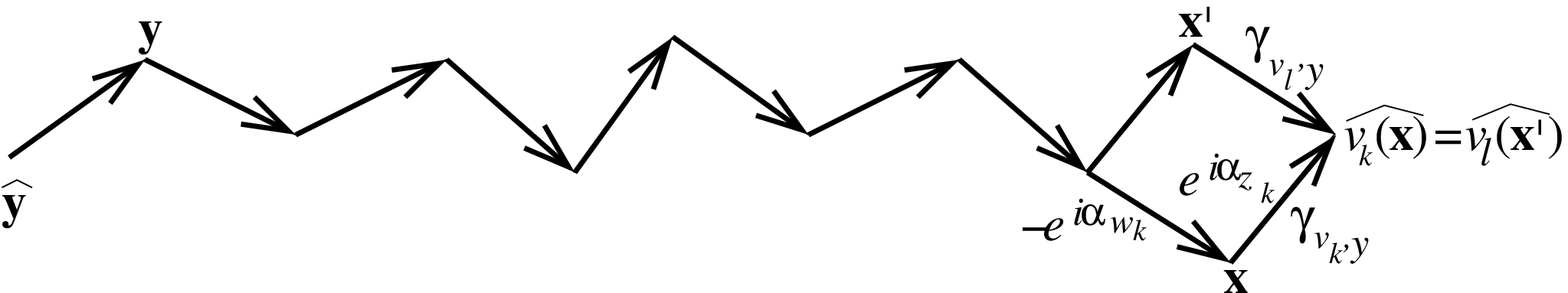}
\end{center}
\caption{\label{proof4} $T_x$ and $T_y$ have $0$ train-track in common. Construction of the paths $\gamma_{v_k(\xb),y}$ and $\gamma_{v_\l(\xb'),y}$, when only $T_x^1$ separates $\yb$ from $\xb$.}
\end{figure}

All other cases are symmetric.
\end{itemize}

\underline{\em If $T_x$ and $T_y$ have $1$ train-track $T$ in common}. Then $\gamma_{v_k(\xb),y}$ and $\gamma_{v_\l(\xb'),y}$ are paths constructed according to Case $2$ of Section \ref{subsubsec823}. Recall that in Case $2$, the path $\gamma_{x,y}$ starts from the initial vertex $\hat{\yb}$ with a parallel direction of the common train-track $T$, then follows its boundary for a positive number of steps, and ends with the opposite direction up to the ending vertex $\hat{\xb}$. By construction, the initial vertex of $\gamma_{v_k(\xb),y}$ and of $\gamma_{v_\l(\xb'),y}$ is the same, let us denote it by $\hat{\yb}$. We are left with showing that both paths have the same ending vertex. The argument is similar to the previous case. We do not repeat it here, but only provide Figures \ref{proof5} and \ref{proof6} below, which illustrate it.

\begin{figure}[ht]
\begin{center}
\includegraphics[width=9.3cm]{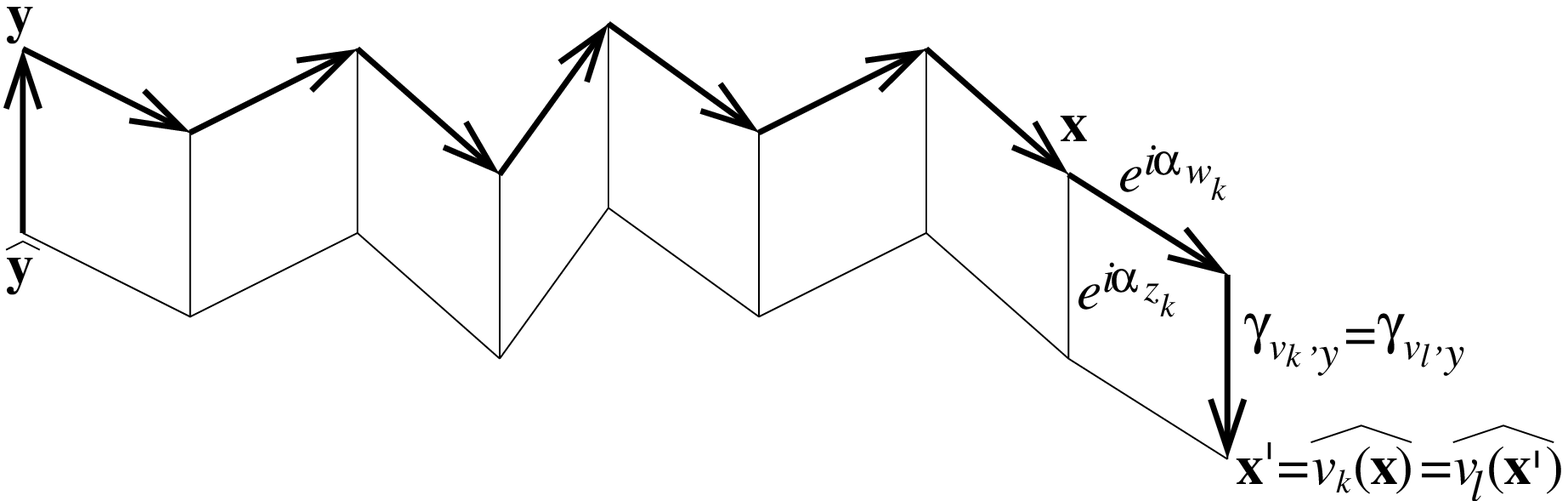}
\end{center}
\caption{\label{proof5} $T_x$ and $T_y$ have $1$ train-track in common. Construction of the paths $\gamma_{v_k(\xb),y}$ and $\gamma_{v_\l(\xb'),y}$, when $T_x^1, \,T_x^2$ do not separate $\yb$ from $\xb$.}
\end{figure}

\begin{figure}[ht]
\begin{center}
\includegraphics[width=9.3cm]{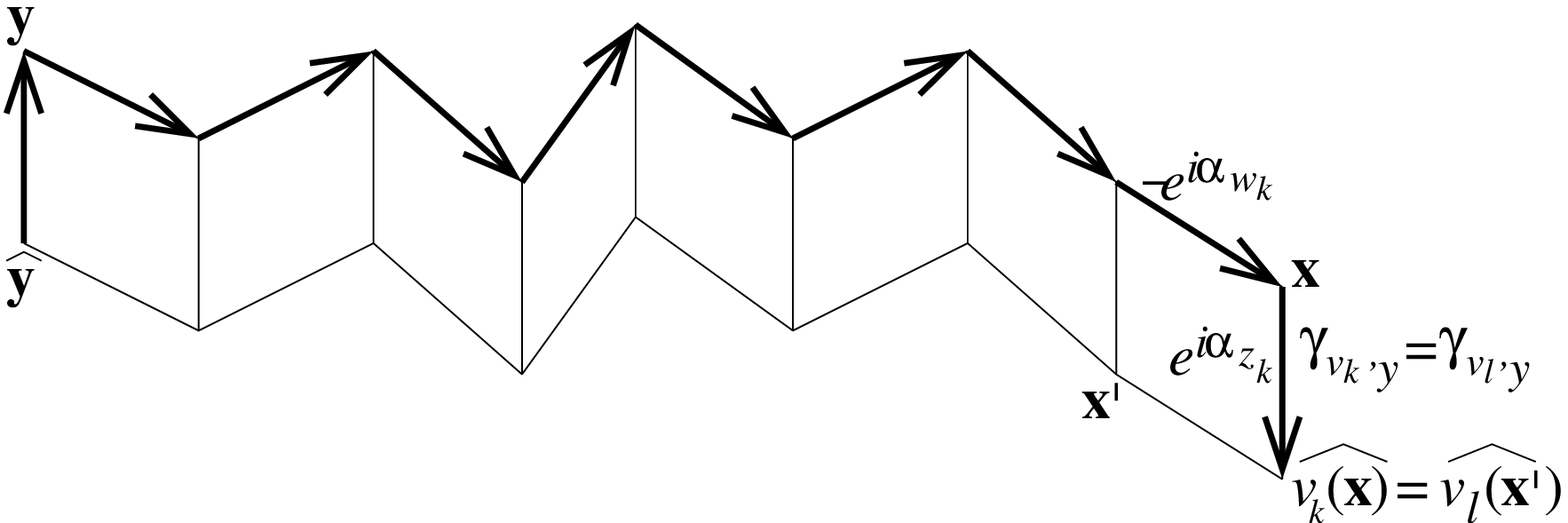}
\end{center}
\caption{\label{proof6} $T_x$ and $T_y$ have $1$ train-track in common. Construction of the paths $\gamma_{v_k(\xb),y}$ and $\gamma_{v_\l(\xb'),y}$, when only $T_x^1$ separates $\yb$ from $\xb$.}
\end{figure}
Note that the specific instances of Case $2$ do not occur, since $x$ is of type `$v$'.

\underline{\em If $T_x$ and $T_y$ have $2$ train-tracks in common}. This can only occur when $y=v_k(\xb)$ or $y=v_\l(\xb')$. So we need to check that:
\begin{equation*}
s_{v_k(\xb),v_k(\xb)}=s_{v_\l(\xb'),v_k(\xb)},\quad\text{ and }\quad s_{v_k(\xb),v_\l(\xb')}=s_{v_\l(\xb'),v_\l(\xb')}.
\end{equation*}
Referring to Figures \ref{fig8} and \ref{fig9}, we see that this is true by definition.
\end{proof}

We now state Proposition \ref{prop:KI} computing the matrix product $KI$ of Equation \eqref{eq:KI}. Whenever no confusion occurs, we drop the argument $\xb$ in $w_k(\xb)$, $z_k(\xb)$, $v_k(\xb)$.

\begin{prop}\label{prop:KI}
For all vertices $x$ and $y$ of $\GD$, we have:
\begin{equation*}
(KI)_{x,y}=
\begin{cases}
\frac{1}{2}&\text{ if
$(x,y)=(w_k,w_k),\,(z_k,z_k),\,(v_k,w_k),\,(v_k,z_k)$}\\
-\frac{1}{2}&\text{ if $(x,y)=(z_k,w_k),\,(w_k,z_k)$}\\
1&\text{ if $(x,y)=(v_k,v_k)$}\\
0&\text{ else}.
\end{cases}
\end{equation*}
\end{prop}
\begin{proof}

Let $x$ be a vertex of $\GD$. Then $x=w_k(\xb)$, $z_k(\xb)$ or $v_k(\xb)$ for some $k~\in~\{1,\cdots,d(\xb)\}$. Let us show that whenever $y\notin\{w_k(\xb),z_k(\xb),v_k(\xb)\}$, then the intersection of the three angular sectors $s_{w_k(\xb),y}\cap s_{z_k(\xb),y}\cap s_{v_k(\xb),y}$ is non empty. Using $5$ Proposition \ref{prop:casered}, this implies that $(KI)_{x,y}=0$.

Recall that the construction of the angular sector $s_{x,y}$, given in Section \ref{subsubsec823}, relies on the path $\gamma_{x,y}$ encoding the poles of the integrand of $K^{-1}_{x,y}$. The three vertices $w_k(\xb)$, $z_k(\xb)$, $v_k(\xb)$ all belong to the same decoration $\xb$, so that the exponential function is the same in all three cases. The function $f_{w_k(\xb)}(\lambda)$ has pole $e^{i\alpha_{w_k(\xb)}}$, the function $f_{z_k(\xb)}(\lambda)$ has pole $e^{i\alpha_{z_k(\xb)}}$, and the function $f_{v_k(\xb)}(\lambda)$ has poles $e^{i\alpha_{w_k(\xb)}}$, $e^{i\alpha_{z_k(\xb)}}$. As a consequence the vertices $w_k(\xb)$, $z_k(\xb)$, $v_k(\xb)$ define the same two train-tracks denoted by $T_x=\{T_x^1,T_x^2\}$. Let $T_y$ be the one/two train-track(s) associated to the vertex $y$. Since the construction of the path $\gamma_{x,y}$ is split according to the intersection properties of $T_x$ and $T_y$, we use the same decomposition here.

\underline{\em If $T_x$ and $T_y$ have $0$ train-track in common}. 
Then the paths $\gamma_{w_k,y}$, $\gamma_{z_k,y}$, $\gamma_{v_k,y}$ are minimal, and are constructed according to Case $1$ of Section \ref{subsubsec823}. Recall that in Case $1$, the construction of the angular sector $s_{x,y}$ from the minimal path is independent of the choice of minimal path from $\hat{\yb}$ (the initial vertex) to $\hat{\xb}$ (the ending vertex). Moreover, returning to the construction of the convex polygon $P$ from a minimal path $\gamma$, we have that if $\gamma_1$ and $\gamma_2$ are two minimal paths such that $\gamma_1\subset\gamma_2$ (meaning that edges of $\gamma_1$ are also edges of $\gamma_2$), then the boundary of the convex polygon $P_1$ is included in the boundary of the convex polygon $P_2$. As a consequence, the angular sector $s_2$ corresponding to $P_2$ is included in the angular sector $s_1$ corresponding to $P_1$.
\begin{itemize}
\item {\em If $T_x^1,\,T_x^2$ do not separate $\yb$ from $\xb$}. Then the minimal paths $\gamma_{w_k,y}$, $\gamma_{z_k,y}$, $\gamma_{v_k,y}$ can be chosen as in Figure \ref{proof1} below.

\begin{figure}[ht]
\begin{center}
\includegraphics[width=9.5cm]{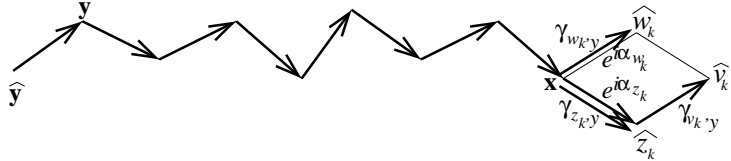}
\end{center}
\caption{\label{proof1} $T_x$ and $T_y$ have $0$ train-track in common. Construction of the paths $\gamma_{w_k,y}$, $\gamma_{z_k,y}$, $\gamma_{v_k,y}$, when $T_x^1$ and $T_x^2$  do not separate $\yb$ from $\xb$.}
\end{figure}

Observe that $\gamma_{z_k,y}\subset \gamma_{v_k,y}$, so that by the above remark $s_{v_k,y}\subset s_{z_k,y}$. Since the construction of the angular sector $s_{v_k,y}$ is independent of the choice of minimal from $\hat{\yb}$ to $\widehat{v_k}$, we can also choose $\gamma_{v_k,y}$ to be the concatenation of the minimal path from $\hat{\yb}$ to $\xb$, and  the vectors $e^{i\alpha_{w_k}}$, $e^{i\alpha_{z_k}}$. Then, $\gamma_{w_k,y}\subset\gamma_{v_k,y}$, and $s_{v_k,y}\subset s_{w_k,y}$. We deduce that: 
\begin{equation*}
s_{z_k,y}\cap s_{w_k,y}\cap s_{v_k,y}=s_{v_k,y}\neq\emptyset, 
\end{equation*}
since $s_{v_k,y}$ is an angular sector of size at least $\pi$.
\item {\em If only $T_x^1$ separates $\yb$ from $\xb$}. Then the minimal paths $\gamma_{w_k,y}$, $\gamma_{z_k,y}$, $\gamma_{v_k,y}$ can be chosen as in Figure \ref{proof1a} below.

\begin{figure}[ht]
\begin{center}
\includegraphics[width=9.8cm]{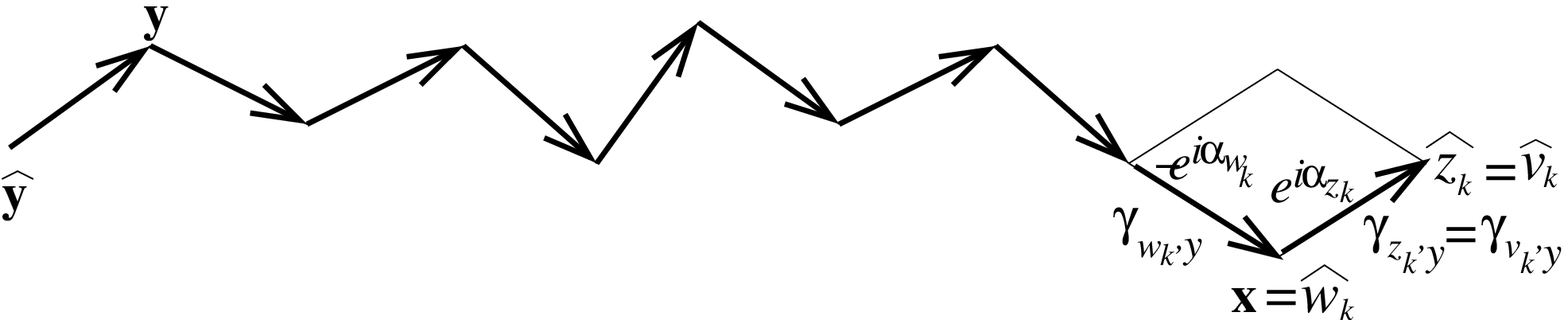}
\end{center}
\caption{\label{proof1a}  $T_x$ and $T_y$ have $0$ train-track in common. Construction of the paths $\gamma_{w_k,y}$, $\gamma_{z_k,y}$, $\gamma_{v_k,y}$, when only $T_x^1$ separates $\yb$ from $\xb$.}
\end{figure}
We have $\gamma_{w_k,y}\subset\gamma_{z_k,y}=\gamma_{v_k,y}$, so that $s_{w_k,y}\cap s_{z_k,y}\cap s_{v_k,y}=s_{v_k,y}\neq\emptyset$, since $s_{v_k,y}$ is an angular sector of size at least $\pi$.

\item{\em If $T_x^1$ and $T_x^2$ separate $\yb$ from $\xb$}. Then the minimal paths $\gamma_{w_k,y}$, $\gamma_{z_k,y}$, $\gamma_{v_k,y}$ can be chosen as in Figure \ref{proof1b} below.

\begin{figure}[ht]
\begin{center}
\includegraphics[width=10cm]{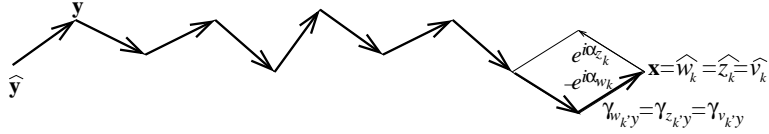}
\end{center}
\caption{\label{proof1b}  $T_x$ and $T_y$ have $0$ train-track in common. Construction of the paths $\gamma_{w_k,y}$, $\gamma_{z_k,y}$, $\gamma_{v_k,y}$, when $T_x^1$ and $T_x^2$ separate $\yb$ from $\xb$.}
\end{figure}
Then, the three paths are equal, so that $s_{w_k,y}\cap s_{z_k,y}\cap s_{v_k,y}=s_{w_k,y}\neq\emptyset$.
\end{itemize}

\underline{\em If $T_x$ and $T_y$ have $1$ train-track $T$ in common}. Let us assume that $T=T_x^2=T_y^2$. Then, the paths $\gamma_{z_k,y}$ and $\gamma_{v_k,y}$ are constructed according to Case $2$ of Section \ref{subsubsec823}, and $\gamma_{w_k,y}$ is constructed according to Case $1$. 

Suppose that $y\neq w_k,z_{k+1},z_k,w_{k-1}$, so that we are not in the exceptions to Case $2$. We only handle the case where $T_x^1$ and $T_x^2$ do not separate $\yb$ from $\xb$, see Figure \ref{proof2}, the other cases are handled in a similar way. Then, the paths $\gamma_{w_k,y}$, 
$\gamma_{z_k,y}$, $\gamma_{v_k,y}$ are constructed as in Figure \ref{proof2} below.

\begin{figure}[ht]
\begin{center}
\includegraphics[width=9cm]{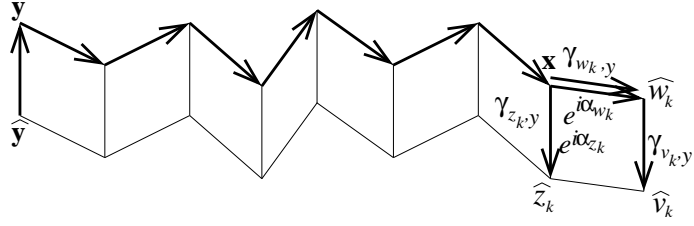}
\end{center}
\caption{\label{proof2} $T_x$ and $T_y$ have $1$ train-track in common. Construction of the paths $\gamma_{w_k,y}$, $\gamma_{z_k,y}$, $\gamma_{v_k,y}$, when $T_x^1$ and $T_x^2$ do not separate $\yb$ from $\xb$.}
\end{figure}

By construction of the convex polygon associated to these paths, we have: $s_{v_k,y}=s_{z_k,y}$, and $s_{v_k,y}\subset s_{w_k,y}$. We deduce that:
\begin{equation*}
s_{v_k,y}\cap s_{z_k,y}\cap s_{w_k,y}=s_{v_k,y}\neq\emptyset,
\end{equation*}
since $s_{v_k,y}$ is an angular sector of size $\pi-\eps$, for some $\eps>0$.

Suppose now that $y=w_k,z_{k+1},z_{k}$, or $w_{k-1}$. In these four cases, the angular sectors $s_{z_k,y}$, $s_{w_k,y}$, $s_{v_k,y}$ are drawn on Figure \ref{sector1}.

\begin{figure}[ht]
\begin{center}
\includegraphics[width=6.8cm]{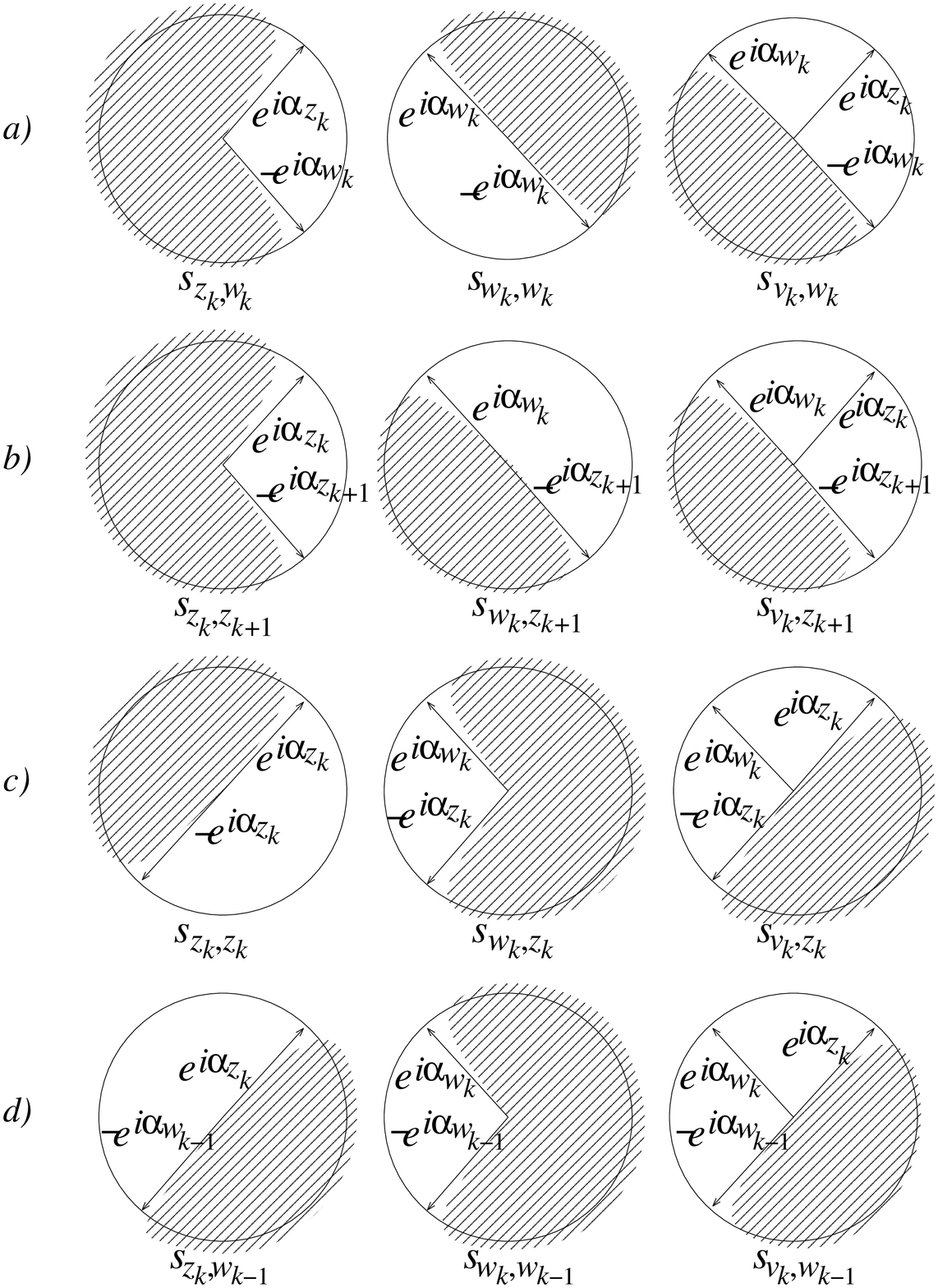}
\end{center}
\caption{\label{sector1} Angular sectors $s_{z_k,y}$, $s_{w_k,y}$, $s_{v_k,y}$ in the four exceptions to Case $2$.}
\end{figure}

From $b)$ and $d)$ we see that when $y=z_{k+1}$ or $y=w_{k-1}$,
the intersection $s_{z_k,y}\cap s_{w_k,y}\cap s_{v_k,y}\neq\emptyset$. In the remaining two cases we do explicit computations.
\begin{itemize}
\item {\em Computations for $y=w_k$}. In this case, see Figure \ref{sector1} $a)$, we have $s_{z_k,w_k}\cap~s_{v_k,w_k}~\neq~\emptyset$. Hence, by $4$ of Proposition \ref{prop:casered}, we know that:
\begin{align*}
(KI)_{w_k,w_k}=-(KI)_{z_k,w_k}&=(KI)_{v_k,w_k}=\\
&=\Bigl[\Bigl(-\oint_{\C_{w_k,w_k}}+\oint_{\C_{v_k,w_k}}\Bigr)f_{w_k}(\lambda)\Bigr]
f_{w_k}(-\lambda)\log\lambda \frac{\ud\lambda}{(2\pi)^2}.
\end{align*}
Using the definition of $f_{w_k}(\lambda)$, and denoting by $\C$ a generic simple closed curve oriented counterclockwise, containing all poles of the integrand, and avoiding a half-line starting from $0$, yields:
\begin{align}\label{eq:intww}
\frac{1}{(2\pi)^2}
\oint_{\C}f_{w_k}(\lambda)f_{w_k}(-\lambda)\log\lambda\ud\lambda&=-\frac{1}{(2\pi)^2}
 \oint_{\C}\frac{e^{i\alpha_{w_k}}}{(\lambda-e^{i\alpha_{w_k}})(\lambda+e^{i\alpha_{w_k}})}\log\lambda\ud\lambda\nonumber\\
&=-\frac{i}{4\pi}\left[\log_\C(e^{i\alpha_{w_k}})-\log_\C(-e^{i\alpha_{w_k}})\right].
\end{align}
As a consequence, see Figure \ref{sector1} $a)$
\begin{align*}
&(KI)_{w_k,w_k}=-(KI)_{z_k,w_k}=(KI)_{v_k,w_k}=\\
&=-\frac{i}{4\pi}\left[
-\log_{\C_{w_k,w_k}}(e^{i\alpha_{w_k}})+\log_{\C_{w_k,w_k}}(-e^{i\alpha_{w_k}})+\log_{\C_{v_k,w_k}}(e^{i\alpha_{w_k}})-\log_{\C_{v_k,w_k}}(-e^{i\alpha_{w_k}})
\right]\\
&=-\frac{i}{4\pi}[-i\alpha_{w_k}+i(\alpha_{w_k}+\pi)+i\alpha_{w_k}-i(\alpha_{w_k}-\pi)]=\frac{1}{2}.
\end{align*}

\item {\em Computations for $y=z_k$}. In this case, see Figure \ref{sector1} $c)$, we have $s_{w_k,z_k}\cap s_{v_k,z_k}\neq~\emptyset$. Hence, by $3$ of Proposition \ref{prop:casered}, we know that:
\begin{align*}
(KI)_{w_k,z_k}=-(KI)_{z_k,z_k}&=-(KI)_{v_k,z_k}=\\
&=\Bigl[\Bigl(-\oint_{\C_{z_k,z_k}}+\oint_{\C_{v_k,z_k}}\Bigr)f_{z_k}(\lambda)\Bigr]
f_{z_k}(-\lambda)\log\lambda \frac{\ud\lambda}{(2\pi)^2}.\\
\end{align*}
Using the definition of $f_{z_k}(\lambda)$, yields
$f_{z_k}(\lambda)f_{z_k}(-\lambda)=-\frac{e^{i\alpha_{z_k}}}{(\lambda-e^{i\alpha_{z_k}})(\lambda+e^{i\alpha_{z_k}})}$,
so that using Figure \ref{sector1} $c)$, we obtain:
\begin{align*}
&(KI)_{w_k,z_k}=-(KI)_{z_k,z_k}=-(KI)_{v_k,z_k}=\\
&=-\frac{i}{4\pi}\left[
-\log_{\C_{z_k,z_k}}(e^{i\alpha_{z_k}})+\log_{\C_{z_k,z_k}}(-e^{i\alpha_{z_k}})+\log_{\C_{v_k,z_k}}(e^{i\alpha_{z_k}})-\log_{\C_{v_k,z_k}}(-e^{i\alpha_{z_k}})
\right]\\
&=-\frac{i}{4\pi}[-i\alpha_{z_k}+i(\alpha_{z_k}-\pi)+i\alpha_{z_k}-i(\alpha_{z_k}+\pi)]=-\frac{1}{2}.
\end{align*}
\end{itemize}

\underline{\em If $T_x$ and $T_y$ have $2$ train-tracks in common}. This can only occur when 
$y=v_k(\xb)$ or $v_\l(\xb')$, such that $v_k(\xb)\sim v_\l(\xb')$. In these two cases, the angular sectors $s_{z_k,y}$, $s_{w_k,y}$, $s_{v_k,y}$ are drawn on Figure \ref{fig2}.\vspace{0.1cm}\\

\begin{figure}[h]
\begin{center}
\includegraphics[width=7.5cm]{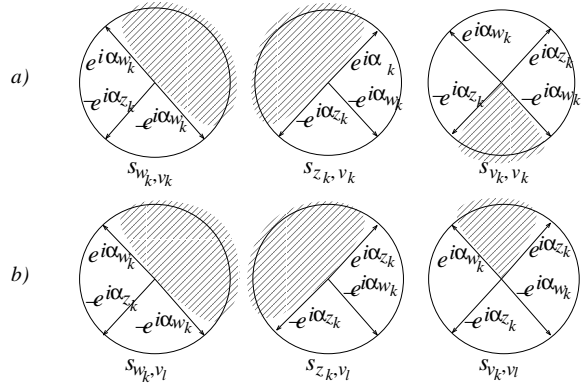}
\end{center}
\caption{Angular sectors $s_{z_k,y}$, $s_{w_k,y}$, $s_{v_k,y}$ in Case $3$.\label{fig2}} 
\end{figure}

From $b)$ we see that $(KI)_{w_k,v_\l}=(KI)_{z_k,v_\l}=(KI)_{v_k,v_\l}=0$. For $y=v_k(\xb)$, we do explicit computations.
\begin{itemize}
\item {\em Computations for $y=v_k$}. We use $1$ and $2$ of Proposition \ref{prop:casered}. Let us first compute:
\begin{equation*}
\Bigl(-\oint_{\C_{z_k,v_k}}+\oint_{\C_{v_k,v_k}}\Bigr)f_{z_k}(\lambda)f_{v_k}(-\lambda)\log\lambda \frac{\ud\lambda}{(2\pi)^2}.
\end{equation*}
Using the definition of $f_{z_k}(\lambda)$, $f_{v_k}(\lambda)$, and denoting by $\C$ a generic contour, we have:
\begin{align}\label{eq:calcul}
\frac{1}{(2\pi)^2}&\oint_{\C} f_{z_k}(\lambda) f_{v_k}(-\lambda)
 \log\lambda\ud\lambda\nonumber\\
&=-\frac{1}{4\pi^2}\oint_{\C}\left(\frac{e^{i\alpha_{z_k}}}{(\lambda-e^{i\alpha_{z_k}})(\lambda+e^{i\alpha_{z_k}})}
-\frac{e^{i\frac{\alpha_{z_k}+\alpha_{w_k}}{2}}}{(\lambda-e^{i\alpha_{z_k}})(\lambda+e^{i\alpha_{w_k}})}
\right)\log\lambda\,\ud\lambda\nonumber\\
&=-\frac{i}{4\pi}\left[
\log_\C(e^{i\alpha_{z_k}})-\log_\C(-e^{i\alpha_{z_k}})-\frac{\log_\C(e^{i\alpha_{z_k}})-\log_\C(-e^{i\alpha_{w_k}})}{\cos\left(\frac{\alpha_{w_k}-\alpha_{z_k}}{2}\right)}
\right].
\end{align}
As a consequence:
\begin{align*}
\left(-\oint_{\C_{z_k,v_k}}+\oint_{\C_{v_k,v_k}}\right)&f_{z_k}(\lambda)f_{v_k}(-\lambda)\log\lambda \frac{\ud\lambda}{(2\pi)^2}=\\
&=\frac{1}{4\pi}\left[-\pi+\frac{\pi-(\alpha_{w_k}-\alpha_{z_k})}{\cos\left(\frac{\alpha_{w_k}-\alpha_{z_k}}{2}\right)}-\pi-\frac{\pi-(\alpha_{w_k}-\alpha_{z_k})}{\cos\left(\frac{\alpha_{w_k}-\alpha_{z_k}}{2}\right)}\right]\\
&=-\frac{1}{2}.
\end{align*}

Let us now compute,
\begin{equation*}
\left(-\oint_{\\C_{w_k,v_k}}+\oint_{\\C_{v_k,v_k}}\right)f_{w_k}(\lambda)
f_{v_k}(-\lambda)\log\lambda \frac{\ud\lambda}{(2\pi)^2}.\\
\end{equation*}
Exchanging $z_k$ with $w_k$ in \eqref{eq:calcul} yields:
\begin{align}\label{eq:fwv}
\frac{1}{(2\pi)^2}&\oint_{\C} f_{w_k}(\lambda) f_{v_k}(-\lambda)
 \log\lambda\ud\lambda\nonumber\\
&=-\frac{i}{4\pi}\left[
\log_\C(e^{i\alpha_{w_k}})-\log_\C(-e^{i\alpha_{w_k}})-\frac{\log_\C(e^{i\alpha_{w_k}})-\log_\C(-e^{i\alpha_{z_k}})}{\cos\left(\frac{\alpha_{w_k}-\alpha_{z_k}}{2}\right)}
\right].\nonumber\\
\end{align}
As a consequence:
\begin{align*}
\left(-\oint_{\C_{w_k,v_k}}+\oint_{\C_{v_k,v_k}}\right)&f_{w_k}(\lambda)
f_{v_k}(-\lambda)\log\lambda \frac{\ud\lambda}{(2\pi)^2}\\
&=\frac{1}{4\pi}
\left[\pi-\frac{\pi-(\alpha_{w_k}-\alpha_{z_k})}{\cos\left(\frac{\alpha_{w_k}-\alpha_{z_k}}{2}\right)}+\pi+\frac{\pi-(\alpha_{w_k}-\alpha_{z_k})}{\cos\left(\frac{\alpha_{w_k}-\alpha_{z_k}}{2}\right)}
\right]\\
&=\frac{1}{2}.
\end{align*}
Thus, by $1$ and $2$ of Proposition \ref{prop:casered}, we deduce:
\begin{equation*}
(KI)_{w_k,v_k}=0,\quad (KI)_{z_k,v_k}=0,\quad (KI)_{v_k,v_k}=1.
\end{equation*}
\end{itemize}
\end{proof}

By Proposition \ref{prop:KI}, proving that $KK^{-1}=\Id$, amounts to proving the following lemma for $KC$.

\begin{lem}\label{lem:C}
For all vertices $x$ and $y$ of $\GD$, we have:
\begin{equation*}
(KC)_{x,y}=
\begin{cases}
\frac{1}{2}&\text{ if
$(x,y)=(w_k,w_k),\,(z_k,z_k),\,(w_k,z_k),\,(z_k,w_k)$}\\
-\frac{1}{2}&\text{ if $(x,y)=(v_k,w_k),\,(v_k,z_k)$}\\
0&\text{ else}.
\end{cases}
\end{equation*}
\end{lem}
\begin{proof}
By definition of $C$, we have $C_{x,y}=0$, as soon as $x$ or $y$ is
of type $`v'$. As a consequence, using Equation \eqref{eq:KI} and our choice of Kasteleyn orientation, we deduce:
\begin{align}\label{eq:C1}
(KC)_{w_k,y}&=-C_{z_k,y}+\eps_{w_k,z_{k+1}}C_{z_{k+1},y}\nonumber\\
(KC)_{z_k,y}&=\eps_{z_k,w_{k-1}}C_{w_{k-1},y}+C_{w_k,y}\nonumber\\
(KC)_{v_k,y}&=C_{z_k,y}-C_{w_k,y}.
\end{align}
Let us first prove that for every vertex $y$ of $\GD$, we have $\eps_{w_k,z_{k+1}}C_{z_{k+1},y}=C_{w_k,y}$, or equivalently:
\begin{equation}\label{eq:C}
C_{z_{k+1},y}=\eps_{w_k,z_{k+1}}C_{w_k,y}.
\end{equation}
When $y$ belongs to a different decoration than $x$, then both sides of \eqref{eq:C} are equal to $0$. Let us thus suppose that $y$ is in the same decoration as $x$. If $y\neq z_{k+1}$, then by definition:
\begin{equation*}
C_{z_{k+1},y}=\frac{(-1)^{n(z_{k+1},y)}}{4}=\frac{1}{4}(-1)^{n(z_{k+1},w_k)}(-1)^{n(w_k,y)}
=\eps_{w_k,z_{k+1}}C_{w_k,y},
\end{equation*}
so that \eqref{eq:C} holds. If $y=z_{k+1}$, then by definition, $C_{z_{k+1},z_{k+1}}=-\frac{1}{4}$.
Moreover, 
\begin{equation*}
\eps_{w_k,z_{k+1}}C_{w_k,z_{k+1}}=\eps_{w_k,z_{k+1}}\frac{(-1)^{n(w_k,z_{k+1})}}{4}
=\eps_{w_k,z_{k+1}}\frac{\eps_{z_{k+1},w_k}}{4}=-\frac{1}{4},
\end{equation*}
so that \eqref{eq:C} also holds in this case.
As a consequence, Equation \eqref{eq:C1} becomes:
\begin{align*}
&(KC)_{w_k,y}=(KC)_{z_k,y}=-(KC)_{v_k,y}=-C_{z_k,y}+C_{w_k,y}.
\end{align*}
Let us now end the proof of Lemma \ref{lem:C}. If $y$ does not belong to the same decoration as $x$, or if $y$ belongs to the same decoration as $x$ and is of type `$v$', then $C_{w_k,y}=C_{z_k,y}=0$, so that:
\begin{equation*}
(KC)_{w_k,y}=(KC)_{z_k,y}=(KC)_{v_k,y}=0.
\end{equation*}
If $y$ belongs to the same decoration as $x$, but not to the same triangle of the decoration, then
\begin{equation*} 
C_{w_k,y}=
\frac{(-1)^{n(w_k,y)}}{4}=\frac{1}{4}(-1)^{n(w_k,z_k)}(-1)^{n(z_k,y)}=\eps_{z_k,w_k}C_{z_k,y}
=C_{z_k,y},
\end{equation*}
since by our choice of Kasteleyn orientation, $\eps_{z_k,w_k}=1$. Thus,
\begin{equation*}
(KC)_{w_k,y}=(KC)_{z_k,y}=(KC)_{v_k,y}=0.
\end{equation*}
If $y=w_k$, then by definition $C_{w_k,w_k}=\frac{1}{4}$, and $C_{z_k,w_k}=-C_{w_k,z_k}=-\frac{1}{4}$, thus:
\begin{equation*}
(KC)_{w_k,w_k}=(KC)_{z_k,w_k}=-(KC)_{v_k,w_k}=\frac{1}{2}.
\end{equation*}
Finally, if $y=z_k$, then by definition $C_{w_k,z_k}=\frac{1}{4}$, and $C_{z_k,z_k}=-\frac{1}{4}$, so that:
\begin{equation*}
(KC)_{w_k,z_k}=(KC)_{z_k,z_k}=-(KC)_{v_k,w_k}=\frac{1}{2}.
\end{equation*}

\end{proof}

\appendix
\section{Probability of occurrence of single edges}\label{app:calculs}

Let us compute the probability of occurrence of single edges in dimer configurations of the Fisher graph $\GD$ chosen with respect to the Gibbs measure $\P$ of Theorem \ref{thm:measure}. Every edge of $\GD$ is of the form $ w_k z_k$, $ w_k z_{k+1}$, $w_k v_k $ or $ v_k v_\l $ as represented on Figure~\ref{fig:proba_edges} below. The vertices $z_k$, $w_k$, $z_{k+1}$  and $v_k$ belong to the decoration of $\xb$, and $v_\l$ belongs to that of $\yb$. The edge of $G$ joining $\xb$ and $\yb$ is the diagonal of a rhombus with half-angle $\theta$. In order to simplify notations, let us write 
\begin{equation*}
 \alpha_{z_k}=\alpha,\quad\quad\alpha_{w_k}=\beta=\alpha+2\theta.
\end{equation*}

\begin{figure}[ht]
  \begin{center}
    \includegraphics[width=4cm]{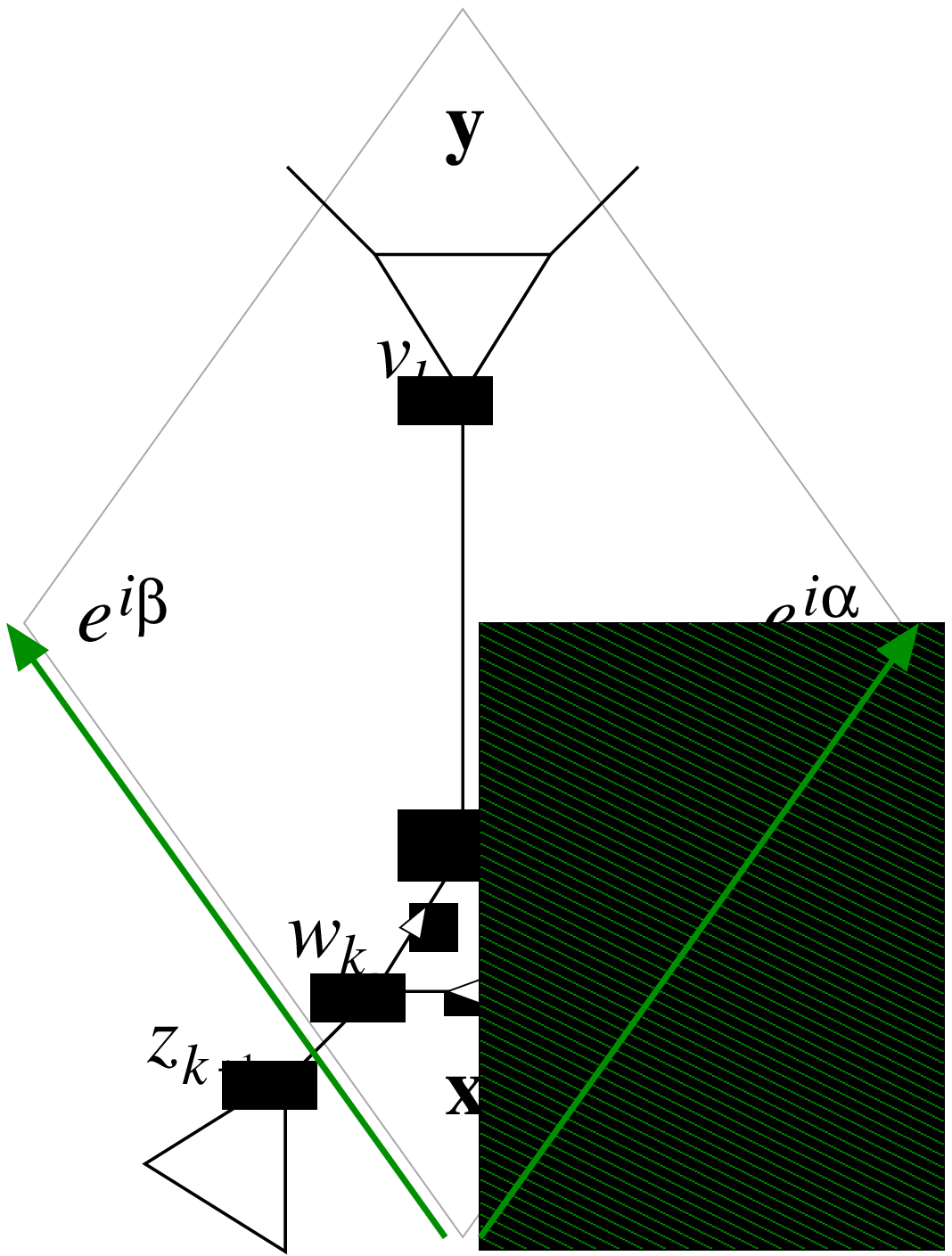}
    \caption{A piece of the Fisher graph $\GD$ near a rhombus with half-angle $\theta$, adjacent to the decorations of $\xb$ and $\yb$. The sides of the rhombus are represented by the unit vectors $e^{i\alpha}$ and $e^{i\beta}$, with $\beta-\alpha=2\theta$.}
    \label{fig:proba_edges}
  \end{center}
\end{figure}

By Theorem~\ref{thm:measure}, we know that the probability of an edge $e=uv$ of $\GD$ is given by
\begin{equation*}
\P(e)= K_{u,v}\Pf \begin{bmatrix} K^{-1}_{u,u} & K^{-1}_{v,u} \\ K^{-1}_{u,v} & K^{-1}_{v,v} \end{bmatrix} = K_{u,v} K^{-1}_{v,u},
\end{equation*}
where the coefficient $K^{-1}_{u,v}$ of the inverse Kasteleyn matrix is given by Theorem \ref{inverse}.

\subsubsection*{Probability of the edge $w_k z_k$}
The vertices $w_k$ and $z_k$ belong to the same decoration $\xb$, so that there is no contribution from the exponential function, since $\expo_{\xb,\xb}=1$. By our choice of Kasteleyn orientation, we have $K_{z_k,w_k}=1$. Using Formula \eqref{invKast} and the definition of the function $f$, yields:
\begin{align*}
\P(w_k z_k)=K_{z_k,w_k}K^{-1}_{w_k,z_k}&= \frac{1}{4\pi^2}\oint_{\mathcal{C}_1} f_{w_k}(\lambda) f_{z_k}(-\lambda) \log(\lambda) \ud\lambda + C_{w_k,z_k}\\
  &= \frac{1}{4\pi^2} \oint_{\mathcal{C}_1} -\frac{e^{\frac{i\beta}{2}}}{e^{i\beta}-\lambda}\frac{e^{\frac{i\alpha}{2}}}{e^{i\alpha}+\lambda}\log(\lambda)\ud \lambda + \frac{1}{4},
\end{align*}
where $\C_1=\C_{w_k,z_k}$ is the contour defined according to Case $1$ of Section \ref{subsubsec823}, represented on Figure~\ref{fig:proba_contour12}. The integral is evaluated by Cauchy's theorem:
\begin{align*}
\frac{1}{4\pi^2}\oint_{\mathcal{C}_1} \frac{e^{i\frac{\alpha+\beta}{2}}\log(\lambda)}{(\lambda-e^{i\beta})(\lambda+e^{i\alpha})} \ud \lambda &= \frac{i}{2\pi} e^{i\frac{\alpha+\beta}{2}}\left( \frac{\log_{\C_1}(e^{i\beta})-\log_{\C_1}(-e^{i\alpha})}{e^{i\alpha}+e^{i\beta}} \right)\\
&=\frac{i}{4\pi\cos(\frac{\beta-\alpha}{2})}\left( i\beta -i(\alpha+\pi)\right) \\ &=\frac{\pi-2\theta}{4\pi\cos\theta}.
\end{align*}
Therefore, $\displaystyle\P(w_k z_k)=\frac{1}{4}+\frac{\pi-2\theta}{4\pi\cos\theta}$.

\begin{figure}[ht]
\begin{center}
\includegraphics[height=4cm]{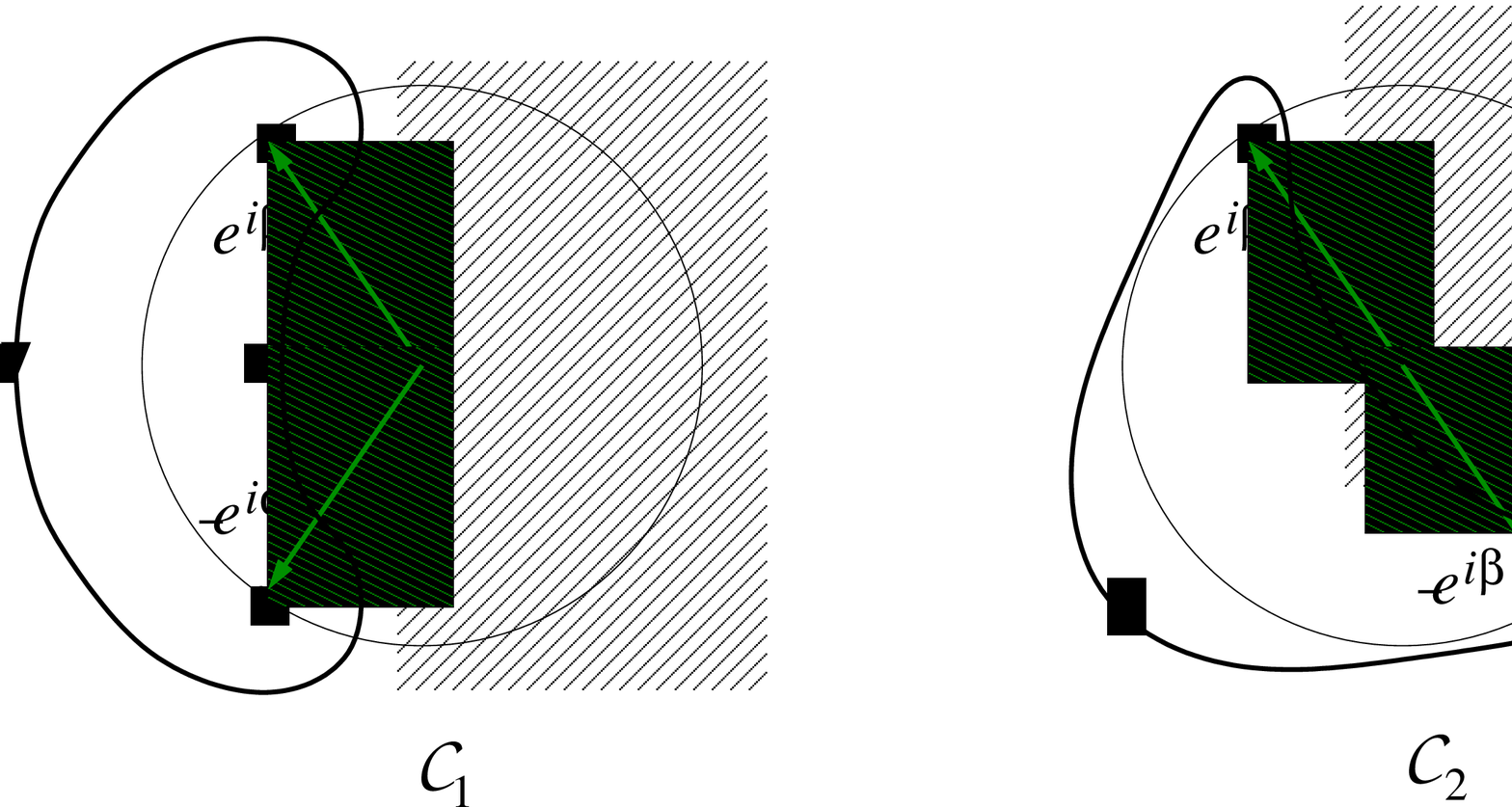}
\caption{\label{fig:proba_contour12} Left: the contour $\C_1=\C_{w_k,z_k}$ involved in the integral term of $K^{-1}_{w_k,z_{k}}$. Right: the contour $\C_2=\C_{z_{k+1},w_k}$ involved in the integral term of $K^{-1}_{z_{k+1},w_k}$. In both cases, the arrows represent the poles of the integrand, and the shaded zone is the angular sector avoided by the contour.}
\end{center}
\end{figure}

\subsubsection*{Probability of the edge $w_k z_{k+1}$}
Since the vertices $w_k$ and $z_{k+1}$ belong to the same decoration, there is no contribution from the exponential function. Moreover, 
by Equation \eqref{eq:kw01}, we know that $K_{w_k,z_{k+1}}f_{z_{k+1}}(\lambda)=-f_{w_k}(\lambda)$, and by Equation \eqref{eq:C}, we have $K_{w_k,z_{k+1}}C_{z_{k+1},w_k}=C_{w_k,w_k}=~\frac{1}{4}$. Therefore,
\begin{align*}
\P(w_k z_{k+1})&=K_{w_k,z_{k+1}}K^{-1}_{z_{k+1},w_k} \\
&=\frac{1}{4\pi^2}\oint_{\C_2} K_{w_k,z_{k+1}}f_{z_{k+1}}(\lambda) f_{w_{k}}(-\lambda) \log(\lambda)\ud \lambda + K_{w_k,z_{k+1}}C_{z_{k+1},w_k}\\
&=-\frac{1}{4\pi^2}\oint_{\C_2}f_{w_{k}}(\lambda)f_{w_{k}}(-\lambda)\log(\lambda)\ud \lambda+\frac{1}{4}.
\end{align*}
where $\C_2=\C_{z_{k+1},w_k}$ is the contour defined according to Case $2$ of Section \ref{subsubsec823}, represented on Figure \ref{fig:proba_contour12}. Using Equation \eqref{eq:intww} where we computed this integral for a generic contour $\C$, we obtain:
\begin{align*}
 -\frac{1}{4\pi^2}\oint_{\C_2}f_{w_{k}}(\lambda)f_{w_{k}}(-\lambda)\log\lambda\ud \lambda&=
\frac{i}{4\pi}[\log_{\C_2}(e^{i\beta})-\log_{\C_2}(-e^{i\beta})]=
\frac{i}{4\pi}(i\beta-i(\beta+\pi))=\frac{1}{4}.
\end{align*}
We conclude that  $\displaystyle\P(w_k z_{k+1})=\frac{1}{2}$.

\subsubsection*{Probability of the edges $w_k v_k$ an $z_k v_k$}
Again, there is no contribution from the exponential function. By our choice of Kasteleyn orientation, we have $K_{v_k,w_k}=-1$, and by definition we have $C_{v_k,w_k}=0$. Using Formula \eqref{invKast}, this yields:
\begin{align*}
 \P(w_k v_k)=K_{v_k,w_k}K^{-1}_{w_k,v_k}=
-\frac{1}{4\pi^2}\oint_{\C_3}f_{w_k}(\lambda)f_{v_k}(-\lambda)\log(\lambda) \ud\lambda,
\end{align*}
where $\C_3=\C_{w_k,v_k}$ is the contour defined according to Case $2$ of Section \ref{subsubsec823}, represented on Figure \ref{fig:proba_contour34}.
Using Equation \eqref{eq:fwv} where we computed this integral for a generic contour $\C$, we obtain:
\begin{align*}
\P(w_k v_k)&=
\frac{i}{4\pi}\left(
\log_{\C_3}(e^{i\beta})-\log_{\C_3}(-e^{i\beta})-
\frac{\log_{\C_3}(e^{i\beta})-\log_{\C_3}(-e^{i\alpha})}{\cos\left(\frac{\beta-\alpha}{2}\right)}\right)\\
&=\frac{i}{4\pi}\left(i\beta-i(\beta+\pi)-\frac{i\beta-i(\alpha+\pi)}{\cos\theta}\right)\\
&=\frac{1}{4}-\frac{\pi-2\theta}{4\pi\cos\theta}.
\end{align*}

By symmetry, this is also the probability of occurrence of the edge $z_k v_k$.

\begin{figure}[ht]
  \begin{center}
    \includegraphics[height=4cm]{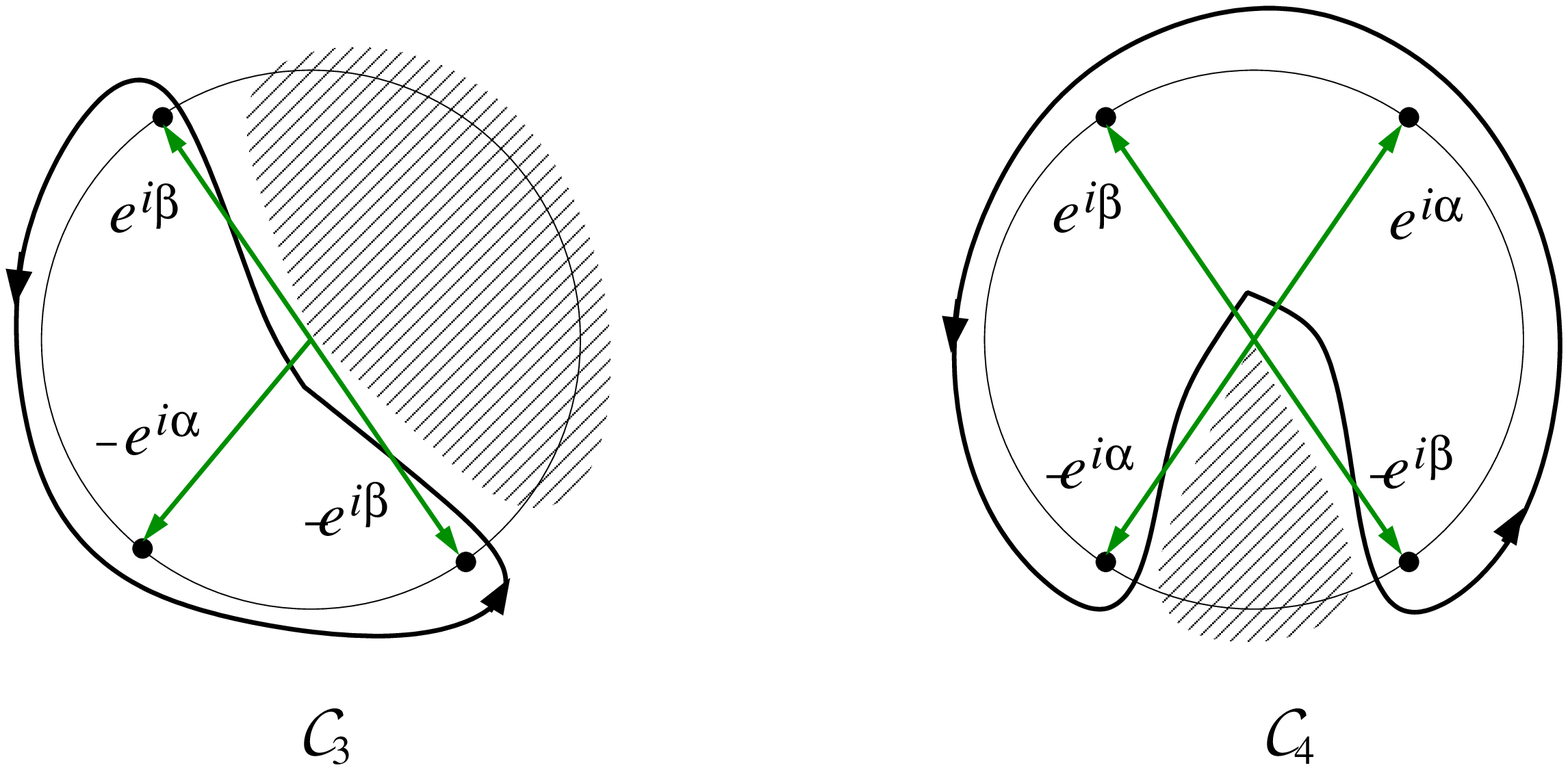}
    \caption{\label{fig:proba_contour34} 
Left: the contour $\C_3=\C_{w_k,v_k}$ involved in the integral term of $K^{-1}_{w_k,v_{k}}$. Right: the contour $\C_4=\C_{v_\l,v_k}$ involved in the integral term of $K^{-1}_{v_\l,v_k}$. 
}
\end{center}
\end{figure}

\subsubsection*{Probability of the edge $v_k v_\l$}
By definition, we have $C_{v_\l,v_k}=0$, and by Equation \eqref{eq:Kv0}, we know that:
\begin{equation*}
K_{v_k,v_\l}f_{v_\l}(\lambda)\expo_{\yb,\xb}(\lambda)=f_{w_k}(\lambda)-f_{z_k}(\lambda). 
\end{equation*}
Using Formula \eqref{invKast}, this yields:
\begin{align*}
\P(v_k v_\l)&=K_{v_k,v_\l}K^{-1}_{v_\l,v_k}\\
&= \frac{1}{4\pi^2}\oint_{\C_4} K_{v_k,v_\l} f_{v_\l}(\lambda)f_{v_k}(-\lambda)\expo_{\yb,\xb}(\lambda)\log\lambda\,\ud\lambda\\ &=\frac{1}{4\pi^2}\oint_{\C_4}\left[f_{w_k}(\lambda)-f_{z_k}(\lambda)\right]f_{v_k}(-\lambda)\log(\lambda) \ud\lambda\\
&=\frac{1}{4\pi^2}\oint_{\C_4}f_{w_k}(\lambda)f_{v_k}(-\lambda)\log(\lambda) \ud\lambda-
\frac{1}{4\pi^2}\oint_{\C_4}f_{z_k}(\lambda)f_{v_k}(-\lambda)\log(\lambda) \ud\lambda.
\end{align*}
where $\C_4=\C_{v_\l,v_k}$ is the contour defined according to Case $3$ of Section \ref{subsubsec823}, represented on Figure \ref{fig:proba_contour34}.
Using Equation \eqref{eq:fwv}, the first term is equal to:
\begin{align*}
 \frac{1}{4\pi^2}\oint_{\C_4}f_{w_k}(\lambda)&f_{v_k}(-\lambda)\log(\lambda) \ud\lambda=\\
&=-\frac{i}{4\pi}\left(
\log_{\C_4}(e^{i\beta})-\log_{\C_4}(-e^{i\beta})-
\frac{\log_{\C_4}(e^{i\beta})-\log_{\C_4}(-e^{i\alpha})}{\cos\left(\frac{\beta-\alpha}{2}\right)}\right)\\
&=-\frac{i}{4\pi}\left(i\beta-i(\beta-\pi)-\frac{i\beta-i(\alpha+\pi)}{\cos\theta}\right)\\
&=\frac{1}{4}+\frac{\pi-2\theta}{4\pi\cos\theta}.
\end{align*}
By Equation \eqref{eq:calcul} and Figure \ref{fig:proba_contour34}, the second term is equal to:
\begin{align*}
 -\frac{1}{4\pi^2}\oint_{\C_4}f_{z_k}(\lambda)&f_{v_k}(-\lambda)\log(\lambda) \ud\lambda=\\
&=\frac{i}{4\pi}\left(
\log_{\C_4}(e^{i\alpha})-\log_{\C_4}(-e^{i\alpha})-
\frac{\log_{\C_4}(e^{i\alpha})-\log_{\C_4}(-e^{i\beta})}{\cos\left(\frac{\beta-\alpha}{2}\right)}\right)\\
&=\frac{i}{4\pi}\left(i\alpha-i(\alpha+\pi)-\frac{i\alpha-i(\beta-\pi)}{\cos\theta}\right)\\
&=\frac{1}{4}+\frac{\pi-2\theta}{4\pi\cos\theta}.
\end{align*}

As a consequence $\displaystyle\P(v_k v_\l)=\frac{1}{2}+\frac{\pi-2\theta}{2\pi\cos\theta}.$

Let us make a few simple comments about the values of these probabilities.

\begin{enumerate}
  \item The value $\frac{1}{2}$ for the probability of the edge $w_k z_{k+1}$ can be explained as follows. By Fisher's correspondence, once the configuration of the edges coming from edges of $G$ attached to the decoration of $\xb$ is fixed, there are two possibilities for the dimer covering inside the decoration, which both have the same weight. There is always one of the two possibilities containing the edge $w_k z_{k+1}$. Therefore, this edge appears in a random dimer configuration half of the time.
  \item Notice that $\P(w_k z_k)= \frac{1}{2}\P(v_k v_\l)$. This is explained by the fact that the edge $w_k z_k$ appears only if $v_k v_\l$ is not present in the dimer configuration, and it appears only in one of the two allowed configurations, once the state of the edges coming from edges of $G$ is fixed.
  \item Using the two previous points, and the fact that the probability of the edges incident to a given vertex must sum to $1$, one can deduce the probability of all the edges from the probability of the edge $v_k v_\l$.
\end{enumerate}

\bibliographystyle{alpha}
\bibliography{ising}

\end{document}